\documentclass[cmp]{svjour}  
\usepackage{amsmath}
\usepackage{amsfonts,amssymb}

\newtheorem{prop}[theorem]{Proposition}

\newcommand{\bea}{\begin{eqnarray}}\newcommand{\eea}{\end{eqnarray}}
\newcommand{\beq}{\begin{equation}}\newcommand{\eq}{\end{equation}}
\newcommand{\beqs}{\begin{equation*}}\newcommand{\eqs}{\end{equation*}}
\newcommand{\les}{\lesssim}
\def\bm{\left( \begin{array}{cc}}\def\endm{\end{array}\right)}

\def\Vs{{{V\underline{}}^*}}\def\Us{{{U\underline{}}^*}}\def\us{u^{\!*}}
\def\Ts{{{T\underline{}}^*}}\def\Ss{{{S\underline{}}^*}}
\def\Ls{{{L\underline{}\!}^*}}\def\Lbs{{{\underline{L\!}}^*}}
\def\Lt{{\widetilde{L\underline{}\!}\,}}\def\Lbt{{\widetilde{\underline{L\!}\,}}}
\def\Lh{{\hat{L\underline{}}}}\def\Lbh{{\hat{\underline{L\!}\,}}}
\def\Lb{\underline{L\!}}\def\uL{\underline{L\!}}
\def\utau{\underline{\tau\!}\,}\def\uchi{\underline{\chi}}

\def\a{\alpha}\def\b{\beta}\def\de{\delta}\def\eps{\varepsilon}

\def\pa{\partial}\def\pab{\bar\pa}\def\opa{\overline{\pa }}
\def\pas{\text{$\pa\mkern -10.0mu$\slash\,}} \def\sls#1{\text{$#1\mkern -13.0mu$\slash\,}}

\def \rectangle#1#2{\hbox{\vrule\vbox to #2 {\hrule\hbox to #1{\hfil}\vfil\hrule}\vrule}}
\def\Box{\square}\def\Boxr{\widetilde{\square}}\def\sq{\,\,\rectangle{7pt}{7 pt}\,\,}

\def\tr{\text{tr}}

\numberwithin{equation}{section}

\begin{document}

\title {On the asymptotic behavior of solutions to the Einstein vacuum equations in wave coordinates}
\author {Hans Lindblad}
\institute{Johns Hopkins University}
\titlerunning{Asymptotic behavior of solutions to Einstein's equations}
\maketitle
\begin{abstract}
We give asymptotics for Einstein vacuum equations in wave coordinates with small
asymptotically flat data. We show that the behavior is wave like at
null infinity and homogeneous towards time like infinity.
We use the asymptotics to show
that the outgoing null hypersurfaces approach the Schwarzschild ones for
the same mass and that the radiated energy is equal to the initial mass.
\end{abstract}

\section{Introduction}
The Einstein vacuum equations $R_{\mu\nu}=0$ in wave coordinates
become a system of nonlinear wave equations for the metric, called the
reduced Einstein equations
\begin{equation}\label{eq:reducedEinstein}
\widetilde{\Box}_g g_{\mu\nu} =F_{\mu\nu}(g)[\partial g,\partial g],
\qquad\text{where}\qquad \widetilde{\Box}_g\!=g^{\alpha\beta}\partial_\alpha\partial_\beta,\quad
g^{\alpha\beta}\!\!=\!(g^{-1})^{\alpha\beta}\!\!,
\end{equation}
is the reduced wave operator and
$F_{\mu\nu}$ are quadratic in $\partial g$.
The metric is assumed to have signature $(-1,1,1,1)$ and satisfy the wave coordinate condition
\begin{equation}\label{eq:WaveCordinateCond}
\partial_\alpha \big(\sqrt{|g|} g^{\alpha\beta}\big)=0,\qquad \text{where} \quad  |g|=|\det{\big(g\big)}|.
\end{equation}
 This is preserved by \eqref{eq:reducedEinstein} if data satisfies the constraint equations. The initial data are assumed to be asymptotically flat, i.e. for some small $M\!>\!0$ and
$0\!<\!\gamma\!<\!1$
\begin{equation}\label{eq:asymptoticallyflatdata}
g_{ij}\big|_{t=0}=(1+M r^{-1})\, \delta_{ij} + o(r^{-1-\gamma}),\quad
\pa_t g_{ij}\big|_{t=0}=o(r^{-2-\gamma}),\quad r=|x|.
\end{equation}
Choquet-Bruhat\cite{CB1} proved local existence to Einstein's equations equations in wave
coordinates.
Christodoulou-Klainerman \cite{CK} proved global existence of small solutions to
Einstein's equations in a coordinate invariant way. It was assumed that the wave
coordinates behaved badly for large times. Nevertheless in Lindblad-Rodnianski \cite{LR3}
we proved global existence in wave coordinates.
In this paper we are studying the precise asymptotic behaviour of solutions to
\eqref{eq:reducedEinstein}-\eqref{eq:WaveCordinateCond} that are small perturbations $g=m+h$ of Minkowski metric $m$.

The decay we prove is
$\varepsilon (t+r)^{-1}\!$
for tangential components of $h$ and
for all components
with a logarithmic loss close to the light cone, see section 1.3.
The asymptotics we give can roughly  be written in the form
\begin{equation}
h(t,r\omega)\sim H(r^*\!-t,\omega)/(t+r)+K\big(\tfrac{r^*\!-t}{t+r^*},\omega\big)/(t+r),
\qquad r^*\!\sim r+M\ln{r},
\end{equation}
where $\omega\!=\!x/|x|$.
$H$ is concentrated close to the outgoing light cones $r^*\!-t$ constant,
 $|H(q^*\!,\omega)|\!\leq\! \varepsilon (1\!+| q^*|)^{-\gamma^\prime}\!\!$,
 and $K$ is homogeneous of degree $0$ with a log singularity at the light cone
 $|K(s,\omega)|\!\leq \!\varepsilon \ln{|s|}$ for the nontangential components,
 see section \ref{sec:fullasymp}.
 $H$ is the radiation field of a free curved wave operator and
  $K$ is the backscattering of the wave operator with
   quadratic source terms.

  The estimates can be used for proving sharp decay of the curvature, weak Penrose peeling
  properties, as in \cite{CK}. We use the asymptotics to prove a Bondi type mass loss law,
  that the radiated energy equals the initial mass. The radiated energy is what is detected
   in the gravitational wave detectors \cite{HN,C2}. For coupling to matter fields or for
   scattering from infinity one needs to know the precise decay or asymptotics also in the
    interior. It is plausible our methods can be used for studying gravitational radiation from
    post-Newtonian sources
\cite{B} and polyhomogeneous expansions at null infinity \cite{CW}.
The method works for other wave equations with semilinear terms that
satisfy a weak null condition.

Below we give heuristics, present the results and explain the structure of the proof.
We start by reviewing the null structure and the global existence result of \cite{LR3}
  in section \ref{sect:existenceintr}.
In section \ref{sect:nonlineareffects} we give a heuristic explanation of the
nonlinear effects on the asymptotic behavior.
In section \ref{sect:sharpdecay} we give Einstein's equations in
asymptotic characteristic coordinates and we state the sharp decay estimates that we prove in sections 2 through 6 (assuming the decay estimates of \cite{LR3}).
In section 1.4 we give a heuristic explanation of the weak null condition
and the asymptotic expansion along outgoing characteristics towards null infinity.
In section 1.5 we state the asymptotics that we prove in sections 7 through 9,
first for tangential components at null infinity and later in the
interior (which depends on the former).
In section 1.6 we state the the asymptotics of the characteristic surfaces and
a Bondi type mass law
we prove in sections 10 and 11.

\subsection{Einstein's equation in wave coordinates, the weak null  structure
and global existence}
\label{sect:existenceintr}
Einstein's equations in wave coordinates form a system for $h=g-m$;
\beq\label{eq:hormander8}
\widetilde{\Box}_g h_{\mu\nu} =F_{\mu\nu} (h) (\pa h, \pa h),
\qquad \text{where}\quad
\widetilde{\Box}_g\!=\!\Box\!+\!\widetilde{h}^{\alpha\beta}\pa_\alpha\pa_\beta.
\eq
Here  $\Box\!=\pa_t^2\!-\!\triangle_x$,
$\widetilde{h}^{\alpha\beta}\!\!=\!g^{\alpha\beta}\!\!-\!m^{\alpha\beta}
\!=\!-h^{\alpha\beta}\!\!+\!O(h^2)$, where $h^{\alpha\beta}\!\!=m^{\alpha\mu} m^{\beta\nu}h_{\mu\nu}$,  and
\beqs
F_{\mu\nu} (h) (\pa h, \pa h)= P(\pa_\mu h,\pa_\nu h) +Q_{\mu\nu}(\pa h,\pa h) +G_{\mu\nu}(h)(\pa h,\pa
h),
 \eqs
where $Q_{\mu\nu}$ satisfy the standard null condition and $G_{\mu\nu}$
is cubic, and by \cite{LR1}
 \beq
 P(D,E)= D_{\alpha}^{\,\, \, \alpha}
E_{\beta}^{\,\, \, \beta}/4- D^{\alpha\beta}
E_{\alpha\beta}/2.\label{eq:P}
 \eq

General wave equations with quadratic nonlinearities  may blow up
as shown in \cite{J1,J2} for
 $\Box\phi=\phi_t^2
 $ or $\Box \phi=\phi_t\,
 \triangle_x\phi$.
The {\it null condition},  e.g. for $\Box
 \phi=\phi_t^2-|\nabla_x\phi|^2$,
 guarantees small data global existence \cite{C1,K1}.
 Einstein's equations in wave coordinates do however not satisfy the null condition.
 For the
 quasilinear equation
$\label{eq:hormander4}
\Box\, \phi= c^{\,\alpha\beta} \phi\,\,\pa_\alpha\pa_\beta \phi,
$
that resembles the quasilinear terms in Einstein's equations,
global existence was proven in \cite{L3,L4,A}.
A simple semilinear system that violates the null condition yet trivially
has global solutions is
\begin{equation}\label{eq:hormander5}
\Box
\,\phi_1=0,\qquad\qquad \Box \,\phi_2=(\pa_t \phi_1)^2.
\end{equation}
In \cite{LR1} we observed that the semilinear terms of \eqref{eq:hormander8}
in a null frame $\mathcal{N}$,
\beq\label{eq:frameintro}
\uL=\pa_t-\pa_r,\quad L=\pa_t+\pa_r, \quad S_1,S_2\in\bold{S}^2,\quad
\langle S_i,S_j\rangle =\delta_{ij}
\eq
  can be modeled by such a system. In fact, it is well known
that for solutions of wave equations derivatives tangential to the outgoing light
cones $\overline{\pa }\in\mathcal{T}=\{L,S_1,S_2\}$ decay faster. Modulo tangential
 derivative $\opa  h$ we have
 \beq\label{eq:transversalderivativeprojection}
 \pa_\mu h\sim L_\mu \pa_q h,\qquad\text{where}\quad \pa_q=(\pa_r-\pa_t)/2,
 \quad L_\mu=m_{\mu\nu} L^\nu,
 \eq
 and modulo quadratic terms with at least  one tangential derivative or cubic
 \beq\label{eq:einsteinfirstapproximationintro}
\widetilde{\Box}_g h_{\mu\nu}\sim L_\mu L_\nu  P(\pa_q h,\pa_q h),\quad
\text{where}\quad
\widetilde{\Box}_g  h_{\mu\nu}\sim \Box h_{\mu\nu} -h_{LL}\pa_q^2  h_{\mu\nu},
\eq
and $h_{LL}=h_{\alpha\beta} L^{\alpha}L^{\beta}$.
 In a null frame the semilinear terms become
\beq
(\widetilde{\Box}_g h)_{TU}\sim 0, \quad T\in \mathcal{T},U\in\mathcal{N}\qquad
(\widetilde{\Box}_g h)_{\uL\uL}\sim 4 P_{\!\mathcal{N}}(\pa_q h,\pa_q h),
\label{eq:simplifiedEinstein}
\eq
since $T^\mu L_\mu=0$, $T\in\mathcal{T}$.
Here by \cite{LR2}, \eqref{eq:P} in a null frame become
\begin{multline}\label{eq:Pnullframe}
P_{\!\mathcal{N}}(D,E)=-\big(D_{LL} E_{\underline{L}\underline{L}}
+D_{\underline{L}\underline{L}} E_{{L}{L}}\big)/8
-\big(2D_{AB}E^{AB}-
D_{\!\! A}^{\, A} E_{B}^{\,\,\,B}\big)/4\\
+\big(2D_{A L}E^{A}_{\,\,\, \underline{L}} +2D_{A
\underline{L}}E^A_{\,\,\,{L}}- D_{\!\!A}^{\, A}
E_{L\underline{L}}-D_{L\underline{L}}E_{\!A}^{\,\,A}\big)/4.
\end{multline}
Hence the right of  \eqref{eq:simplifiedEinstein} only contain
$\pa_q h_{\uL\uL}$ through the term $\pa_q h_{LL} \pa_q h_{\uL\uL}$.
 However, using \eqref{eq:transversalderivativeprojection} the wave coordinate condition \eqref{eq:WaveCordinateCond}
in a null frame becomes
\begin{equation}\label{eq:wavecordhormander12}
\pa_q h_{L T}\sim 0,\quad T\in\mathcal{T},\qquad
\delta^{AB}\partial_q  h_{AB}\sim 0,\quad A,B\in\mathcal{S}=\{S_1,S_2\},
\end{equation}
 modulo tangential derivatives, see \cite{LR2}, so
 $P_{\!\mathcal{N}}(\pa_q h,\pa_q h)\sim P_{\!\mathcal{S}}(\pa_q h,\pa_q h)$,
 where
\beq\label{eq:Psphere}
P_{\mathcal{S}} (D,E)= -D_{AB}\, E^{AB}/2,\qquad A,B\in\mathcal{S}.
\eq
Hence the right in \eqref{eq:simplifiedEinstein}
only depend on components we have better control on.
The main quasilinear term
\eqref{eq:einsteinfirstapproximationintro} is controlled integrating
\eqref{eq:wavecordhormander12} from data \eqref{eq:asymptoticallyflatdata}
\beq\label{eq:hLLasym}
h_{LL}\!\sim 2M\!/r .
\eq

In \cite{LR3} we found solutions to
Einstein's equations in wave coordinates:
\beq\label{eq:honedef}
g=m+h,\qquad h=h^0+h^1, \qquad
h^0_{\a\b}=\tilde{\chi}\big(\tfrac{r}{1+t}\big)\tfrac{M}{r}\de_{\a\b},
\eq
where $m$ is the Minkowski metric, $h^0$ is
the leading term at space-like infinity.
Here $\tilde{\chi}(s)\!=\!1$, when $s\!>\!1/2$ and $\tilde{\chi}(s)\!=\!0$, when $s\!<\!1/4$,
and $\tilde{\chi}^\prime(s)\geq 0$. The mass $M$ is assumed to be small and
$h_1$ and its derivatives are assumed to be small and satisfying he asymptotic flatness condition \eqref{eq:asymptoticallyflatdata} initially, i.e. decay like $r^{-1-\gamma}\!$. We showed that
the solution exist globally and satisfy the decay estimates
\begin{equation}\label{eq:veryweakdecay}
|Z^I
h^1(t,x)|\leq C_N \varepsilon
(1+t)^{-1+C_N\varepsilon}(1+q_+)^{-\gamma} ,\qquad |I|\leq N-2,
\end{equation}
where $q=r-t$ and $q_\pm=\max{(\pm q,0)}$. Here $N\geq 6$ and
$0<\gamma<1$
and $Z^I$ stands for a product of $|I|$
of the vector fields that commute with $\Box$ and the scaling vector field, i.e. $\pa_t, \partial_i$,
$x^i\partial_j-x^j\partial_i$, $x^i\partial_t+t\partial_i$, and
$t\partial_t+x^i\partial_i$, $i,j=1,2,3$.

Our first result says that one can almost remove
the $C_N\varepsilon$ in the exponent in \eqref{eq:veryweakdecay}
by changing to asymptotically characteristic coordinates, see section 1.3.

\subsection{Nonlinear effects on the asymptotic behavior}\label{sect:nonlineareffects}
There are two types of nonlinear distortions of the linear asymptotic behaviour related to the quasilinear terms $g^{\alpha\beta}\pa_\alpha\pa_\beta g_{\mu\nu}$ and the semilinear terms $F_{\mu\nu}(g)[\partial g,\partial g]$, respectively.

\subsubsection{Asymptotic Schwarzschild  coordinates}
In order to unravel the effect of the quasilinear terms one can change to
characteristic coordinates as in \cite{CK}, but this loses regularity and is not explicit.
Instead we use the asymptotic behavior of the metric to determine the characteristic surfaces asymptotically and use this to construct coordinates.
Due to \eqref{eq:hLLasym} the outgoing light cones
of a solution with asymptotically flat data \eqref{eq:asymptoticallyflatdata} approach those of the Schwarzschild metric
\beqs
-\,\frac {r\!-\!{M}\!/{2}\!}{r\!+\!{M}\!/{2}\!} \,\, dt^2
+\,\frac{r\!+\!{M}\!/{2}\!}{r\!-\!{M}\!/{2}\!}\,\, dr^2
+ \big(r+ \frac{M\!}{2\!}\big)^2 (d\theta^2 \!+ \,\sin^2\!\theta
\,d\phi^2)
\!\sim\!
\big(m_{\alpha\beta}+\frac{M\!\!}{r\!\!}\,\delta_{\alpha\beta}\big)\,dx^\alpha dx^\beta
\!\!.
\eqs
The outgoing light cones for the Schwarzschild metric satisfy $t\!\sim r^*\!\!-q^*\!$, where
$r^*\!=r\!+M\ln{r}$.
 We show that there is a solution to the eikonal equation that approaches the one for Schwarzschild
\begin{equation*}
g^{\alpha\beta} \pa_\alpha u \,\pa_\beta u=0,\qquad  u\to  \us=t-r^*,\quad \text{when}\qquad r>t/2\to\infty.
\end{equation*}
In fact the terms that could cause the largest deviation are controlled by the wave coordinate condition.
We therefore make the change of variables
$x=r\omega\to x^*\!=r^*\omega$, for large $r$, and the wave operator
 $\widetilde{\Box}_g$ asymptotically becomes the constant coefficient wave operator $\Box^{\,*}$ in the $(t,x^*)$ coordinates.
We use the vector fields $Z^*$, that commute with this wave operator $\Box^{\,*}$
and the scaling.

\subsubsection{Sources on light cones}
The inhomogeneous terms in \eqref{eq:reducedEinstein} cause a more serious distortion of the asymptotic behaviour. A solution of a linear homogeneous wave equation $\Box\phi=0$ decays like $t^{-1}$ and has a radiation field
\begin{equation*}
\phi(t,x)\sim F(r-t,\omega)/r.
\end{equation*}
The same is true if only
$|\Box \phi |\lesssim \langle t+r\rangle^{-2-\eps}\langle t-r\rangle^{-1}$
decay sufficiently fast.
However, quadratic inhomogeneous terms do not decay sufficiently
$$
\phi_t(t,x)^2\sim F_q(r-t,\omega)^2/r^2,
\qquad\text{where}\quad |F_q(q,\omega)|\lesssim (1+|q|)^{-1} .
$$
The asymptotics for the wave equation with such sources was studied in \cite{L1}:
\begin{equation} \label{eq:source}
-\Box \,\psi=n(r-t,\omega)/r^2,
\end{equation}
where $n(q,\omega)$ has compact support in $q$. The solution
is given by the formula
\beq\label{eq:exactsource}
\psi(t,r\omega)=\int_{r-t}^{\infty}\frac{1}{4\pi}\int_{\bold{S}^2}{\frac{
{n}({q,\sigma})}{t+q-r\langle \omega,{\sigma}\rangle}\,dS({\sigma})}\, dq.
\eq
Close to the light cone $t\!\sim \!r$ the integrand is concentrated when
$n(q,\sigma)\!\sim\! n(q,\omega)$:
\beq \label{eq:approximatesource}
\psi(t,r\omega)\sim \int_{r-t}^\infty \frac{1}{2r}\ln{\Big|\frac{t+q+r}{t+q-r}\Big|} \, n(q,\omega)\, dq,\qquad\text{when}\quad r\sim t,
\eq
and this leads to a log correction to the asymptotic behavior.
In fact an explicit calculation in spherical coordinates
$\Box \psi\!=r^{-1} (\pa_r\!-\pa_t)(\pa_r\!+\pa_t)(r\psi)+r^{-2}\triangle_\omega \psi$
shows that \eqref{eq:approximatesource} satisfies \eqref{eq:source} up to angular derivatives that decay faster.
\eqref{eq:exactsource} holds also in the interior
and gives $t^{-1}$ decay when $r\!<\!t/2$.

\subsection{Sharp decay in asymptotic Schwarzschild coordinates}
\label{sect:sharpdecay} We present the decay results from section 6
in the coordinates in section 1.2.1. With $\tilde{\chi}$ as in \eqref{eq:honedef} let
\beq\label{eq:goodcoordinates}
 r^*=r+M{\tilde{\chi}}\big(\tfrac{r}{1+t}\big)\ln{|r|},
 \qquad t^*=t,\qquad \omega^*=\omega,\qquad x^*=r^* \omega .
 \eq

\subsubsection{Einstein equations in asymptotic Schwarzschild
 coordinates} We express Einstein's equation in the good coordinates up
 to an error controlled by \cite{LR3}:
\begin{prop}
 \label{prop:approxwaveequationschwarzcoordintro} Let $P_\mathcal{F}$ be $P_{\mathcal{N}}$ or $P_{\mathcal{S}}$ as in \eqref{eq:Pnullframe} and \eqref{eq:Psphere} and let
 \begin{equation*}
 P_{\mu\nu}^*=P_\mathcal{F}\,(\pa_\mu^* h,\pa_\nu^* h)
\qquad\text{or}\qquad
 P_{\mu\nu}^*=L_{\mu}L_\nu P_\mathcal{F}\,(\pa_q^* h,\pa_q^* h).
 \end{equation*}
  Let $Z^*\!\!$ stand for the vector fields $\partial_\alpha^*$,
${x^*}^i\partial^*_j-{x^*}^j\partial^*_i$, ${x^*}^i\partial^*_t+t\partial^*_i$, and
$t\partial^*_t+{x^*}^i\partial^*_i$, $i,j=1,2,3$. Let $\gamma$ be as in \eqref{eq:veryweakdecay}. Then for $|I|\leq N-5$
and $\gamma^\prime<\gamma-C\varepsilon$
\begin{equation*}
 \big|{Z^*}{}^I\big[\Box^{*}
h_{\mu\nu}-P_{\mu\nu}^*\big]\big| \les
{\varepsilon^2 }
{(1+t+\!|q^*|)^{-2-\gamma^\prime}(1+\!|q^*|)^{-2+\gamma}},\qquad
\Box^*\!\!=m^{\alpha\beta}\pa_\alpha^*\pa_\beta^*.
\end{equation*}

\end{prop}

\subsubsection{The sharp decay estimates for tangential components of the metric}
The estimate for tangential components follows from commuting the wave equation above
with the tangential frame using that $P^*_{TU}\!\sim \opa h \, \pa h$  decay faster. (A difficulty is that the commutator with the angular part of
$\Box$ is large
in the interior.)

 \begin{prop}\label{prop:sharpmetricdecayintro1}
 For $T\!\in\!\{L,S_1,S_2\},\,\,
 U\!\in\!\{\underline{L},L,S_1,S_2\}$, $h^1\!$ as in
 \eqref{eq:honedef} we have
(Here \eqref{eq:metricdecaysharpwavecordintro} holds also for $Z^{*I}(h_{UV}^1\!)$ replaced by the Lie derivatives $(\mathcal{L}_{Z^{*}}^I h_{1})_{UV}$.)
 \begin{align}\label{eq:metricdecaysharptanintro}
 |{Z^*}{}^I {h}^{1}_{TU}|
 &\les {\varepsilon}{(1+t+r^*)^{-1}(1+q_+^*)^{-\gamma^\prime}},\quad
 \gamma^\prime\!<\!\gamma-C\varepsilon,\\
 |{Z^*}{}^I {h}^{1}_{LT}|\!+|\delta^{AB}{Z^*}{}^I {h}^{1}_{AB}|
 &\les\frac{\varepsilon}{1+t+r^*}
 \Big(\frac{1+q_-^*}{1+t+r^*}\Big)^{\gamma^\prime}\!\!\!,\qquad |I|\leq N-6.
\label{eq:metricdecaysharpwavecordintro}
 \end{align}

 \end{prop}

\subsubsection{The sharp decay estimates for all components}
The estimate for all components follows from using the estimates for tangential
 components in $P^*_{\mu\nu}$:
 \begin{prop}\label{prop:sharpmetricdecayintro2} With
 $\gamma^\prime=\gamma-C\varepsilon$ we have for $|I|\leq N-6$
 \begin{equation*}\label{eq:metricdecaysharp}
 |{Z^*}{}^I {h}^{1}|\les \frac{\varepsilon+\varepsilon^2
 S^0(t,r^*)}{(1+t+r^*)(1+q_+^*)^{\gamma^\prime}},\quad\text{where}\quad
 S^0(t,r^*)
 =\frac{t}{r^*}\ln{\Big(\frac{\langle \,t+r^*\,\rangle}{\langle\,t-r^*\,\rangle}\Big)}.
 \end{equation*}
 \end{prop}

\subsection{The asymptotics at null infinity}
Here we give the heuristics for the asymptotics along the surfaces $q^*\!=r^*\!-t$  constant towards null infinity.

\subsubsection{The weak null condition and the asymptotic system for wave equations}
Consider a general system of quasilinear wave equations in $3$ space
dimensions:
 \beq\label{eq:hormander1}
\Box \, \phi_I = {\sum}_{|\alpha|\le|\beta|\le 2,\,\, |
\beta|\geq 1}  A_{I,\alpha\beta}^{JK}\,  \pa^\alpha
\phi_J\, \pa^\beta\phi_K + {\text{cubic terms}},
 \eq
 with small initial data.
 If we neglect derivatives tangential to the outgoing light
cones and cubic terms, that decay faster, we get
\begin{equation*}
\aligned \Box\, \phi\! &=\!r^{-1}(\pa_t+\pa_r)(\pa_r-\pa_t)
(_{\!}r\phi_{\!})+r^{-2}\times \text{angular
derivatives},\\
\pa_\mu&=\tfrac{1}{2}\hat{\omega}_\mu(\pa_r-\pa_t)+\text{tangential
derivatives},\qquad \hat{\omega}=(-1,\omega),
\endaligned
\end{equation*}
where $x=r\omega$, $\omega\in \bold{S}^2$ are polar coordinates. We
see that asymptotically
\begin{equation*}
(\pa_t+\pa_r)(\pa_r\!-\pa_t)(r\phi_I )\sim {r}^{-1}\!\sum
 A_{I\!,nm}^{JK} 2^{-n-m}(\pa_r\!-\pa_t)^n (r\phi_{\!J})
\,\, (\pa_r\!-\pa_t)^m (r\phi_K),
\end{equation*}
where
$
A_{I\!,nm}^{JK}(\omega):=
\sum_{|\alpha|=n,\,|\beta|=m}
A_{I\!,\alpha\beta\,}^{JK} \hat\omega^\alpha \hat
\omega^\beta\!\,
$
and
$\hat{\omega}^\alpha\!=\hat{\omega}_{\alpha_1}\!\cdots\hat{\omega}_{\alpha_n}$.

\cite{H} proposed an asymptotic expansion as $r\!\to
\!\infty$ and $r\!\sim \! t$ of the form
\beqs\label{eq:hormander2}
\phi_I(t,x)\sim \, \Phi_I(q,s,\omega)/r,\quad \text{
where }\quad q=r-t,\,\,\,s=\ln{r},\,\,\,
\eqs
and $\Phi_I$ satisfies the {\it asymptotic system}:
\beq\label{eq:hormander3}
 2\partial_s\partial_q \Phi_I ={\sum}_{n\leq m\leq 2,\, m\geq 1}
  A_{I\, mn}^{JK}(\omega)\,
\partial_q^{\,m} \Phi_{\!J}\,\,\partial_q^{\,n} \Phi_K .
\eq
Solutions to linear wave equations have such an expansion
independent of $s$.
The {\it null condition},
which guarantees small data global existence \cite{C1,K1} is
$
A_{\!I\!,nm}^{JK} \!\equiv 0,
 $
 e.g. for $\Box
 \phi=\phi_t^2-|\nabla_x\phi|^2$.
 On the other hand \cite{J2} showed that solutions to
  $\Box \phi=\phi_t\,
 \triangle_x\phi$
 blowup and
 \cite{H} used  the blow up of the corresponding asymptotic systems
 to predict the precise exponential  blow up time.
 For the
 quasilinear equation
$\label{eq:hormander4}
\Box\, \phi= c^{\,\alpha\beta} \phi\,\,\pa_\alpha\pa_\beta \phi,
$
\cite{L3} observed that the asymptotic system
has global exponentially growing solutions
 $\Phi\sim e^{cs}$.
For the simpler semilinear system \eqref{eq:hormander5} that violates the null condition,
the solution to the asymptotic system
  \beq\label{eq:hormander6}
 2\partial_s \partial_q \Phi_1=0,\qquad 2\partial_s\partial_q \Phi_2=(\partial_q \Phi_1)^2,
 \eq
 is
 $
 \Phi_1\!=F_1(q,\omega),$ $\Phi_2\!= s\, F_2(q,\omega)+F_3(q,\omega),
 $ $F_2(q,\omega)\!=\int (\pa_q F_1(q,\omega))^2 dq/2$ so
 \begin{equation*}\label{eq:hormander7}
 \phi_1 \sim \, F_1(t-r,\omega)/r,
 \qquad \phi_2\sim  \ln{r} \,\, F_2(t-r,\omega)/r+F_3(t-r,\omega)/r.
 \end{equation*}

In view of these examples
we  say that \eqref{eq:hormander1} satisfies the
{\it weak null
condition} \cite{LR1} if \eqref{eq:hormander3} has global solutions growing at most
exponentially in $s$. The methods here works for the subclass where it grows at most polynomially in $s$. 

\subsubsection{The asymptotic system for Einstein's equations in wave coordinates}
With
\beqs
h_{\mu\nu}(t,x)\sim  \, H_{\mu\nu}(q,s,\omega)/r,\quad \text{
where }\quad q=r-t,\,\,\,s=\ln{r},\quad r\sim t,
\eqs
the asymptotic system for Einstein's equations in a null frame takes the
form:
\begin{align}
\big( 2\pa_s -H_{LL}\partial_q\big) \pa_q H_{TU}&= 0,\label{eq:hormandereinsteinone}\\
 \big( 2\pa_s -H_{LL}\partial_q\big)
 \pa_q H_{\underline{L}\underline{L}}&=4 P_{\bold{S}^2}(\pa_q H,\pa_q H),\label{eq:hormandereinsteintwo}
 \end{align}
 by \eqref{eq:simplifiedEinstein}. By \eqref{eq:wavecordhormander12}
 the wave coordinate condition takes the asymptotic form
\begin{equation*}\label{eq:hormander12}
\pa_q H_{L T}= 0,\qquad T\in\{L,S_1,S_2\},\qquad
\delta^{AB}\partial_q  H_{AB}=0,\qquad A,B\in\{S_1,S_2\},
\end{equation*}
and because the solution for large $r$ asymptotically is Schwarzschild
\beqs
H_{LL}=2M,\qquad H_{LA}=0,\qquad \delta^{AB}H_{AB}= 2M,
\qquad A,B\in\{S_1,S_2\}.
\eqs
Here the right hand side of
 \eqref{eq:hormandereinsteintwo} only depends on tangential components
$P_{\mathcal{S}} (D,E)
= -D_{AB}
\, E^{AB}\!/\,2
$
and the quasilinear term simplifies if we introduce the integral curves to the vector field
$(2\partial_s-2M\partial_q)$ given by
$r-t=q(s)=q^*-Ms=q^*-M\ln{(1+r)}$. Hence if we change
variables to $q^*$ and integrate
\begin{align*}
 H^*_{TU}(q^*,\omega,s)
&=H_{TU}^{\infty}(q^*,\omega),\\
 H^*_{\underline{L}\underline{L}}(q^*,\omega,s)&=
 H_{\underline{L}\underline{L}}^{\infty}(q^*,\omega)
 -\, s\,\int  P_{\mathcal{S}}\big(\pa_{q^*} H^{\infty},\pa_{q^*} H^{\infty}\big)(q^*,\omega)dq^* .
\end{align*}
We get an asymptotic radiation field as for a linear
homogeneous wave equation for all but one component which is
multiplied by a logarithm $s=\ln{(1+r)}$.

\subsection{The full asymptotics of the metric}\label{sec:fullasymp}
Using the decay estimates
we will prove the asymptotics stated below in sections 7-9.
First we prove the asymptotics for the tangential components
in section 1.5.1. First after that can one define the asymptotic
source term in section 1.5.2 since it depends on the
tangential components. Next we make a different decomposition into a part that comes
from the asymptotic source and a remainder. In section 1.5.3 we show that the remainder has the asymptotics of a free wave.
For the backscattering of the source we first give the asymptotics
close to the light cone in section 1.5.4 and in section 1.5.5
we give a formula that also gives the leading interior behaviour.

\subsubsection{Asymptotics for tangential components
close to the light cone}
Let
\beqs
 h^{*}_{\mu\nu}(t,r^*\omega)=h_{\mu\nu}(t,r\omega),\qquad
H^{*}_{TU}(q^*,\omega,r^*)=r^* h^{*}_{TU}(r^*-q^*, r^*\omega).
 \eqs
With similar estimates used to prove decay for tangential components we prove:
\begin{prop}\label{prop:lim} The limit
\begin{equation*}\label{eq:limit}
H^{\infty}_{TU}(q^*,\omega)
={\lim}_{\, r^*\to\infty}\,H^{*}_{TU}(q^*,\omega,r^*)
,\qquad
 T\in\{L,A,B\},\,\,\,\, U\in\{\underline{L},L,A,B\},
\end{equation*}
exists and satisfies $H_{TU}^{\infty}=H_{UT}^{\infty}$,
\begin{equation*}
H^{\infty}_{LA}(q^*,\omega)=0,\qquad
H^{\infty}_{LL}(q^*,\omega)=
{\delta}{}^{AB} H^{\infty}_{AB}(q^*,\omega)=2M.
 \end{equation*}
 For $|\alpha|+k\leq N-6$, we have
\begin{align*}
\big|\,\pa_\omega^\alpha \big((1+|\,q^*|)\pa_{q^*}\big)^k
H^{\infty}_{TU}(q^*,\omega)\big| &\les \varepsilon, \\
 \big|\pa_\omega^\alpha
\big((1+|q^*|)\pa_{q^*}\big)^k\big[H^{*}_{TU}(q^*,\omega,r^*)
 -H^{\infty}_{TU}(q^*,\omega)\big]\big| &\les
\varepsilon \Big(\frac{1+q_-^*}{1+t+r^*}\Big)^{\gamma^\prime}.
 \end{align*}
\end{prop}

\subsubsection{The asymptotics source}
Now, once that we have shown that the limit of the tangential components exist we can define
\begin{equation}\label{eq:Newsfucntion}
n(q^*,\omega)=-P_{\mathcal{S}}\big(\pa_{q*} H^\infty,\pa_{q^*} H^\infty\big)(q^*,\omega).
\end{equation}
It follows that
\begin{equation*}
\big|\pa_\omega^\alpha
\big(\langle q^*\rangle\pa_{q^*}\big)^k n(q^*,\omega)\big|\les \varepsilon^2(1+|q^*|)^{-2},
\end{equation*}
for $|\alpha|+k\leq N-7$. (Here $P_\mathcal{S}$ given by \eqref{eq:Psphere}.)
Let $k_{\mu\nu}$ be the solution to
 \beqs\label{eq:einsteinsourceintro}
-\Box^* k_{\mu\nu}=L_\mu(\omega) L_\nu (\omega)
{n(r^*-t,\omega)}{{r^{*}}^{-2}}\chi\big(\tfrac{\langle\,r^*-t\,\rangle}{t+r^*}\big)
\eqs
with vanishing data,
where $\chi(s)\!=\!1$, when $|s|\!\leq\! 1/2$ and $\chi(s)\!=\!0$, when $|s|\!\geq \!3/4$.
\begin{prop} We have
\begin{equation*}
\big|Z^{*I}\big(\Box^* h_{\mu\nu}-\Box^* k_{\mu\nu}\big)
\big| \les {\varepsilon^2}{(1+t+r^*)^{-2-\gamma+C\varepsilon}
(1+|q^*|)^{-2+\gamma}}\!,\quad |I|\!\leq\! N\!-\!7.
  \end{equation*}
  Let
 \beqs\label{eq:honeeintro}
{h}^{1\,e}=h-{h}^{0e},\qquad\text{where}\quad
{h}^{0e}_{\mu\nu}=\delta_{\mu\nu} M\chi^e(r^*-t)/r^*.
 \eqs
Here $\chi^e(s)=1$, when $s\geq 2$ and $\chi^e(s)=0$, when $s\leq 1$.
Then
 \beqs\label{eq:inhomwaveeqdecay10intro}
\big|\,Z^{*I}\big(h^{1e}_{\mu\nu}-k_{\mu\nu}\big) \big| \les
{\varepsilon^2}{(1+t+r^*)^{-1}(1+|q^*|)^{-\gamma^\prime}},
\qquad |I|\leq N-7.
\eqs
\end{prop}
These estimates tell us that the leading behavior in the exterior is determined by
 the Schwarzschild metric and the leading behavior in the interior by the solution to the
 wave equation with a source given by the far field $n(q^*\!,\omega)$.

\subsubsection{The asymptotics for all components of the metric in the interior}
Let
\begin{equation*}
 h^{e*}_{\mu\nu}(t,r^*\omega)=h^{1e}_{\mu\nu}(t,r^*\omega)-k_{\mu\nu}(t,r^*\omega),\qquad
H^{e*}_{\mu\nu}(q^*,\omega,r^*)=r^* h^{e*}_{\mu\nu}(r^*-q^*, r^*\omega).
 \end{equation*}
With similar estimates used to prove decay for all components we prove:
\begin{prop}\label{prop:lim} The limit
\beqs\label{eq:limit}
H^{e\infty}_{\mu\nu}(q^*,\omega)
=\lim_{r^*\to\infty}H^{e*}_{\mu\nu}(q^*,\omega,r^*),
\eqs
 exists and satisfies $H_{\mu\nu}^{e\infty}=H_{\nu\mu}^{e\infty}$.
 Moreover for $|\alpha|+k\leq N-7$ we have
\begin{align*}
\big|\,\pa_\omega^\alpha \big((1+|\,q^*|)\pa_{q^*}\big)^k
H^{e\infty}_{\mu\nu}(q^*,\omega)\big|
&\les {\varepsilon}{(1+|q^*|)^{-\gamma^\prime}},\\
\big|\pa_\omega^\alpha
\big((1+|\,q^*|)\pa_{q^*}\big)^k\big[H^{e*}_{\mu\nu}(q^*,\omega,r^*)
 -H^{e\infty}_{\mu\nu}(q^*,\omega)\big]\big|
 &\les
{\varepsilon }{(1+t+r^*)^{-\gamma^\prime}}.
 \end{align*}
\end{prop}

\subsubsection{Null infinity asymptotics of the solution with the asymptotic source}
Let
\begin{equation*}
k_{\mu\nu}^1(t,r^*\omega)=L_\mu(\omega) L_\nu(\omega) \int_{r^*-t}^\infty \frac{1}{2r^*}\ln{\Big(\frac{t+r^*+q^*}{t-r^*+q^*}\Big)} n\big(q^*,\omega\big)\, d q^*\, \chi\big(\tfrac{\langle\,r^*-t\,\rangle}{t+r^*}\big).
\end{equation*}
An explicit calculation in spherical coordinates show that
\begin{prop} For any $a<1$ we have
\begin{align*}
\big|Z^{*I}\Box^* (k_{\mu\nu}-k^1_{\mu\nu})(t,r^*\omega)\big|&\les\varepsilon {\langle t+r\rangle^{-3}}\ln{\big(\tfrac{\langle \,t+r^*\,\rangle}{\langle\,t-r^*\,\rangle}\big)}{\langle (r^*-t)_+\rangle^{-a}},\\
\big|Z^{*I}(k_{\mu\nu}-k^1_{\mu\nu})(t,r^*\omega)\big|&\les
{\varepsilon}{(1+t+r^*)^{-1}\langle (r^*-t)_+\rangle^{-a}} .
\end{align*}
\end{prop}
Let
\begin{equation*}
 k^0_{\mu\nu}(t,r^*\omega)=k^{1}_{\mu\nu}(t,r^*\omega)-k_{\mu\nu}(t,r^*\omega),\qquad
K^{0*}_{\mu\nu}(q^*,\omega,r^*)=r^* k^{0*}_{\mu\nu}(r^*-q^*, r^*\omega).
 \end{equation*}
With similar estimates used to prove decay for all components we prove:
\begin{prop}\label{prop:lim} The limit
\beqs \label{eq:limit}
K^{0\infty}_{\mu\nu}(q^*,\omega)
=\lim_{r^*\to\infty}K^{0*}_{\mu\nu}(q^*,\omega,r^*),
\eqs
exists and satisfies $K_{\mu\nu}^{0\infty}=K_{\nu\mu}^{0\infty}$,
and for $a<1 $ and $|\alpha|+k\leq N-7$
\begin{align*}
\big|\,\pa_\omega^\alpha \big((1+|\,q^*|)\pa_{q^*}\big)^k
K^{0\infty}_{\mu\nu}(q^*,\omega)\big|
&\les \varepsilon (1+q^*_+)^{-a},\\
 \big|\pa_\omega^\alpha
\big((1+|q^*|)\pa_{q^*}\big)^k\big[K^{0*}_{\mu\nu}(q^*,\omega,r^*)
 -K^{0\infty}_{\mu\nu}(q^*,\omega)\big]\big|
 &\les\varepsilon \Big(\frac{1+q_-^* }{1+t+r^*}\Big)^{a}.
\end{align*}
\end{prop}

\subsubsection{Interior asymptotics of the solution with the asymptotic source}
The results so far suffice to prove existence of the radiation
field, but to get more precise behavior we can
subtract off a better approximation using formulas from
\cite{L1}. Note that \eqref{eq:approxsourcesol2intro}
is asymptotically homogenous in any region
$r\!/t<c<1$.
\begin{prop}\label{lem:approxsource}
 Let $F$ and $\phi$ be as in the previous proposition
and set
 \beq\label{eq:approxsourcesol2intro}
k^2_{\mu\nu}(x^*,t)=\int_{r^*-t}^\infty \frac{1}{4\pi}\int_{\bold{S}^2}{\frac{
L_\mu(\sigma)L_\nu(\sigma) n({q^*},{\sigma})}{t+{q^*}-\langle\, x^*,{\sigma}\rangle}\,
dS({\sigma})}\,\chi\big(\tfrac{\langle\,{q^*}\,\rangle}{t+r^*}\big) d {q^*}.
\eq
 Then for any $a<1$ and $|I|+|J|+|K|\leq N-7$
 \beqs\label{eq:errorest}
|\,\Omega^I {S^*}^J\pa_t^K(k-k^2)|
 \les\varepsilon^2 (1+t+r^*)^{-1}(1+|\, r^*-t|)^{-a}.
 \eqs
\end{prop}

\subsection{Applications}
Here we use the decay estimates and asymptotics to study the
asymptotic surfaces and null foliation as in \cite{KN}
and the radiated energy and mass as in \cite{C2}.
The estimates and asymptotics  will also be used for proving global
 existence with matter fields and
scattering from infinity, in future work.

\subsubsection{Asymptotics of the characteristic surfaces}
We show that the eikonal eq.
\beq \label{eq:eikonalintro} g^{\alpha\beta}\pa_\alpha u\,\, \pa_\beta
u=0,\qquad \text{in}\quad r>|t|/2 ,
\eq
has a unique solution with asymptotic data at infinity $u\sim \us\!=t-r^*$, as $t\to\infty$:
\begin{prop} The eikonal equation \eqref{eq:eikonalintro} has a solution
$u\!=\tilde{u}+\us$ satisfying
\beq
{\sum}_{|I|\leq 2}|Z^{* I}\tilde{u}|+|Z^{* I}\tilde{v}|
\leq
C_1\varepsilon\Big(\frac{1+(r^*\!-|t|)_-}{1+t+|\,q^*|}\Big)^{\gamma\prime}, \quad  r>|t|/2.
 \eq
\end{prop}
By time reflection there is also a solution $v\!=\!\tilde{v}+v^*\!\!$, so
$v\!\sim v^*\!\!=t+r^*\!$,
as $t\!\to \!-\infty$.
\cite{CK} used $u$ to define modified vector fields. This shows
why it works with $\us\!$.

\subsubsection{The mass loss law}
Using that the asymptotic of $H_{LL}$ close to the light cone is $2M$ and on the other hand is given by the source there has to a be a relation between $M$ and the source formula above determined by $n$. We have
\begin{prop} Let $n(q^*,\omega)$ given by \eqref{eq:Newsfucntion} and let $M$ be as in \eqref{eq:asymptoticallyflatdata}. Then
\beqs
\frac{1}{2}\int_{-\infty}^{+\infty} \int_{\bold{S}^2}n(q^*,\omega) \frac{dS(\omega)\!}{4\pi}\, dq^*=M.
\eqs
\end{prop}
 The proposition in particular implies that, if $n\!=\!0$ then $M\!\!=\!0$, and then by the positive mass theorem the space time is Minkowski space.
The proposition can be interpreted as that the outgoing radiation equals to the initial mass.

\section{The weak decay of the metric in the Minkowski coordinates and decay from the wave coordinate condition.}
\label{sect:weakdecay}
We start from decay estimates for the metric from
 \cite{LR3}. There we constructed metrics $g$ satisfying
Einstein's equations in wave coordinates of the form
$$
g=m+h_0+h_1,\quad\text{where}\quad
h^0_{\a\b}=\tilde{\chi}\big(\tfrac{r}{1+t}\big)\tfrac{M}{r}\de_{\a\b},
$$
where $m$ is the Minkowski metric and  $h^0\!$ is
picking up the decay at space-like infinity.
Here $\tilde{\chi}(s)\!=\!1$, when $s\!>\! 1\!/2$, $\tilde{\chi}(s)\!=\!0$,
when $s\!<\!1\!/4$,
and $\tilde{\chi}^\prime(s)\!\geq \!0$. We proved that
for any $N\!\geq \!6$ and $0\!<\!\gamma\!<\!1$ there are solutions satisfying
\begin{equation*} \label{eq:enerNdef}
E_N(t)\!=\!\!{\sum}_{|I|\leq N} \|w^{1/2}\pa Z^I
h^1(t,\cdot)\|_{L^2}\leq C_N \varepsilon
(1\!+t)^{C_N\varepsilon}\!\!,\qquad
 w^{1/2}\!=\big(1\!+q_+\big)^{1/2+\!\gamma}\!\!,
\end{equation*}
where $q\!=\!r-t$ and $q_\pm\!\!=\!\max{(\pm q,0)}$, provided that
this norm initially
and $M\!\!\leq\!\varepsilon^2$ are small.
(Here $Z^I$ stands for a product of $|I|$
of the vector fields that commute with the constant coefficient wave
operator and the scaling vector field, i.e.$\pa_t$, $\partial_i$,
$x^i\partial_{j}-x^j\partial_i$, $x^i\partial_t\!+t\partial_i$, and
$t\partial_t\!+x^i\partial_i$, $i,\!j\!=1,2,3$.) In \cite{LR3} we proved:
\begin{prop}[Weak decay]\label{prop:weakdecay}
For $|I|\leq N-4$ we have \begin{equation}
| \,\pa Z^I h^1| +\frac{|Z^I h^1|}{1\!+|q|}+|\,\pab Z^I h^1|\,
\frac{1\!+t\!+r}{1+|q|} \les \frac{\varepsilon(1\!+q_+)^{-\gamma}}
{(1\!+t\!+r)^{1-C\varepsilon}(1\!+|q|)}.
\label{eq:decay1}
\end{equation}
 This estimates, but with $(1\!+q_+)^{-\gamma}\!$ replaced by
 $(1\!+q_+)^{-C\varepsilon}\!\!\!$, hold also for $h^0\!\!+h^1\!\!$.
\end{prop}
Here $\overline{\partial}$ stands for derivatives tangential to
the outgoing light cones, i.e. linear combinations of $L,S_1,S_2$
where $L\!=\partial_t\!+\partial_r$, $\underline{L}\!=\partial_t\!-\partial_r$,
and $S_1,S_2$ are orthonormal vectors
(in the Minkowski metric) that span the tangent space of
the sphere $\bold{S}^2\!\!$ whose components are independent of $t,r$ (i.e. $S_i\!=S_i^k(\omega)\pa_k$).
$A,B,C,D$ will denote any of the vector fields
$S_1,S_2$.
Repeated use of these are summed over.

Similarly we express the inverse of the metric as
\beqs
g^{\mu\nu}=m^{\mu\nu}+h_0^{\mu\nu}+h_1^{\mu\nu},\quad
h_0^{\mu\nu}=-\tilde{\chi}\big(\tfrac{r}{1+t}\big)\tfrac{M}{r}\de^{\mu\nu}.
\eqs
Then $m^{\mu\nu}\!\!+h_0^{\mu\nu}\!\!\!-h^{\!1\mu\nu} $ is an approximate
inverse to $g_{\mu\nu}\!=m_{\mu\nu}\!+h^0_{\mu\nu}\!+h^1_{\mu\nu}$ up
to $O(h^2)$ so
$h_1^{\mu\nu}\!\!=\!-h^{\!1 \mu\nu}\!\! +O(h^2)$.
Therefore $h_1$ will satisfy the same estimates \eqref{eq:decay1}.
Certain components of $h^{\mu\nu}\!\!=g^{\mu\nu}\!\!-m^{\mu\nu}$ expressed in a null frame ${h}_{UV}\!\!=\!U_\mu V_\nu h^{\mu\nu}\!\!$,\linebreak where $V_{\!\mu}\!\!=\!m_{\mu\nu\!}V^\nu\!\!$ and
$U,V\!\!\in\! \mathcal{N}\!=\!\{L,\underline{L},S_1,S_2\}$, have improved decay. This comes
from the wave
coordinate condition; $\pa_\mu \big (g^{\mu\nu}\!\!\!\sqrt {|_{\!}\det
g|}\big)\!=0$ that can be expressed
\begin{equation} \pa_\mu \widehat{h}{}^{\mu\nu\!}\!
= \Lambda^{\nu}(h,\pa h), \quad
\text{where}\quad  \widehat{h}{}^{\mu\nu}\!\!\!=\!h^{\mu\nu}\!\!\!-m^{\mu\nu}
\tr h_{\!}/2,\quad \tr h\!=m_{\alpha\beta} h^{\alpha\beta}\!\!,\label{eq:wavehatdivergence}
 \end{equation}
and $\Lambda^{\nu}(h,\pa h)\!=(m^{\mu\nu\!}m_{\alpha\beta}\!-\!g^{\mu\nu} \! g_{\alpha\beta})
\pa_\mu g^{\alpha\beta}\!/2=O(h\,\pa h)$.
 Using this we get
\begin{prop}[Weak wave coordinate decay]\label{prop:wavecoorddecay}
\!\!For $\!|I|\!\leq\! N\!-4\!$ and $r^{\!*\!}\!\geq \!t_{\!}/8\!:$
\begin{align}
 |\,\pa Z^I {h}_{1LT}|
+|\,\pa Z^I \delta^{AB}{h}_{1AB} |&\les
\varepsilon(1+t+r)^{-2+C\varepsilon}(1\!+q_+)^{-\gamma},
\label{eq:wavecoordinatederivative}
\\
|Z^I {h}_{1LT}| \!+|Z^I\delta^{AB}{h}_{1AB}|
\les \varepsilon &
(1\!+t\!+r)^{-1-\gamma+C\varepsilon}\!\! +
\varepsilon(1\!+t)^{-2+C\varepsilon}(1\!+q_-).
\label{eq:wavecoordinatefunction}
\end{align}
\eqref{eq:wavecoordinatederivative}-\eqref{eq:wavecoordinatefunction} also hold
for $Z^{I \!}(h_{LT})$ replaced by $(\mathcal{L}_Z^I h)_{LT}$
 and $Z^{I\!} (h_{AB})$ by $(\mathcal{L}_Z^I h)_{\!AB}$, where
$\mathcal{L}_Z h^{\mu\nu}\!\!=\!Zh^{\mu\nu}\!\!-\pa_{\alpha\!} Z^\mu  h^{\alpha\nu}\!\! -\pa_{\alpha\!} Z^\nu h^{\mu\alpha}\!$ is the Lie derivative,
$Z\!\in\!\mathcal{N}\!$.
\end{prop}
\begin{remark} In general $(\mathcal{L}_Z h)_{LT}\!\!\neq \!Z (h_{LT})\!$ but
$(\mathcal{L}_\Omega h)_{LL}\!=\Omega (h_{LL})$.
\end{remark}

\begin{lemma}\label{lem:nulldivergence} With $\pa_q=(\pa_r\!-\pa_t)/2 $, $\pa_s=(\pa_r\!+\pa_t)/2$
and sum over $A=S_1,S_2$,
\begin{equation}\pa_q \big(L_\mu U_\nu
k^{\mu\nu}\big)\!=  \pa_s
\big(\Lb_{\,\mu} U_\nu k^{\mu\nu}\big) -A_\mu U_\nu \pa_A
k^{\mu\nu}+U_\nu \,\partial_\mu k^{\mu\nu}\!\!,
\qquad U\in\mathcal{N}.\label{eq:wavecoordinateframedivergence}
\end{equation}
Moreover
\begin{multline}\pa_q (r^2 L_\mu L_\nu
k^{\mu\nu})\!=  \pa_s(r^2\Lb_{\,\mu}L_\nu k^{\mu\nu})
+r\,\delta^{AB}\!A_\mu  B_\nu k^{\mu\nu}\!\!
+L_\nu r^2\partial_\mu k^{\mu\nu}\\
-r^2 \pa_A (A_\mu  L_\nu k^{\mu\nu})
+r^2 (\pa_i A^i)A_\mu L_\nu  k^{\mu\nu}
\!\!.\label{eq:wavecoordinateframeRdivergenceL}
\end{multline}
\end{lemma}
\begin{proof} \eqref{eq:wavecoordinateframedivergence} follows from
expressing the divergence in the null frame,
$\pa_\mu F^\mu\!=L_\mu\pa_q F^\mu -\Lb_{\,\mu} \pa_s F^\mu\!+A_\mu \pa_A F^\mu\!,$
and using that $\pa_s$ and $\pa_q$ commute
with the frame.
\eqref{eq:wavecoordinateframeRdivergenceL} follows from
\eqref{eq:wavecoordinateframedivergence} since
$L_\mu \pa_s r-\Lb_{\,\mu} \pa_q r\!=R_\mu$, $R\!=\!(0,\omega)$, $rA^i\pa_i L_\mu\!\!=\!A_\mu$ and
\beq
A^i\pa_i A_j=\pa_i (A_j A^i)-A_{\!j\,}\pa_i A^i\!
=\pa_i (\delta_i^{\, j}\!-\omega_{\!j}\, \omega^i) -A_{\!j\,}\pa_i A^i\!=-2\omega_{\!j}/r -A_{\!j\,}\pa_i A^i\!.
\eq

\end{proof}
\begin{lemma} \label{lem:Liederivtive} If $Z\!=\!Z^\alpha\pa_\alpha$, where $\pa_\beta
Z^\alpha\!$ are constants then
 ${\cal L}_Z \pa_\mu k^{\mu\nu}\!\!=\!\pa_\mu {\cal L}_Z k^{\mu\nu}\!\!$.
\end{lemma}

\begin{proof}[of Proposition \ref{prop:wavecoorddecay}]
\eqref{eq:wavecoordinatefunction} follows from integrating
\eqref{eq:wavecoordinatederivative}
 in the $t-r$ direction from initial data.
 \eqref{eq:wavecoordinatederivative} follows from \eqref{eq:decay1} and
\beq\label{eq:wavecoord}
|\pa_q Z^I h_{1LT}|+|\pa_q Z^I \delta^{AB} h_{1AB}|
\les \sum_{|J|\leq |I|+1} \frac{|Z^J h_1|}{1\!+t\!+|\,q|}
+\!\!\sum_{|J|+|K|\leq |I|\!\!\!\!\!\!\!\!\!\!\!\!\!\!\!} |Z^J h| \,|\pa Z^K h|.
\eq
It suffices to
prove \eqref{eq:wavecoord} for $r\!>\!3_{\,} t_{\!}/4$ using that
$|\partial \phi|\leq C (1\!+|_{\,}t-r|)^{-1}  \!\sum_{|J|=1} |Z^J\phi|$.

With
 $\widehat{h}{}_i^{\mu\nu}\!\!\!=\!h_i^{\mu\nu}\!\!-m^{\mu\nu}\tr\, h_{i\!}/2$,
 where $\tr\, h\!\!=m_{\alpha\beta} h^{\alpha\beta}
 \!\!=\!-{h}_{\underline{L} L}\!+\delta^{CD}{h}_{CD}$,
  we have
  \beqs \pa_\mu \widehat{h}{}_0^{\mu\nu}\!\! =-\pa_\mu
\big(\tilde{\chi}\big(\tfrac{r}{1\!+t}\big)
M r^{-1}\big(\delta^{\mu\nu}\!-m^{\mu\nu}\big)\big)\!
=2\tilde{\chi}^{\,\prime}\big(\tfrac{r}{1+t}\big)M(1\!+t)^{-2}\,
\delta^{\nu\,0},
\eqs
 and hence
 \beqs
 \pa_\mu \widehat{h}{}_1^{\mu\nu}\!=
-\tfrac{1}{2}\big(g^{\mu\nu} g_{\alpha\beta}
-m^{\mu\nu}m_{\alpha\beta}\big)
\pa_\mu g^{\alpha\beta}
-2\tilde{\chi}^{\,\prime}\big(\tfrac{r}{1+t}\big) M(1\!+t)^{-2}\delta^{0\nu}.
\eqs
 We are now going to commute vector fields through the equations
 in Lemma \ref{lem:nulldivergence}:
$$
[Z,\pa_q] = C^Z_{\underline{L}} \pa_q + C^Z_{L} \pa_s + C^Z_A
\pa_A,\qquad t/2<r<2t,
$$
for some smooth homogeneous of degree $0$ functions $C^Z_{U}$. On
the other hand
\begin{equation*}\label{eq:express}
\pa_s = {\sum}_{|J|=1} \,\frac {c^L_J Z^J}{t+r},\qquad
 \pa_C = {\sum}_{|J|=1}\, \frac {c^C_J Z^J}{t+r},\qquad t/2<r<2t,
\end{equation*}
with some smooth $c^{\,U}_J$ homogeneous of degree $0$.
Using these identities for the terms in right of \eqref{eq:wavecoordinateframedivergence} and terms generated in the
commutator $[Z,\pa_q]$ we obtain
\beqs |\pa_q  Z^I \widehat{h}_{1 LU} |\les
{\sum}_{|J|\leq |I|+1} \,\,\frac{|Z^J \widehat{h}_1|}{1\!+t+r}
+{\sum}_{|J|+|K|\leq |I|}\,\, |Z^J h|\, |\pa Z^K h|.
\eqs
\eqref{eq:wavecoord} follows since $\widehat{h}_{LT}=h_{1LT}$ and
$\widehat{h}{}_{1L\underline{L}}=\delta^{AB}{h_1}_{AB}$.
\end{proof}

There is an improvement in the $L$ component of the  quadratic terms in \eqref{eq:wavehatdivergence}
\beq
\Lambda_L=O(h\opa h)+O(h_{LL}\pa \tr h)+O(h^2\pa h).\label{eq:LambdaL}
\eq

\section{The asymptotic approximation to Einstein's equations}
Using the decay estimates from section 2 we can neglect terms that decay faster
 (even if they depend on higher derivatives).
From \cite{LR3} we know that
$$
\Boxr_g h_{\mu\nu}=F_{\mu\nu}(h)(\pa h,\pa h)\sim P\,(\pa_\mu
h,\pa_\nu h),
$$
where $\Boxr_g=g^{\alpha\beta}\pa_\alpha\pa_\beta$ and
\beqs P
\,\big( h,k\big)= \big(  m^{\alpha\beta}
m^{\alpha^\prime\beta^\prime}\!\!/\,4 -m^{\alpha\alpha^\prime}
m^{\beta\beta^\prime}\!\!/\,2\big) h_{\alpha\beta}
k_{\alpha^\prime\beta^\prime}
 \eqs
and $F_{\mu\nu}(h)(\pa h,\pa h)$ are quadratic forms in $\pa h$ depending on $h$  such that
 \begin{multline} \big| Z^I
\big(F_{\mu\nu}(h)(\pa h,\pa h)-P\,(\pa_\mu h,\pa_\nu
h)\big)\big| \\
\les \sum_{|J|+|K|\leq |I|}\!\!\!\!\!\! |\,\overline{\pa}Z^{J} h|\, |\,\pa Z^{K}h|
+\sum_{|J_1|+\dots+|J_k|+|K|+|L|\leq |I|,\,1\leq k\leq |I|
\!\!\!\!\!\!\!\!\!\!\!\!\!\!\!\!\!\!\!\!\!\!\!\!\!\!\!\!\!\!\!\!\!\!\!\!\!
\!\!\!\!\!\!\!\!\!\!\!\!\!\!\!\!\!\!\!\!\!\!\!\!\!\!\!\!\!\!\!\!\!}
|Z^{J_1} h|\cdots|Z^{J_k} h|\,|\pa Z^K h|\, |\,\pa Z^L h|.
\label{eq:Festimate}
\end{multline}
We can express the tensors $h_{\mu\nu}$ in the null frame
$h_{UV}=U^\mu V^\nu h_{\mu\nu}$ and we can express the
metric $m^{\alpha\beta}$ and hence
$P$
in terms of the null frame. By \cite{LR2};
\begin{multline}\label{eq:nullframeP}
P_\mathcal{N\,}(h,k)
=-\frac{1}{8}\big({h}_{LL}
{k}_{\underline{L}\underline{L}}
+{h}_{\underline{L}\underline{L}}
{k}_{{L}{L}}\big) -\frac{1}{4}\delta^{CD}\delta^{C^\prime
D^\prime}\big(2{h}_{CC^\prime}{k}_{DD^\prime}-
{h}_{CD} {k}_{C^\prime D^\prime}\big)\\
+\frac{1}{4}\delta^{CD}\big(2{h}_{C
L}{k}_{D\underline{L}} +2{h}_{C
\underline{L}}{k}_{D{L}}- {h}_{CD}
{k}_{L\underline{L}}-{h}_{L\underline{L}}
{k}_{CD} \big).
\end{multline}
 The special structure is important;
 the worst component ${h}_{\underline{L}\underline{L}}$
 is multiplied with a good component ${h}_{LL}$ that can be controlled by the wave coordinate condition:
 \begin{equation}\label{eq:PN}
P_\mathcal{N\,}({h},{k})
 =-\frac{1}{8}\big({h}_{LL}
 {k}_{\underline{L}\underline{L}}
 +{h}_{\underline{L}\underline{L}}
{k}_{{L}{L}}\big)
 +\!\!\!\!\!\!\!\!\!\!\sum_{S,T\in\,{\mathcal T}\!, \,\, U,V\in\,{\mathcal N}}
\!\!\!\!\!\!\!\!\!\! c^{U\!S\,VT}
{h}_{US}{k}_{VT}.
 \end{equation}
 where the sum is over $S,T\in {\mathcal T}=\{L,S_1,S_2\}$ and
 $U,V\in {\mathcal N}=\{\underline{L},L,S_1,S_2\}$.
 \begin{remark} $P_\mathcal{N\,}(h,k)$ is a bilinear form on tensors
 expressed in frame $\mathcal{N}$ whereas $P(h,k)$ is a bilinear form on tensors in the original coordinates. Now $P_\mathcal{N\,}(h,k)=P(h,k)$,
 and $P_\mathcal{N\,}(\pa_q h,\pa_q k)\!=\!P(\pa_q {h},\pa_q{k})$
 since $\pa_q$ commutes with contractions with the frame.
 However, by $P_\mathcal{N\,}(\pa_\mu h,\pa_\nu h)$ we mean the form acting on the tensors $\pa_\mu h_{UV}=\pa_\mu \big(h_{\alpha\beta}U^\alpha V^\beta\big)\!\neq\! U^\alpha V^\beta \pa_\mu h_{\alpha\beta}$
 which is different from $P(\pa_\mu h,\pa_\nu h)$:
 \begin{equation*}
P_\mathcal{N\,}
(\pa_\mu{h},\pa_\nu{k})
 =-\frac{1}{8}\big(\pa_\mu{h}_{LL}
 \pa_\nu{k}_{\underline{L}\underline{L}}
 +\pa_\mu{h}_{\underline{L}\underline{L}}
\pa_\nu{k}_{{L}{L}}\big)
 +\!\!\!\!\!\!\!\!\!\!\!\!\!\!\sum_{S,T\in\,{\mathcal T}\!, \,\, U,V\in\,{\mathcal N}}
 \!\!\!\!\!\!\!\!\! c^{U\!S\,VT}
\pa_\mu{h}_{US}\pa_\nu{k}_{VT}.
 \end{equation*}
\end{remark}
\begin{prop}[Asymptotic Approximate Einsten's equations]
\label{prop:approxwaveequation}
Let
\begin{equation*}
P_{\mu\nu}=\chi\big(\tfrac{\langle \,r-t\,\rangle }{t+r}\big) P_\mathcal{N\,}
 (\pa_\mu
 {h},\pa_\nu{h}),\qquad\text{or}\qquad
 P_{\mu\nu}=\chi\big(\tfrac{\langle \,r-t\,\rangle }{t+r}\big)
 L_{\mu}L_\nu P_\mathcal{N\,}(\pa_q h,\pa_q h).
 \end{equation*}
 where $\chi_{\!}\!\in\!  C_0^\infty\!\!$ satisfies
 $\chi({q})_{\!}\!=\!0$, when
$|{q}|\!\geq \!3_{\!}/4$ and $\chi({q})_{\!}\!=\!1$, when $|{q}|\!\leq\! 1_{\!}/2$.\!\!
Then
\begin{equation*}
 \Big|Z^I\big[\Box_{\,0}\,
h_{\mu\nu}-P_{\mu\nu}\big]\Big|\\
\les \frac{\varepsilon^2}{(1\!+t\!+|q|)^{3-C\varepsilon}(1\!+|q|)}
+\frac{\varepsilon^2}{(1\!+t\!+|q|)^{2+\gamma-C\varepsilon} (1\!+|q|)^{2}},
\end{equation*}
 for $|I|\leq N-4$. Here
 the asymptotic Schwarzschild wave operator is given by
\beq\label{eq:box0def}
\Box_{\,0}=\big(m^{\alpha\beta}+h_0^{\alpha\beta}\big)\pa_\alpha\pa_\beta,
\quad\text{where}\quad
h_0^{\alpha\beta}=-\tfrac{M}{r}\tilde{\chi}\big(\tfrac{r}{1+t}\big)\,
\delta^{\alpha\beta}.
\eq
\end{prop}
The proof will be a consequence of the following lemmas and previous estimates.
\begin{lemma}\label{lem:lemma1} We have $P_\mathcal{N\,}(\pa_q h,\pa_q k)
 =P(\pa_q {h},\pa_q{k})$ and
\begin{align*}\Big|Z^I\Big[P(\pa_\mu h,\pa_\nu
h)-P_\mathcal{N\,}
 (\pa_\mu {h},\pa_\nu{h})\Big]\Big|&\les
\sum_{|J|+|K|\leq |I|+1}\frac{|Z^{J} h|}{1\!+t\!+|q|}\, |\pa Z^K h|,\\
 \Big|Z^I\Big[P(\pa_\mu h,\pa_\nu h)
-L_\mu
L_\nu P_\mathcal{N\,}(\pa_q h,\pa_q h)\Big]\Big|&\les \sum_{|J|+|K|\leq
|I|+1}\frac{|Z^{J} h|}{1\!+t\!+|q|}\, |\pa Z^K h|.
\end{align*}
 \end{lemma}
 \begin{proof} The first inequality follows since if
 $U\in {\mathcal N}=\{\underline{L},L,S_1,S_2\}$ then
 \beq\label{eq:express}
 \pa_\mu U^\nu
 =c_{\mu\nu}^U(\omega,r/t) /(t+|q|),\qquad t/8<r<8t,
 \eq
 for some smooth functions $c_{\mu\nu}^U(\omega,r/t)$ homogeneous of
 degree $0$.
 The second inequality follows from using \eqref{eq:express} after expanding in the
null frame
 \beq\label{eq:derframe}
\pa_\mu \!=L_\mu \pa_q \! - \underline{L}_\mu
\pa_s \!+A_\mu \pa_A ,\qquad
\pa_q\!=(\pa_r\!-\pa_t)/2 ,\quad  \pa_s\!=(\pa_r\!+\pa_t)/2.
 \eq
\end{proof}

 We can further decompose $
g^{\alpha\beta}\!\!=m^{\alpha\beta}\!\!+h_0^{\alpha\beta}\!\!+h_1^{\alpha\beta}\!$,
where $h_0^{\alpha\beta}\!$ is given by \eqref{eq:box0def}. With the asymptotic Schwarzschild wave operator
given by \eqref{eq:box0def}
we write \beq \Box_{\,0}
\,h_{\mu\nu}=F_{\mu\nu}-F^1_{\mu\nu},\qquad
F^1_{\mu\nu}=h_1^{\alpha\beta}\pa_\alpha\pa_\beta h_{\mu\nu} .\eq
 To estimate $F_1$ we use the following;
 \begin{lemma}\label{lem:lemma2} We have
 \begin{equation*}
\big|Z^I (k^{\alpha\beta}\pa_\alpha\pa_\beta \phi)\big|\les
{\sum}_{|J|+|K|\leq |I|+1}
\frac{|Z^J\! k_{LL}| |\pa Z^K \phi|}{1+|q|}\\
 + \frac{|Z^J \! k|  |\pa Z^K \phi|}{1+t+r}.\,
\end{equation*}
\end{lemma}
\begin{proof} If we expand in null frame, $k^{\alpha\beta}\pa_\alpha\pa_\beta
 = k^{UV}U^\alpha V^\beta \pa_\alpha\pa_\beta$, and use that
 \beqs
\uL^\alpha \pa_\alpha= {\sum}_{|J|=1}\frac{c_{J}^{\uL}(\omega,r/t)}{t-r} Z^J,\quad
 T^\alpha \pa_\alpha= {\sum}_{|J|=1}\frac{c_{J}^T(\omega,r/t)}{t+r} Z^J,\quad T\in
 \mathcal{T},
 \eqs
for some smooth functions $c_{J}^U(\omega,r/t)$,  when $t/8<r<8t$,  we can write
 \beqs
k^{\alpha\beta}\pa_\alpha\pa_\beta =
{\sum}_{|J|= 1}\, k^{\uL\uL}
\frac{C_{J\uL\uL}^\gamma(\omega,r/t)}{t-r} Z^J\pa_\gamma
+{\sum}_{|J|= 1}\, k^{\alpha\beta}
\frac{C_{J\alpha\beta}^\gamma(\omega,r/t)}{t+r} Z^J\pa_\gamma,
 \eqs
 for some smooth functions $C_{J\uL\uL}^\gamma(\omega,r/t)$ and
 $C_{J\alpha\beta}^\gamma(\omega,r/t)$, where $k^{\uL\uL}=k_{LL}/4$.
 When $|t-r|<1$ we replace the first sum with
 $k^{\uL\uL}L^\alpha L^\beta\pa_\alpha\pa_\beta$
 and similarly for the second when $t+r<1$. This proves the lemma when
 $|I|=0$. To prove the lemma in general we just have to note that
 $Z (t-r)=c_Z(\omega)(t-r)$.
\end{proof}


\section{Asymptotic Schwarzschild coordinates}
Recall from the previous section, the asymptotic Schwarzschild wave
operator:
\beq\label{eq:boxzerosphericalcoordinates}
\Box_{\, 0} =\Box\,
-\frac{\chi_0}{r}\big(\pa_t^2\!+\!\triangle_x\!\big)
=-\big(1+\frac{\chi_0}{r}\big)\pa_t^2+
\big(1-\frac{\chi_0}{r}\big)
\big(\pa_r^2+\frac{2}{r}\pa_r+\frac{1}{r^2}
{\triangle}_\omega\big) ,
\eq
 where $\chi_0\!=M{\tilde{\chi}}\big(\tfrac{r}{1+t}\big)$
and ${\triangle}_\omega$ is the
 Laplacian on the sphere. The
wave operator
 $\Box_{\,0}\!$ is better
 expressed in asymptotic Schwarzschild coordinates.
 When $r\!>\!t/2$ these are the Regge-Wheeler coordinates
 that transforms the wave operator in the Schwarzschild metric in
 the radial case to the constant coefficient operator.
 In the interior $r\!<\!t/4$ these are the regular coordinates.
 Specifically,
 \beqs
 r^*=r+\chi_0\ln{r},
 \qquad t^*=t,\qquad \omega^*=\omega,\qquad x^*=r^* \omega .
 \eqs
We show that $\Box_{\,0}$ is close to the flat wave operator
in Schwarzschild coordinates:
 \beq\label{eq:boxstarsphericalcoordinates}
 \Box^{\,*}\!=-\pa_t^2+\triangle_{x^*}\!\!
 =-\pa_{t^*}^2+\pa_{r^*}^2\!+2{r^*}^{-1}\pa_{r^*}\!+{r^*}^{-2}
 \triangle_\omega.
 \eq
 Since $\pa r^{*\!}\!/\pa r\!\sim  1\!+\!\chi_{0\!}/r$ and
 $r^*\!/r\!=\! 1\!+\!\chi_0 \ln{r}\!/r$ it follows that
 \beqs
 \Box^{\,*}\! \sim -\pa_t^2+\big(1+{\chi_0}/{r}\big)^{\!-2}\pa_r^2\!
 +{2}{r}^{-1}\pa_{r}\!+{r}^{-2}\triangle_\omega\sim\Box_0.
 \eqs
\begin{prop}[Asymptotic Schwarzschild coordinates]\label{prop:starwaveeq} We have
 \begin{align}
 \Big| Z^I\big(\Box^{\,*}\!-
 \Box_{\,0}\big)\phi\Big|&\les \frac{M\ln{|1+r|}}{1+t+r}
 {\sum}_{|J|\leq |I|+1}\,
 \frac{|\pa
 Z^J\phi|}{1+t+r},\label{eq:waveopdiff}\\
\big|\pa^\alpha Z^I(\pa_{\mu}^*-\pa_\mu )\phi\big| &\les
\frac{M\ln{|1+r|}}{1+t+r} {\sum}_{|J|\leq|I|,\, |\beta|\leq
|\alpha|} \,\,\, |\pa^\beta\pa Z^J\phi|.\label{eq:partialstardifference}
 \end{align}
Moreover, with $N\!=\!1\!+\!\tfrac{M\ln{|1+r|}}{1+|q|}$ and
$N_*\!=\!1\!+\!\tfrac{M\ln{|1+r|}}{1+|q^*|}$, and $Z$ homogeneous,
 \begin{align}
 N_*^{-k}\!\!\!\!\!\!
 \sum_{|I|\leq k,|\alpha|\leq \ell\!\!\!\!\!\!}|\pa^\alpha \!Z^I\phi|
 &\les \sum_{|I|\leq k,|\alpha|\leq \ell\!\!\!\!\!\!} |\pa^{*\alpha}\! Z^{*I} \phi|
 \les N^k\!\!\!\!\!\!
 \sum_{|I|\leq k,|\alpha|\leq\ell\!\!\!\!\!\!}|\pa^\alpha\! Z^I\phi|,
 \label{eq:ZstarZ}\\
N_*^{-1}(1+|q|)
 &\les 1+|\,q^*|\les N(1+|\,q|).
 \end{align}
\end{prop}
Here $M\ln{|1\!+r|}\!\les\varepsilon (1\!+r)^{\varepsilon}$ if $M\!\les\!\varepsilon^2$.
Since the estimates are clearly true when $t$
is bounded we can translate the operators in time to
reduce to the case when $t\!\geq\! 1$ and
$\chi_0\!=\!\chi_0\big(\tfrac{r}{t}\big)\!=\!M{\tilde{\chi}}\big(\tfrac{r}{t}\big)$ is homogeneous.
 In the proof $\chi\big(\tfrac{x}{t}\big)$ respectively $\chi^\prime\big(\tfrac{x}{t}\big)$
 with some indices, and possibly depending on other variables,
  will denote homogeneous smooth functions of $y=x/t$ such that
 \beq\label{eq:supportconditions}
 \chi(y)=0,\quad |y|\leq 1/4,\quad \text{respectively}\quad
 \chi^\prime(y)=0,\quad \big||y|-1/2\big|\geq 1/4.
 \eq
Throughout the proof we will also use the estimate
\beq
|\pa\phi|\!\les {\sum}_{|I|\leq 1}|Z^{I}\!\phi|/(1\!+|q|)
\eq
 and its consequence using that $\chi^\prime$ is supported in a set where $q\sim t\sim r$
 \beq\label{eq:derivstivessupportedinsidetovectofields}
 \big|Z^I \big(r^{-1}\chi^\prime \pa^2\phi\big)\big|
 \les{\sum}_{|J|\leq |I|+1}|\pa Z^{J}\!\phi|/(1\!+t+r)^2.
 \eq

 We start by making reductions of the operators
 to estimate \eqref{eq:waveopdiff}.
 Using
\beq\label{eq:vectorfieldwaveoperatoridentities}
\pa_t^2-\pa_r^2=(t+r)^{-1}\big((\pa_t-\pa_r)S+\omega^i(\pa_t-\pa_r)\Omega_{0i}\big),
\qquad {\triangle}_\omega\!=\!\sum \Omega_{ij}^2,
\eq
we see that modulo terms controlled by the right of  \eqref{eq:waveopdiff} we can replace
$\Box_{\,0}$ by
\beq\label{eq:boxzerosphericalcoordinatessimplified}
\Box_{\, 0} =-\pa_t^2+
c^{-2}
\pa_r^2+{2}{r}^{-1}\pa_r+r^{-2}
{\triangle}_\omega, \qquad\text{where}\quad c=1+{\chi_0}/{r}.
\eq
Let us first collect some identities for the change of variables:
\begin{lemma}  Let
$\kappa=r^*-r=\chi_0\big(\tfrac{r}{t}\big)\ln{|r\,|}$, where $\chi_0=M\tilde\chi$. We have
\begin{equation}
\frac{\pa}{\pa r}-\frac{\pa}{\pa r^*\!}=\frac{\pa \kappa}{\pa r}\frac{\pa}{\pa r^*\!},
 \qquad
 \frac{\pa}{\pa t}-\frac{\pa}{\pa t^*\!\!}=\frac{\pa \kappa}{\pa t}\frac{\pa}{\pa r^*\!},
 \end{equation}
 where
\begin{equation*}
\frac{\pa \kappa}{\pa r}\!= \frac{\chi_0\big(\tfrac{r}{t}\big)\!}{r}
+\frac{\chi_0^\prime\big(\tfrac{r}{t}\big)\ln{|r|}}{t},
\qquad
\frac{\pa \kappa}{\pa t}\!=-\frac{\chi_0^\prime\big(\tfrac{r}{t}\big)
\tfrac{r}{t}\ln{|r|}}{t}.
\end{equation*}
\end{lemma}
The following quantities will be $\sim 1$
\beq
a\!=\!1\!+\kappa_r\!=\!1\!+{\chi_0}/{r}\!+{\chi_0^\prime\ln{r}}/{t},\quad
 b\!=\!1\!+{\kappa}/{r}\!=\!1\!+{\chi_0\ln{r}}/{r}.
\eq
For some smooth functions $f_{\!p}(b)$
\beq
\frac{1}{r^p}-\frac{1}{{r^*}^p}=\frac{\chi_0\ln{r}}{r^{p+1}}f_p(b),\qquad
\pa_{r^*}-\pa_{r}=\frac{1}{a}\frac{\chi_0+\chi_0^\prime\ln{r}\!\!}{r}\,\,\pa_r.
\eq
 It follows that to show that \eqref{eq:waveopdiff} holds for the difference of
\eqref{eq:boxstarsphericalcoordinates} and  \eqref{eq:boxzerosphericalcoordinatessimplified}
it only remains to show that it holds for the difference of the principal radial parts:
 \begin{lemma}\label{lem:radialwaveopschwarzcoord}
 With $a=1+\kappa_r=1+\chi_0/r+\chi_0^\prime\ln{r}/t$ we have
 \beq
 -\pa_{t^*}^2+\pa_{r^*}^2=p_{\,}(\pa_t,\pa_r)+a^{-1}\big(p_{\,}(\pa_t,\pa_r)\kappa\big)\pa_r,
 \eq
 where  with
  $c=1+\chi_0/r$
 \beqs
  p_{\,}(\pa_t,\pa_r)
  =-\pa_t^2\!+\frac{1\!-\!\kappa_t^2\!\!}{a^2}\,\,\pa_r^2+\frac{2\kappa_t\!}{a}\,\pa_t\pa_r
  =-\pa_t^2\!+\frac{1}{c^2}\pa_r^2+\frac{\chi_0^\prime\ln{r}\!\!}{t}
  \,q(\pa_t,\pa_r).
 \eqs
 Here
 \beq
 q(\pa_t,\pa_r)=\frac{-(c+a)+(c-a)\,r^2\!/t^2}{a^2c^2}\,\pa_r^2
 -\frac{2r\!/t}{a}\,\pa_r\pa_t.
 \eq
 \end{lemma}
 \begin{proof} We have
 \begin{multline*}
 -\pa_{t^*}^2+\pa_{r^*}^2
 =-(\pa_t-{\kappa_t}\,{a}^{-1}\pa_r)^2+({a}^{-1}\pa_r)^2\\
 =-\pa_t^2+\frac{1-\kappa_t^2\!\!}{a^2}\,\,\pa_r^2+\frac{2\kappa_t}{a}\pa_t\pa_r
 +\Big(\big(\pa_t-\frac{\kappa_t}{a}\pa_r\big)
 \big(\frac{\kappa_t}{a}\big)+\frac{1}{a}\pa_r\big(\frac{1}{a}\big)\Big))\pa_{r}\\
  =-\pa_t^2+\frac{1-\kappa_t^2\!\!}{a^2}\,\,\pa_r^2+\frac{2\kappa_t}{a}\pa_t\pa_r
 -\frac{1}{a}\Big(
 \big(-\pa_t^2+\frac{1-\kappa_t^2\!\!}{a^2}\,\,\pa_r^2+\frac{2\kappa_t}{a}\pa_t\pa_r\big)
 \kappa\Big)\pa_r.
 \end{multline*}
 \end{proof}
 It follows from the lemma that
 \beq
 (\!-\pa_{t^*}^2\!\!+\pa_{r^*}^2\!)\phi-(-\pa_t^2\!\!+c^{-2}\pa_r^2)\phi
 \!=\!\frac{\chi_0^\prime\!\ln{r}\!\!}{t}\,q(\pa_t,\pa_r)\phi
  +\frac{1}{a}\big(p(\pa_t,\pa_r)\kappa\big)\pa_r\phi.
\eq
The first term can be estimated using \eqref{eq:derivstivessupportedinsidetovectofields}
since it can be written
\beq\label{eq:theabove}
t^{-1}{\chi_0^\prime\ln{r}}\,q(\pa_t,\pa_r)\phi=
r^{-1}\, {\ln{r}}\,\chi^{\,\prime \,\alpha\beta}\big(\tfrac{x}{t},a,b,c\big)\pa_\alpha\pa_\beta\phi,
\eq
for some smooth functions $\chi^{\,\prime \,\alpha\beta}$ satisfying the second support condition in
\eqref{eq:supportconditions}.
Since any of the vector fields $Z$ applied to functions of this form
produces functions of the same form it follows from
\eqref{eq:derivstivessupportedinsidetovectofields} that \eqref{eq:theabove} can be estimated
by the right hand side of \eqref{eq:waveopdiff}. Similarly the second term can be written
\beq\label{eq:theabove}
\big(p(\pa_t,\pa_r)\kappa\big)\pa_r\phi/a\!=\!
 \big(\chi^{\,\alpha\beta\mu}\big(\!\tfrac{x}{t},\!a,\!b,\!c\big)
\pa_\alpha\pa_\beta\kappa\big)\pa_\mu \phi
\!=\!\chi^{ \,\mu}\big(\!\tfrac{x}{t},\!a,\!b,\!c\big)\pa_\mu \phi/t^2\!\!,
\eq
for some smooth $\chi^{\,\alpha\beta\mu}$ and $\chi^{ \,\alpha\beta\mu}$
satisfying the first support condition in \eqref{eq:supportconditions}.
This can be estimated by the right hand side of \eqref{eq:waveopdiff}.
We will us the following
\begin{lemma}\label{lem:changeofvariableidentities} We have
\begin{equation}
\frac{\pa}{\pa r}-\frac{\pa}{\pa r^*\!}=\frac{\pa \kappa}{\pa r}\frac{\pa}{\pa r^*\!},
 \qquad
 \frac{\pa}{\pa t}-\frac{\pa}{\pa t^*\!\!}=\frac{\pa \kappa}{\pa t}\frac{\pa}{\pa r^*\!},
 \qquad \frac{\pa x^{\!*\,i}\!\!\!}{\pa t}=\omega_i \frac{\pa \kappa}{\pa t},\!\!\!\!
 \end{equation}
 \begin{equation}
\frac{\pa x^{\!*i}\!}{\pa
x^j}\!=\delta_{\!j\,}^i\big(1\!+\!\frac{\kappa}{r}\big)\!
+\omega_i\,\omega_{\!j}\,
\big(\frac{\pa \kappa}{\pa r}-\frac{\kappa}{r}\big)
\!=\!\big(\frac{\pa \kappa}{\pa r}+1\big)\,\delta_{\!j}^i
-\!\big(\frac{\pa \kappa}{\pa r}-\frac{\kappa}{r}\big)\Pi_{\!j}^i,\!\!\!\!\!\!\!
\end{equation}
where $\Pi_{ij}\!=\delta_{ij}-\omega_i\, \omega_j$ is the projection to the
tangent space of $\bold{S}^2$. Moreover
\beq
\frac{\pa}{\pa x^i\!}-\frac{\pa}{\pa x^{\!*\,i}}
 =\frac{\kappa}{r}\frac{\pa}{\pa x^{\!*\,i}\!\!}\,
 + \big(\frac{\pa }{\pa  x^i}(\frac{\kappa}{r})\big)\, r\frac{\pa }{\partial r^{\!*}}
=\frac{\pa \kappa\!\!}{\pa r}\frac{\pa}{\pa x^{\!*\,i}\!\!}\,
  -\,\big(\frac{\pa \kappa}{\pa r}
  -\frac{\kappa}{r}\big)\Pi_i^j\!\frac{\pa }{\partial x^{\!*j}},
\eq
where $\Pi_i^j\pa^*_j=r^{*-1}\omega^j\Omega^*_{ji}$, and
$\det{(\pa x^*\!/\pa x)}=(1\!+\kappa/r)^2(1\!+\kappa_r)$. Furthermore
 \beq
 \Omega_{ij}\!=\!\Omega^*_{ij},
 \qquad\!\! S\!=\!S^*\!\!+S(\tfrac{\kappa}{r})r\pa_{r^*},
 \qquad\!\!
\Omega_{0i}\!=\!\Omega_{0i}^*\!+
 \Omega_{0i}(\tfrac{\kappa}{r}) r\pa_{r^{*\!}}+\tfrac{\kappa}{r}(t\pa^*_i\!-x_i\pa_{t^*}\!).
 \eq
 Moreover, with $a=\pa r^*\!/\pa r=1\!+\pa\kappa/\pa r$ and $b=r^*\!/r=1+\kappa/r$
 \beq\label{eq:theinversederivatives}
 \frac{\pa}{\pa x^{\!*\,i}\!}\!
 =\!\frac{1}{a}\frac{\pa}{\pa x^{i}\!}
  +\!\frac{1}{ab\!}\big(\frac{\pa \kappa\!}{\pa r\!}\,
  -\frac{\kappa}{r}\big)\frac{\omega^j\Omega_{ji}\!\!}{r},\qquad
  \frac{\pa}{\pa t^*\!\!}\!= \!\frac{\pa}{\pa t}
  -\!\frac{1}{a}\frac{\pa \kappa}{\pa t}\frac{\pa}{\pa r\!},\qquad
  \frac{\pa}{\pa r^*\!\!}=\!\frac{1}{a}\frac{\pa}{\pa r\!}.
 \eq
 \end{lemma}

 We start by proving \eqref{eq:ZstarZ}. First we will show that
 \beq\label{eq:ZstarZstepone}
|\pa^\alpha \!Z^I\phi|
\les {\sum}_{|J|+|\gamma|\leq |I|,\,\,|\beta|\leq |\alpha|}
\big(1\!+M\ln{|1\!+r|}\big)^{|\gamma|}\, |\pa^{*\beta}\pa^{*\gamma}\! Z^{*J} \phi|.
 \eq
 The first inequality in \eqref{eq:ZstarZ} follows from this using
$|\pa^*\phi|\!\les \!\sum_{|I|\leq 1}|Z^{*I}\!\phi|/(1\!+|q^*|)$.\break
We will use induction to prove
that $Z^I $ is a sum of terms of the form $\kappa_{J,\gamma}\,\pa^{*\gamma}\!Z^{*J}\!\!$,
with $|\gamma|\!+\!|J|\!=\!|I|$, where
$\kappa_{J,\gamma}\!=\!\chi_{J,\gamma,0}+\dots+\chi_{J,\gamma,|\gamma|}(\ln{r})^{|\gamma|}\!$
and $\chi_{J,j}$ are homogeneous of degree $0$. Here $\chi_{I,0}\!\!=\!1$ and $\chi_{J,j}$ are
 supported in $r\!\geq\! t/2$ for $|J|\!<\!|I|$. By the Lemma
 \ref{lem:changeofvariableidentities}
 we have for some  $\chi_{Z,j}^\nu$ homogeneous of degree $0$
\beqs
Z\big(\kappa_{J,\gamma}\,\pa^{*\gamma}\!Z^{*J}\big)
=(Z\kappa_{J,\gamma})\,\pa^{*\gamma}\!Z^{*J}
+\kappa_{J,\gamma}\,
\big(Z^*+(\chi_{Z,0}^\nu+\chi_{Z,1}^\nu\ln{r})\pa^*_\nu\big)\pa^{*\gamma}\!Z^{*J}.
\eqs
Here $Z\kappa_{J,\gamma}$ is of  the form
$\chi_0\!+\dots \!+\!\chi_{|\gamma|}(\ln{r})^{|\gamma|}$
and $\kappa_{J,\gamma}(\chi_{Z,0}^\mu\!+\!\chi_{Z,1}^\mu\ln{r})$  is of the form
$\chi_0\!+\!\dots \!+\!\chi_{|\gamma|+\!1}(\ln{r})^{|\gamma|+\!1}\!$. This proves the assertion and
 \eqref{eq:ZstarZstepone} for $|\alpha|\!\!=\!0$. To prove it for $|\alpha|>0$ we claim that
 $\pa^\alpha\!$ applied to $\kappa_k \psi$, where
 $\kappa_k$ is of the form $\chi_0\!+\dots +\!\chi_{k}(\ln{r})^{k}\!\!$,
 is a sum of terms of the form
$\kappa_{k+\ell}\,t^{-\ell}\pa^{*\beta}\psi$
with $|\beta|+\ell=|\alpha|$.
In fact by Lemma \ref{lem:changeofvariableidentities} we have for some
$\chi_{\mu,j}^\nu$ homogeneous of degree $0$
\beqs
\pa_\mu \big(\kappa_{k+\ell}\,t^{-\ell}\pa^{*\beta}\big)\!
=\big(\pa_\mu (\kappa_{k+\ell}\,t^{-\ell})\big)\pa^{*\beta}\!\!
+ \kappa_{k+\ell}\,t^{-\ell}\big(\pa^*_\mu\!
+(\chi_{\mu,0}^\nu\!+\chi_{\mu,1}^\nu\ln{r})\,t^{-1}\pa^*_\nu\big)\pa^{*\beta}\!.
\eqs
Here $\pa_\mu (\kappa_{k+\ell}\,t^{-\ell})$
and
$\kappa_{k+\ell}\,t^{-\ell}(\chi_{\mu,0}^\nu+\chi_{\mu,1}^\nu\ln{r})t^{-1}\!$
are both of the form\break
$(\chi_0\!+\!\dots\! \!+\!\chi_{k+\ell+1}(\ln{r})^{k+\ell+1})\, t^{-\ell-1}\!\!$,
which proves the assertion and \eqref{eq:ZstarZstepone} follows.

We will now prove the second inequality in \eqref{eq:ZstarZ} which would
follow from
 \beq\label{eq:ZstarZsteptwo}
|\pa^{*\alpha} \! Z^{*I}\phi|
\les {\sum}_{|J|+|\gamma|\leq |I|,\,\,|\beta|\leq |\alpha|}
\big(1\!+M\ln{|1\!+r|}\big)^{|\gamma|}\, |\pa^{\beta}\pa^{\gamma}\! Z^J \phi|.
 \eq
The proof uses the argument above with
the inverse identities \eqref{eq:theinversederivatives} that give
\begin{align}
Z^*\!&=Z+\big(f_{Z,0}(a,b)\chi_{Z,0}^\nu+f_{Z,1}(a,b)\chi_{Z,1}^\nu\ln{r}\big)\pa_\nu,
\label{eq:thestarvectorfielddifference}\\
\pa^*_\mu\!
&=\pa_\mu+\big(f_{\mu,0}(a,b)\chi_{\mu,0}^\nu\!
+f_{\mu,1}(a,b)\chi_{\mu,1}^\nu\ln{r}\big)\,t^{-1}\pa_\nu,\label{eq:thestarderivativedifference}
\end{align}
where $\chi_{Z,j}^\nu$ and $\chi_{\mu,j}^\nu$ are homogenous of degree $0$,
$f_{Z,j}(a,b)$ and $f_{\mu,j}(a,b)$ are smooth functions when $a\!=\!1\!+\pa\kappa/\pa r\!>\!0$
and $b\!=\!1\!+\kappa\!>\!0$. The only difference is when derivatives fall on $a$ or
$b$ that just produces lower order terms.
\eqref{eq:partialstardifference}
follows directly from applying
vector fields and derivatives to  \eqref{eq:thestarvectorfielddifference}.

\section{Decay estimates for the inhomogeneous wave equation}
\begin{lemma}\label{lem:radialwaveeqderdecay}
Let $\Box_r\phi\!=\!r^{-1}(\pa_r^2\!-\pa_t^2)(r\phi)$. Then with $q\!=\!r\!-t$,
we have for $r\!>\!t/2$:
 \begin{equation*} (1\!+t\!+r)|\pa\phi(t,r\omega)|
\les
 \!\max_{\xi=t+r\!,  \,3|q|} \!\sum_{|I|\leq 1\!\!\!} |Z^I\phi
 \big(\tfrac{\xi-q}{2},\!\tfrac{\xi+q}{2}\omega\big)|
 +\!\!\int_{3|q|}^{t+r}\!\!\!\!
  |r\Box_r^{\, }
   \phi(\tfrac{\xi-q}{2},\!\tfrac{\!\xi+q}{2}\omega\big)| d\xi.
 \end{equation*}
\end{lemma}
\begin{proof}
If we  integrate
$
\big(\pa_{t}+\pa_{r}\big)
 \big(\pa_{r}-\pa_{t}\big)(r\phi)=r\,\Box_r^{\,} \phi
 $
in the $t\!+\!r$ direction from the intersection with the line $r\!=\!t/2$
($q\!<\!0$) or the line $r\!=\!2t$ ($q\!>\!0$) we get
 \begin{equation*}
 |(\pa_{r}-\pa_{t})(r\phi)(t,r\omega)|
 \leq
 |(\pa_{r}-\pa_{t})(r\phi)
 \big(\tfrac{3|q|-q}{2},\!\tfrac{3|q|+q}{2}\omega)|\,
 +\!\int_{3|q|}^{t+r}\!\!\!\!
  |\,r\Box_r^{\,}
   \phi(\tfrac{\xi-q}{2},\!\tfrac{\xi+q}{2}\omega\big)|\, d\xi.
 \end{equation*}
 Now
 \beqs
 (1+t+r)|\,\pa \phi(t,r\omega)|
 \les |(\pa_{r}-\pa_{t})(r\phi)(t,r\omega)|
 +{\sum}_{|I|\leq 1}|Z^I\phi(t,r\omega)|,
 \eqs
 and
 \beqs
|(\pa_{r}-\pa_{t})(r \phi)
 \big(\tfrac{3|q|-q}{2},\tfrac{3|q|+q}{2}\omega\big)|
 \les {\sum}_{|I|\leq 1} |Z^I\phi
 \big(\tfrac{3|q|-q}{2},\tfrac{3|q|+q}{2}\omega\big)|.
 \eqs
 \end{proof}

\begin{lemma}\label{lem:inhomwaveeqdecay} If $-\Box \phi=F$, with
vanishing data, where
\beq \label{eq:goodinhomdecay} |F|\leq
\frac{C}{(1+r)(1+t+r)(1+|\,t-r|)^{1+\delta}},\qquad \delta>0
 \eq
 then with $ \langle \, q\,\rangle=\sqrt{1+q^2}$
 \begin{equation}\label{eq:logest1}
 |\phi|\leq\frac{C S^0(t,r)}{(1+t+r)\,(1+q_+)^\delta},\qquad
 \text{where}\quad S^0(t,r)
 =\frac{t}{r}\ln{\Big(\frac{\langle \,t+r\,\rangle}{\langle\,t-r\,\rangle}\Big)}.
 \end{equation}
 Here $q_+=r-t$, when $r\geq 0$ and $q_+=0$, when $r\leq t$.
 On the other hand if
 \begin{equation*}\label{eq:goodinhomdecay2} |F|\leq
\frac{C}{(1+r)(1+t+r)^{1+\mu}(1+|\,t-r|)^{1-\mu}(1+q_+)^{\delta_+}
(1+q_-)^{\delta_-}},
 \end{equation*}
 with $0< \delta_+<\mu,\,\, \, 0\leq
\delta_-\leq \delta_+,$ then
 \beq\label{eq:logest2}
 |\phi|\leq\frac{C}{(1+t+r)(1+q_+)^{\delta_+}(1+q_-)^{\delta_-}}.
 \eq
\end{lemma}
\begin{proof} Let $\overline{F}(t,r)\!=\!\sup_{\omega\in S^2} |F(t,r\omega)|$
and let $F_0\!=\!F H$ where $H\!=\!1$,
when $t\!>\!0$ and $H\!=\!0$, when $t\!<\!0$.
 Since $|F_0|\leq \overline{F}_0$ it follows from
 the positivity of the fundamental solution that $|\phi|\leq
 |\overline{\phi}|$ where $\overline{\phi}$ is the solution of
 $-\Box\overline{\phi}=\overline{F}_0$ with vanishing initial data.
 Since the wave operator is invariant under rotations
 it follows that $\overline{\phi}$ is independent of the angular
 variables so
 $
 (\pa_t-\pa_r)(\pa_t+\pa_r)(r\overline{\phi}(t,r))=r\overline{F}_0.
 $
 If we now introduce new variables $\xi=t+r$ and $\eta=t-r$ and
 integrate over the region $R=\{(\xi,\eta);\,
 -\infty\leq \eta\leq t-r,\, \, t-r\leq \xi \leq t+r\}$
 using that $r\overline{\phi}(t,r)$ vanishes when $\eta=-\infty$ and
 when $r=0$, i.e. $\xi=\eta=t-r$ we obtain
 \beqs
 r\overline{\phi}(t,r)=4\int_{t-r}^{t+r} \int_{-\infty}^{t-r}
 \rho \overline{F}_0(s,\rho) H(s) \, d\eta d\xi,
 \qquad s=\frac{\xi+\eta}{2},\quad \rho=\frac{\xi-\eta}{2}.
 \eqs
 In the first case we have
 \beqs
 r\overline{\phi}(t,r)\leq 4\int_{t-r}^{t+r} \int_{-\xi}^{t-r}
 \frac{H(\xi+\eta) }{(1+|\xi|)^{}(1+|\eta|)^{1+\delta}}
 \, d\eta d\xi.
 \eqs
 If $t\!>\!r$ \eqref{eq:logest1} follows from integrating this,
 since
  $
  \frac{1}{r}\log{\big(\frac{1+t+r}{1+t-r}\big)}
 \!\leq \!\frac{C}{1+t+r} S^0(t,r).
  $
  If $r>t$ then we integrate first in the $\xi$ direction
\begin{multline*}
 r\overline{\phi}(t,r)\les \int_{\!-(t+r)}^{-(r-t)}\!\!
 \int_{|\eta|}^{t+r}
 \!\!\!\!\!\frac{ d\xi d\eta }{(1+|\,\xi|)(1+|\eta|)^{1+\delta}}
\les \int_{\!-(t+r)}^{-(r-t)}
 \frac{ \log{\big|\frac{1+t+r}{1+|\eta|}\big|} \, d\eta }
 {(1+|\eta|)^{1+\delta}}\\
 \les \frac{C}{(1+t+r)^\delta}
 \int_{\frac{1+r-t}{1+r+t}}^1\!\!\!\!\!
 \frac{\ln{\big|\frac{1}{s}\big|}\, ds}
 {\,\,\,s^{1+\delta}\!\!\!},
 \end{multline*}
  and \eqref{eq:logest1} for $r>t$ follows from this.
To prove \eqref{eq:logest2} we must estimate
 \begin{equation*}
 r\overline{\phi}(t,r)\leq 4\int_{t-r}^{t+r} \int_{-\xi}^{t-r}
 \frac{ H(\xi+\eta) \,\, d\eta \,d\xi}{(1+|\xi|)^{1+\mu}
 (1+|\eta|)^{1-\mu}(1+\!\eta_-)^{\delta_+}
 (1+\!\eta_+)^{\delta_-}}.
 \end{equation*}
 If $r>t$ then we integrate first in the $\xi$ direction
\begin{multline*}
 r\overline{\phi}(t,r)\les \int_{\!-(t+r)}^{-(r-t)}\!\!\int_{|\eta|}^{t+r}
 \!\!\!\!\!\frac{\ d\xi d\eta }{(1+|\xi|)^{1+\mu}
 (1+\!|\,\eta|)^{1+\delta_+ -\mu}}
\les \int_{\!-(t+r)}^{-(r-t)}
 \frac{\, d\eta }{(1+\!|\,\eta|)^{1+\delta_+}}
\end{multline*}
which is $\les (1+\!|\,t-r|)^{-\delta_+}$ so
\eqref{eq:logest2} for $r>t$ follows.
 If $t>r$ it follows from integrating in the $\eta$ direction
  that if
 $1+\delta_+-\mu<1$ we have with ${\delta}=\min(\delta_-,\,\delta_+)$
 \begin{multline*}
 r\overline{\phi}(t,r)\leq\! 4\!\int_{t-r}^{t+r}\!\!\!
 \frac{1 }{(1+|\xi|)^{1+\delta_+}}
 +\frac{(1+\!|\,t-r|)^{\mu-\delta_-} \!\!\!}{(1+|\xi|)^{1+\mu}}
 \, d\xi
  \leq \frac{C r}{(1+t+r)(1+|t-r|)^{{\delta}}}.
 \end{multline*}
\end{proof}

\begin{lemma}\label{lem:homwaveeqdecay} If $w$ is the solution of
 \begin{equation*}
 -\Box w=0,\quad \quad w\big|_{t=0}=w_0,\quad \partial_t w\big|_{t=0}=w_1
 \end{equation*}
 then for any $0<\gamma<1$;
 \begin{multline}\label{eq:homoest}
 (1+t+r)(1+|\,r-t|)^\gamma |w(t,x)|\\
 \les {\sup}_x \big( (1+|x|)^{2+\gamma}
 ( |w_1(x)|+|\,\partial w_0(x)|) +(1+ |x|)^{1+\gamma}|w_0(x)|\big).
 \end{multline}
\end{lemma}
\begin{proof} The proof is an immediate consequence of
Kirchoff's formula
 $$
w(t,x)=t\!\int {\big(w_1(x+t{\omega})+\langle
w^{\prime}_0(x+t{\omega}),{\omega}\rangle \big)\,\frac{dS({\omega})}{4\pi}} +
 \int {w_0(x+t{\omega})\,\frac{dS({\omega})}{4\pi}},
 $$
 where $dS(\omega)$ is the measure on $\bold{S}^2$.
 If $x\!=\!r\bold{e}_1$, where $\bold{e}_1\!=\!(1,0,0)$ then
for $k\!=\!1,2$
 \begin{multline*}
 \int\!\! \frac{dS(\omega)/4\pi}{1\!+|r\bold{e}_1\!+t\omega|^{k+\gamma}}
 =\int_{\!-1}^1\frac{d\omega_1/2}{1\!+\big(
 (r\!-t\omega_1)^2\!+t^2(1\!-\omega_1^2)\big)^{(k+\gamma)/2}}\\
 =\int_0^2\!\! \frac{d s/2}{1\!+\big(
 (r\!-t)^2\!+2rt s\big)^{(k+\gamma)/2}}.
 \end{multline*}
 \eqref{eq:homoest} follows directly if $|r\!-t|\! \geq \! t/2$.
If $t/2\!<\!r\!<\!2t$ we change variables $\tau\!=\!rt s$.
  If $k\!=\!2$ it can be bounded by
  $(rt)^{-1}(1+|r\!-t|)^{-\gamma}$ and if $k\!=\!1$ by
  $(rt)^{-1}(1\!+rt)^{(1-\gamma)/2}$.
 \end{proof}

\section{Sharp decay  in asymptotic Schwarzschild coordinates.}\label{sec:Sch}
Using Proposition \ref{prop:starwaveeq} to change to asymptotic coordinates Proposition \ref{prop:approxwaveequation} become
\begin{prop}[Asymptotic Einstein in Schwarzschild coordinates]
\label{prop:approxwaveequationschwarzcoord}
  Let
 \beqs
 P_{\mu\nu}^*=\chi\big(\tfrac{\langle \,r^*-t\,\rangle }{t+r^*}\big)
 P_\mathcal{N\,}
 (\pa_\mu^*
 {h},\pa_\nu^*{h}),\quad\text{or}\quad
 P_{\mu\nu}^*=\chi\big(\tfrac{\langle \,r^*-t\,\rangle }{t+r^*}\big)
 L_{\mu}L_\nu P\,(\pa_q^* h,\pa_q^* h),
 \eqs
 where $\chi\!\in\!  C_0^\infty\!$ satisfies
 $\chi({q})\!=\!0$, when
$|{q}|\!\geq \!3/4$, $\chi({q})\!=\!1$, when $|{q}|\!\leq\! 1\!/2$. We have
\begin{equation}\label{eq:AsymptoticEinsteinAsymptoticSchwarzschild}
 \big|{Z^*}{}^I\big[\Box^{*}
h_{\mu\nu}-P_{\mu\nu}^*\big]\big| \les
\frac{\varepsilon^2(1\!+|q^*|)^{-1}\!\!} {(1_{\!}\!+t\!+r^*)^{3-C\varepsilon}\!\!}\,
+\frac{\varepsilon^2(1\!+|q^*|)^{-2}\!\!}
{(1_{\!}\!+t\!+r^*)^{2+\gamma-C\varepsilon}\!\!}\, , \quad |I|\!\leq\! N\!-5.
\end{equation}
\end{prop}
and Propositions
 \ref{prop:weakdecay} and \ref{prop:wavecoorddecay} become
\begin{lemma}[Weak decay]\label{lem:weakdecaystar} For $|I|\leq N-4$ we have
 \beq \label{eq:decay1star} (1+\!|\,q^*|) \,| \pa^*
Z^{*I} h^{1*}| +|Z^{*I} h^{1*}|\les
{\varepsilon(1+q_+^*)^{-\gamma}} {(1+t+\!|\,q^*|)^{-1+C\varepsilon}},
  \eq
  where we have written $h\!=\!h^{0*}\!+h^{1*}\!\!$, where
$h^{0*}_{\a\b}\!=\!\tilde{\chi}\big(\tfrac{r^*}{1+t}\big)\frac{M}{r^*}\de_{\a\b}$.
When $r^*\!\geq \!t_{\!}/8$
 \begin{equation}
 \label{eq:decay1starwavecoord}
 | \pa^* \!Z^{*I}\! {h}_{LT}^{1*}|
 \les \frac {\varepsilon(1+q_+^*)^{-\gamma}}
{(1\!+t\!+|q^*|)^{2-C\varepsilon}},\qquad |Z^{*I}\!
{h}_{LT}^{1*}| \les \frac
{\varepsilon(1+|_{\,}q_-^*|)^{\gamma}}
{(1\!+t\!+|q^*|)^{1+\gamma-C\varepsilon}}.
 \end{equation}
Moreover \eqref{eq:decay1starwavecoord} also hold for ${h}_{LT}^{1*}$ replaced by $\delta^{AB}{h}_{AB}^{1*}$.
\end{lemma}
 In this section we will prove the following decay estimates:
\begin{prop}[Sharp decay]\label{prop:sharpmetricdecay} For $|I|\leq N-6$ with
 $\gamma^\prime=\gamma-C\varepsilon$ we have
 \beqs\label{eq:metricdecaysharp}
 |{Z^*}{}^I {h}^{1*}|\les \frac{\varepsilon^2
 S^0(t,r^*)}{(1+t+r^*)(1+q_+^*)^{1-C\varepsilon}}
 +\frac{\varepsilon}{1+t+r^*}   \frac{1}{(1+|\,q^*|)^{\gamma^\prime}},
 \eqs
 where
 \beqs S^0(t,r^*)
 =\frac{t}{r^*\!\!}\,\ln{\!\Big(\frac{\langle \,t\!+r^*\,\rangle}{\langle \,t\!-r^*\rangle}\Big)}
 \les
\frac{1}{\varepsilon}
\Big(\frac{\langle\, t\!+r^* \rangle}{\langle \,t\!-r^*\rangle}\Big)^{\varepsilon}\!\!,
\quad \langle q\rangle\!=\sqrt{1\!+q^2}.
\label{eq:loggrowth}
 \eqs
 For $r^*\!\!\geq \!t_{\!}/2 \!$ we have
\begin{align}
\label{eq:metricdecaysharptan}
 |Z^{*I\!} {h}^{1*}_{TU}|
&\les \frac{\varepsilon}{(1+t+r^*)(1+q_+^*)^{\gamma^\prime}},\\
 |\pa Z^{*I\!} {h}^{1*}_{LT}|\!+|\pa Z^{*I\!} \delta^{AB\!}{h}^{1*}_{AB}|
 &\les \frac{\varepsilon}{(1+t+r^*)^{2-\varepsilon}
 (1+|\,q^*|)^{\varepsilon}(1+q_+^*)^{\gamma^\prime}},\label{eq:sharpwsvecoordder}\\
 |Z^{*I\!} {h}^{1*}_{LT}|\!+|Z^{*I\!} \delta^{AB\!} h^{1*}_{AB}|
 &\les \frac{\varepsilon}{(1_{\!}\!+t\!+r^*)^{1+\gamma^\prime}\!\!}
 +\!\frac{\varepsilon}{1_{\!}\!+t\!+r^*\!\!}\,
 \Big(\frac{1\!+q_-^*\!}{1_{\!}\!+t\!+r^*\!}\Big)^{\!\!1-\varepsilon}\!\!\!.\!\!\!
\label{eq:sharpwsvecoordfunc}
 \end{align}
Moreover \eqref{eq:sharpwsvecoordder}-\eqref{eq:sharpwsvecoordfunc} hold for $Z^{*I\!}(h_{UV}^1\!)$ replaced by the Lie derivatives $(\mathcal{L}_{Z^{*}}^I h_{1})_{UV}\!$.
 \end{prop}

 \begin{remark} Using \eqref{eq:wavecoordinateframeRdivergenceL}, \eqref{eq:LambdaL} and \eqref{eq:extraLderest1} one can show \eqref{eq:sharpwsvecoordder} for $h^{1*}_{LL}$ with
 $\varepsilon\!=\!0$ in the exponents. 
 \end{remark}
We will successively improve the estimates starting with those in
Lemma \ref{lem:weakdecaystar}. We now want to contract with a null frame, which does not quite commute with the wave equation.
 The null frame only depends on the angular variables so it commutes
 with the radial part of the wave operator but not the angular:
 \beqs
 \Box^{\,*}=\Box^{\,*}_{\,r}
 +{r^*}^{-2}\triangle_\omega,\quad\text{where}\quad
 \Box_{\,r}^{\, *}\phi
 ={r^*}^{-1}\big(\pa_{t^*}+\pa_{r^*}\big)
 \big(\pa_{r^*}-\pa_{t^*}\big)(r^*\phi).
 \eqs
 Since $\triangle_\omega=\sum{\Omega_{ij}^*}^2$ and $
 |Z^* U|\leq C$, for $U\in \{ A,B, L,\underline{L}\}$,
it follows that
 \beqs
 \big|Z^{*I}\big( \Box^{\, *} {h}_{UV}
 -U^\mu V^\nu\Box^{\, *} h_{\mu\nu}\big)\big|
 \leq{r^{*}}^{-2}\,\, {\sum}_{|J|\leq |I|+1} \,\,|Z^{*J} h|.
 \eqs
 Since $[Z^*\!,\Box^{\,*}]$ is either $0$ or $2\,\Box^{*}\!$
 we  have $|\Box^{\, *} Z^{*I} {h}_{TU}|\!\les
 {\sum}_{|J|\leq |I|} |Z^{*J} \Box^{\,*}{h}_{TU}|$ so
 \begin{equation}\label{eq:boxhtangential}
 \big|\Box^{\, *} Z^{*I} {h}_{TU}\big|\les
  {{r^{*}}^{-2}}\, {\sum}_{|J|\leq |I|  +1}\, |Z^{*J} \!h|
  + {\sum}_{|J|\leq |I|}\,
 \big|Z^{*J}\!\big(T^\mu U^\nu\Box^{\, *} h_{\mu\nu}\big)\big|,
 \end{equation}
 where since  $T^\mu L_\mu=0$ for $T\in\{L,A,B\}$
 \beq\label{eq:tangentialboxh}
 \big|Z^{*J\!}\big( T^\mu U^\nu\Box^{\, *} h_{\mu\nu}\big)\big|
 \leq
 \big|Z^{*J\!}\big(T^\mu U^\nu\big[\Box^{\,*}  h_{\mu\nu}\!
-{L_\mu L_\nu}P_\mathcal{N\,}(\pa_{q^*} h,\pa_{q^*} h)/4\big]\big)\!\big|.
 \eq
 Moreover
\begin{equation}\label{eq:radialboxhtangential}
 \big|\Box_r^* Z^{*I} {h}_{TU}\big|\les
  {r^{*}}^{-2} \,\,{\sum}_{|J|\leq |I|  +2}\,\,|Z^{*J} h|
 +\big|\,\Box^{\, *} Z^{*I} {h}_{TU}\big|.
 \end{equation}
It follows from using the estimates
\eqref{eq:AsymptoticEinsteinAsymptoticSchwarzschild} in
\eqref{eq:tangentialboxh} and
\eqref{eq:decay1star} in \eqref{eq:boxhtangential}   and  \eqref{eq:radialboxhtangential}:
 \begin{lemma}\label{lem:approxwaveequation2} For $|I|\leq N-6$ we have
 for $r^*\geq t/2$
 \begin{align*}
 \big|\,\Box^{\,*}_r Z^{*I} {h}_{TU}\big|
 +\big|\,\Box^{\,*}Z^{*I} {h}_{TU}\big|
&\les \frac{\varepsilon(1+q_+)^{-C\varepsilon}} {{r^*}^2(1+t+\!|\,q^*|)^{1-C\varepsilon}}
+\frac{\varepsilon^2(1+\!|\,q^*|)^{-2+\gamma}}
{(1+t+\!|\,q^*|)^{2+\gamma-C\varepsilon}},\\
 \big|\,\Box^{\,*}_r Z^{*I} {h}_{TU}^{1*}\big|
 +\big|\,\Box^{\,*}Z^{*I} {h}_{TU}^{1*}\big|
&\les \frac{\varepsilon(1+q_+^*)^{-\gamma}}
{{r^*}^2(1+t+\!|\,q^*|)^{1-C\varepsilon}}
+\frac{\varepsilon^2(1+\!|\,q^*|)^{-2+\gamma}}
{(1+t+\!|\,q^*|)^{2+\gamma-C\varepsilon}}.
\end{align*}
\end{lemma}
The second estimate follows since $\Box^* h_{\mu\nu}\!\!=\!\Box^* h^{1*}_{\mu\nu}$, when $r^*\!\geq \!{t}/{2}$.
Using these estimates we obtain the following two lemmas:
 \begin{lemma}\label{lem:tancomp} For $|I|\leq N-6$ and $r^*\geq t/8$ we have
 \begin{align}
(1+t+r^*)|\,\pa^* Z^{*I} {h}_{TU}^{1*}(t,r^*\omega)|&\les
{\varepsilon(1\!+q_+^*)^{-\gamma}}{(1\!+|q^*|)^{-1+C\varepsilon}},
\label{eq:derhTUest}\\
(1+t+r^*)|Z^{*I} {h}_{TU}^{1*}(t,r^*\omega)|&\les
\varepsilon(1+q_+^*)^{C\varepsilon-\gamma}
+\big( (1+q_-^*)^{C\varepsilon} -1\big).\label{eq:hTUest}
\end{align}
 The last estimate for $q^*<0$ can be replaced by $C_\mu\varepsilon
 (1+|q^*|)^\mu$, for any $\mu>0$.
 The same estimates hold for $h$ in place of $h^{1*}$ if $\gamma$ is
 replaced by
 $C\varepsilon$.
 \end{lemma}
 \begin{proof}
 If we apply  Lemma \ref{lem:radialwaveeqderdecay}
 to $Z^{*I\! }{h}_{TU}^{1*}\!$ using Lemmas \ref{lem:weakdecaystar},
 \ref{lem:approxwaveequation2}
 we get for $r^*\!\!\geq\! t/2$
\begin{multline*}
 (1+t+r^*)|\pa^* Z^{*I} {h}_{TU}^{1*}(t,r^*\omega)|\\
 \les
\frac{\varepsilon(1+q_+^*)^{-\gamma}}{(1+|q^*|)^{1-C\varepsilon}}
+\int_{3\,| q^*|}^{t+r^*}\!\!\!\!
\Big(\frac{\varepsilon(1+q_+^*)^{-\gamma}}{(1+\xi)^{2-C\varepsilon}}
 +\varepsilon^2\frac{(1+|q^*|)^{-2+\gamma}}
 {(1+\xi)^{1-C\varepsilon+\gamma}}
 \Big) d\xi,
 \end{multline*}
 which proves \eqref{eq:derhTUest} for $r^*\!\!\geq\! t/2$ but
 for $t/8\!\leq\! r^*\!\!\leq \!t/2$ it follows from Lemma \ref{lem:weakdecaystar}.
 \eqref{eq:hTUest} follows by
 integrating \eqref{eq:derhTUest} in the $t\!-r^*\!$ direction from the initial
 surface.
 \end{proof}

  Let
 \beq\label{eq:honee}
{h}^{1\,e}=h-{h}^{0e},\qquad\text{where}\quad
{h}^{0e}_{\mu\nu}=\delta_{\mu\nu} M\chi^e(r^*-t)/r^*,
 \eq
where $\chi^e(s)=1$, when $s\geq 2$ and $\chi^e(s)=0$, when $s\leq 1$.
Then $\Box^* h^{0e}=0$.

\begin{lemma} \label{lem:h1edecay} For $|I|\leq N-5$ have
 \beq\label{eq:newestimate}
 |Z^{*I} {h}^{1e}|\les \frac{\varepsilon^2
 S^0(t,r^*)}{(1+t+r^*)(1+q_+^*)^{1-C\varepsilon}}
 +\frac{\varepsilon}{1+t+r^*}
 \frac{1}{(1+|\,q^*|)^{\gamma-C\varepsilon}}.
 \eq
\end{lemma}
 \begin{proof} For $|I|\leq N-5$ have
\begin{multline*}
 \big|\Box^{\, *} Z^{*I} h_{\mu\nu}\big|
 \leq  {\sum}_{|J|\leq |I|} \,\big|Z^{*J} \chi\big(\tfrac{\langle \,r-t\,\rangle }{t+r}\big)
 P_\mathcal{N\,}\big(\pa^{*}_\mu {h},
 \pa^*_\nu{h}\big)\big|\\
 +{\sum}_{|J|\leq |I|}\,
 \Big|Z^{*J} \chi\big(\tfrac{\langle \,r-t\,\rangle }{t+r}\big)
 \Big[\Box^{\,*} \, h_{\mu\nu}
-P_\mathcal{N\,}\big(\pa^{*}_\mu {h},
 \pa^*_\nu{h}\big)\Big]\Big|,
 \end{multline*}
 where the second term is estimated by Proposition \ref{prop:approxwaveequationschwarzcoord}.
Using \eqref{eq:PN}, Lemma \ref{lem:tancomp} and Lemma \ref{lem:weakdecaystar} the first is estimated by
 \begin{multline*}
\sum_{|J|+|K|\leq |I|\!\!\!\!\!\!\!\!\!\!\!\!\!\!} \big|
 P_\mathcal{N\,}\big(\pa^{*}Z^{*J}{h},
 \pa^* Z^{*K}{h}\big)\big|
 \leq \sum_{|J|+|K|\leq |I|}\,\,
 \sum_{S,T\in{\,\mathcal{T}}\!\!, \,\,\, U,V\in\mathcal{N}
 \!\!\!\!\!\!\!\!\!\!\!\!\!\!\!\!\!\!\!\!\!\!\!\!\!}
 |\,\pa^* Z^{*J} {h}_{TU}| \,
|\pa^* Z^{*K} {h}_{SV}|\\
+\sum_{|J|+|K|\leq |I|}\,\,
 \sum_{S,T\in{\,\mathcal{T}}\!\!, \,\,\, U,V\in\mathcal{N}
 \!\!\!\!\!\!\!\!\!\!\!\!\!\!\!\!\!\!\!\!\!\!\!\!\!}
 |\pa^* Z^{*J} {h}_{LL}| \,
|\pa^* Z^{*K} {h}_{\underline{L}\underline{L}}|
\les\frac{\varepsilon^2}{(1+t+r^*)^2(1+|\,q^*|)^{2-C\varepsilon}},
 \end{multline*}
 when $r^*\!\geq t/8$.
 Hence by Proposition \ref{prop:approxwaveequationschwarzcoord}
 \beqs
 \big|\Box^{\, *} Z^{*I} h_{\mu\nu}\big|
\les{\varepsilon^2}{(1+t+r^*)^{-2}(1+|\,q^*|)^{-2+C\varepsilon}}.
 \eqs
 The same estimate holds for $h^{1e}\!=h-h^{0e}$, since
 $\Box^{\, *} h^{0e}\!=0$.
 We now write $h^{1e}_{\mu\nu}\!=\phi\!+w$, where $\Box\, w\!=\!0$ with data
 $(w,\partial_t w)\big|_{t=0}=
 (h^{1e}_{\mu\nu},\partial_t h^{1e}_{\mu\nu})\big|_{t=0}$,
 and apply Lemma \ref{lem:inhomwaveeqdecay}
 to $\phi$
 and Lemma \ref{lem:homwaveeqdecay} to $w$, using
 \eqref{eq:decay1star} to estimate the initial conditions.
\end{proof}
Since $\Box^* h_{\mu\nu}\!\!=\!\Box^* h^{1e}_{\mu\nu}$,
it follows from \eqref{eq:boxhtangential}-\eqref{eq:radialboxhtangential}, \eqref{eq:newestimate} and Proposition
\ref{prop:approxwaveequationschwarzcoord}
 \begin{lemma}\label{lem:approxwaveequation3} For $|I|\leq N-6$ and $r^*\geq t/8$
 we have
 \begin{multline}\label{eq:approxwaveequation3}
 \big|\,\Box^{\,*}_r Z^{*I} {h}^{1e}_{TU}\big|
 +\big|\,\Box^{\,*}Z^{*I} {h}^{1e}_{TU}\big|\\
\les \frac{\varepsilon^2 S^0(t,r^*)(1\!+q_+^*)^{-1+C\varepsilon}\!\!\!\!\!\!}
{{r^*}^2(1+t+r^*)}
+\frac{\varepsilon{(1\!+|q^*|)^{-\gamma+C\varepsilon}\!\!\!\!}}{{r^*}^2(1+t+r^*)}
+\frac{\varepsilon^2(1+|q^*|)^{-2+\gamma}}
{(1\!+t\!+|q^*|)^{2+\gamma-C\varepsilon}\!}\,.
 \end{multline}
\end{lemma}

Using this improved estimate  we get
an improvement of Lemma \ref{lem:tancomp}:

\begin{lemma}\label{lem:approxwaveequation4} For $|I|\leq N-6$ we have for $r^*\geq t/2$
 \beqs
 |Z^{*I} {h}^{1e}_{TU}|
 \les{\varepsilon}{(1\!+t\!+r^*)^{-1}(1\!+q_+^*)^{-\gamma+C\varepsilon}}.
 \eqs
\end{lemma}
\begin{proof}
By Lemma \ref{lem:h1edecay} the commutator satisfies
\begin{multline*}
\big| \big[\,\Box_r^*\, ,\, \chi\big(\tfrac{\langle r-t\rangle }{t+r}\big)\,\big]
Z^{*I} {h}^{1e}_{TU} \big|\\
\les \frac{\big|\chi^\prime\big(\tfrac{\langle r-t\rangle }{t+r}\big)\big|}{1+t+r^*}
|\pa^*\! Z^{*I} {h}^{1e}_{TU}|
+\frac{\big|\chi^{\prime}\big(\tfrac{\langle r-t\rangle }{t+r}\big)\big|
+\big|\chi^{\prime\prime}\big(\!\tfrac{\langle r-t\rangle }{t+r}\big)\big|}{(1+t+r^*)^2}|{Z^*}^I {h}^{1e}_{TU}|\\
\les
{\sum}_{|J|\leq |I|}\frac{\big|\chi^\prime\big(\tfrac{\langle r-t\rangle }{t+r}\big)\big|}{1\!+t\!+r^*}
 |\pa^* \!Z^{*J}\!{h}^{1e}|
+\frac{\big|\chi^{\prime}\big(\!\tfrac{\langle r-t\rangle }{t+r}\big)\big|\!
+\big|\chi^{\prime\prime}\big(\!\tfrac{\langle r-t\rangle }{t+r}\big)\big|}{(1+t+r^*)^2}
|Z^{*J}\! {h}^{1e}|\\
\les \frac{\varepsilon (1\!+q_+)^{-\gamma+C\varepsilon}\!\!\!\!\!}{(1\!+t\!+r^*)^3}\,.
\end{multline*}
Hence \eqref{eq:approxwaveequation3} holds for
$\chi\big(\tfrac{\!\langle r-t\rangle \!}{t+r}\big) Z^{*I} {h}^{1e}_{TU}$
 in place of $Z^{*I} {h}^{1e}_{TU}$ and $r^*\!$ replaced by $1\!+t\!+r^*\!$.
 The proof follows from applying Lemma \ref{lem:inhomwaveeqdecay}
 and Lemma \ref{lem:homwaveeqdecay} to this.
\end{proof}
Using the improved estimates in the wave coordinate condition  \eqref{eq:wavecoord}
we get:
\begin{lemma} For $|I|\leq N-6$ we have
 \begin{equation*}
 \!|\pa Z^{*I\!} {h}^{1e}_{LT}|\!
 +|\pa Z^{*I} \!( {h}^{1e}_{\!AA}\!\!+{h}^{1e}_{BB})|\!
 \les {\varepsilon}{(1\!\!+t\!+r^*)^{-2+\varepsilon}
 (1\!\!+\!|q^*|)^{-\varepsilon}(1\!\!+q_+^*)^{-\gamma+C\varepsilon}}\!\!.
 \end{equation*}
The estimate holds also for $Z^{*I}(h_{UV}^1)$ replaced by the Lie derivatives $(\mathcal{L}_{Z^{\!*}}^I h_{1})_{UV}$.
\end{lemma}
\begin{proof} For $q^*\!\geq\! 0$ this is a direct consequence of
\eqref{eq:decay1starwavecoord}. To get the sharp estimate for $q^*\!\!\leq \!0$
we need to reexpress the divergence in the $x^{\!*\!}$ coordinates and repeat
the proof of Proposition \ref{prop:wavecoorddecay} in these coordinates.
By the invariance of the divergence under change of coordinates
we have $\pa_\mu F^{\mu\!}\! =\!|D|\pa^*_\gamma ( F^{\mu}\! D^\gamma_{\!\mu} \!/ |D|)$,
if $D_{\!\mu}^\gamma \!=\pa x^{\!*\gamma\!\!}/\pa x^\mu$ and $|D|\!=\det{\!D}$. By Lemma \ref{lem:changeofvariableidentities} $|D|\!=\!(1\!+M\! \ln{r}\!/r)^2(1\!+M\!/r)$.
By \eqref{eq:wavehatdivergence}
\begin{equation} \pa_\gamma^* (\widehat{h}{}^{\mu\nu\!}D^\gamma_\mu)\!
=-\tfrac{1}{2}\big(g^{\mu\nu} \! g_{\alpha\beta} -m^{\mu\nu\!}m_{\alpha\beta}\big)
\pa_\mu g^{\alpha\beta}\!\!+ \widehat{h}{}^{\mu\nu}|D|^{-1}\pa_\mu |D|\!\!\!
 \end{equation}
and expressing the divergence in a null frame as in the proof of Lemma
\ref{lem:nulldivergence}
\begin{equation*}\pa_{q^{\!*}} \big(L_\gamma U_\nu
k^{\gamma\nu}\big)=  \pa_{s^{\!*}}
\big(\Lb_{\,\gamma} U_\nu k^{\gamma\nu}\big) -A_\gamma U_\nu \pa_A^*
k^{\gamma\nu}\!\!+U_\nu  \pa_\gamma^* k^{\gamma\nu}\!,
\quad U_\nu \!=\!m_{\nu\mu}U^\mu\!\!, \,\,\,U\!\in\!\mathcal{N},
\end{equation*}
where $\pa_{\!A}^*\!=A^k(\omega)\pa_k^*$ and $k^{\gamma\nu}\!\!=\widehat{h}{}^{\mu\nu\!}D^\gamma_{\!\mu}$.
Here $L_\gamma D_\mu^\gamma\!=\Ls_{\!\!\mu} $,
$\uL_{\,\gamma} D_\mu^\gamma\!=\Lbs_{\!\!\mu}$
and $A_\gamma D^\gamma_\mu=A_\mu r^*\!/r$, by Lemma \ref{lem:changeofvariableidentities}.
The rest of the proof is as in Proposition  \ref{prop:wavecoorddecay}
but with the $x$ coordinates replaced by the $x^*$ coordinates everywhere.
Note that the difference $\Ls-L=O(r^{-1})$ is lower order so
$L_\gamma U_\nu
k^{\gamma\nu}={\Ls}_{\!\!\mu} U_\nu\widehat{h}{}^{\mu\nu}\sim L_\mu U_\nu\widehat{h}{}^{\mu\nu}$.
\end{proof}

This concludes the proof of Proposition \ref{prop:sharpmetricdecay}.
For the $L$ derivative we also have
 \begin{prop}\label{prop:specialsharpmetricdecay}
 With  $\delta_{UV}^{\Lbs\Lbs}\!\!=\!1$ if $U\!=\!V\!\!=\!\Lbs$ and $0$ otherwise we have
 \begin{align}
\frac{1}{r\!} \,|\partial_{\Ls}(r^* {Z^*}{}^I h^{1*}_{UV})|&\les
\frac{\varepsilon(1+q^*_-)^{\gamma-C\varepsilon}}{(1\!+t\!+r^*)^{2+\gamma-C\varepsilon}}
+ \delta_{UV}^{\Lbs\Lbs}\,\frac{\varepsilon(1+q_+^*)^{-\gamma}}{(1\!+t\!+r^*)^2},
\label{eq:extraLderest1}\\
 |\pa_\Ls {Z^*}{}^I h_{LT}^{1*}|&\les \frac{\varepsilon}{(1\!+t\!+r^*)^2}
 \Big(\frac{1\!+q_-^*}{1\!+t\!+r^*}\Big)^{\gamma\prime}.\label{eq:extraLderest2}
 \end{align}
 \end{prop}
\begin{proof}\!\!
Integrating $\!\Lbs\Ls(r^{*\!}\phi)\!\!=\!r^*\Box_{r}^*\phi$ along the flow lines of
$\!\Lbs\!$ from initial data
$$
 \pa_\Ls (r^*{Z^*}{}^I h_{UV}^{1*})
 =\frac{1}{2}\text{$\int_{q^*}^{\,t+r^*} $}\!\!\!\! \!\!\!\! r^*\Box^*_r{Z^*}{}^I h^{1*}_{UV} \, dq^*\!\!
 +\pa_\Ls (r^*{Z^*}{}^I h_{UV}^{1*})\Big|_{ \,r=t+r^*},
$$
where
$(|{Z^*}{}^I h^{1*}|\!+(1\!+r)|\pa {Z^*}{}^I h^{1*}|)|_{t=0}
\!\les {\varepsilon}{(1\!+r)^{-1-\gamma}\!}$
and
\beqs\label{eq:radialh1}
 \big|\Box^{\,*}_{\,r} {Z^*}{}^I {h}_{UV}^{1*}\big|
\les \frac{\varepsilon(1+q_+^*)^{-\gamma}}
{(1\!+t\!+\!|\,q^*|)^{3-C\varepsilon}\!\!}
+\frac{\varepsilon^2(1+\!|q^*|)^{-2+\gamma}}
{(1\!+t\!+\!|q^*|)^{2+\gamma-C\varepsilon}\!\!}+\delta_{UV}^{\Lbs\Lbs}
\frac{\varepsilon^2(1\!+\!|q^*|)^{-2}\!\!\!}
{(1\!+t\!+\!|q^*|)^{2}\!}.
\eqs
Integrating gives \eqref{eq:extraLderest1}, and \eqref{eq:extraLderest2} follows from
also using Proposition \ref{prop:sharpmetricdecay}.
\end{proof}

\section{Asymptotics for the wave equation
with inhomogeneous sources}
Its well known \cite{H}, that a solution of a linear wave
equation $\Box\, u=0$, with sufficiently fast decaying smooth initial
data have an asymptotic expansion
$$
u(t,x)\sim U_0(r-t,\omega)/r+U_1(r-t,\omega)/r^2+...\,\, \, ,$$
where $U_0$ is the Friedlander radiation field.
In fact
 $U(r\!-\!t,\omega,1_{\!}/r)\!=\!r u(t,x)$
 is an analytic function of $1_{\!}/r$ and $U_0$ can be calculated from
data. Data for Einstein's equations are however not fast
decaying and the equations are non-linear.
Still if the right hand side and tangential derivatives decay
fast the limit exists:
 \begin{lemma} \label{lem:radianfield}  Suppose that for some
 $0\leq\delta<1-\gamma$ and $0\leq\gamma^\prime\leq\gamma$
\begin{align*}
 |\,\Box \,\Omega^I S^J\partial_t^K u|
&\les \frac{\varepsilon}{ (1+t+r)^{3-\delta}
(1+|\,r-t|)^{\delta}(1+(r-t)_+)^{\gamma}(1+(r-t)_-)^{\gamma^\prime}},\\
 |\,\triangle_{\,\omega} \Omega^I S^J\partial_t^K u|
&\les \frac{\varepsilon}{ (1+t+r)^{1-\delta}
(1+|\,r-t|)^{\delta}(1+(r-t)_+)^{\gamma}(1+(r-t)_-)^{\gamma^\prime}},
\end{align*}
for $|I|+|J\,|+|K|\leq N$ and $r>t/4$, and
$$
|\,(\pa_t+\pa_r)\big(\Omega^I S^J\partial_t^K r u\big)|
\les {\varepsilon}{(1+r)^{-1-\gamma}},\qquad\text{when}\quad t=0.
$$
 Then
$$
 |(\partial_t\!+\partial_r)( \Omega^I S^J\partial_t^K  r u)(t,r\omega)|
 \!\les \! \varepsilon{(1\!+(r\!-\!t)_-)^{\gamma-\gamma^\prime}\!}{(1\!+t\!+r)^{-1-\gamma}\!},
\quad r\!\geq \!t_{\!}/4.\!\!
 $$
Moreover, limit
 $$
U^\infty(q,\omega)=\lim{}_{r\to\infty} U(q,\omega,r), \qquad
 U(q,\omega,r)=r
u(r-q,r\omega),
$$
exists and satisfies, for  $r>t/4$ and  $|I|+|J|+|K|\leq N$,
\begin{align}
\label{eq:waveeqlimitapprox}
 \big|\,\Omega^I S^J\partial_t^K r u(t,r\omega)
 -\Omega^I (q\,\partial_q)^J(-\partial_q)^K U^\infty(q,\omega)\big|&\les
 \varepsilon \frac{(1+q_-)^{\gamma-\gamma^\prime}\!\!\!\!}{(1+r)^\gamma},\\
 \label{eq:waveeqlimitapprox2}
 \big|\Omega^I (q\,\partial_q)^J\partial_q^K U^\infty(q,\omega)\big|&\les
 \varepsilon \frac{(1+q_-)^{\gamma-\gamma^\prime}\!\!\!\!}{(1+|\,q|)^\gamma} ,
 \end{align}

 \end{lemma}
\begin{proof} We prove the result for $N\!=\!0$ as the case $N\!>\!0$ follows from the same argument. This follows from expressing the wave operator in
spherical coordinates
 $$
 (\pa_t-\pa_r)(\pa_t+\pa_r) (ru)=r^{-1} \triangle_\omega u
-r \,\Box \, u
 $$
 and integrating first in the $r-t$ direction from $(t,r)$ to initial data when $t=0$
 \begin{multline*}
 |(\partial_t+\partial_r)(ru)(t,r\omega)|
 \les \int_{r-t}^{t+r} \!\!\big| \,r^{-1} \triangle_\omega u
-r \,\Box \, u\big| \, d q +|(\partial_t+\partial_r)(ru)(0,(t+r)\omega)|\\
\les \int_{r-t}^{t+r} \frac{\varepsilon\, dq}{ (1+t+r)^{2-\delta}
(1+|\,q|)^{\delta}(1+q_+)^{\gamma}(1+q_-)^{\gamma^\prime}}
+\frac{\varepsilon}{(t+r)^{1+\gamma}}\\
\les \frac{\varepsilon}{(t+r)^{1+\gamma}}
+\frac{\varepsilon(r-t)_-^{1-\delta-\gamma^\prime}}{(t+r)^{2-\delta}}
\les \varepsilon\frac{(1+(r-t)_-)^{\gamma-\gamma^\prime}}{(1+t+r)^{1+\gamma}}\qquad
 r>t/2.
 \end{multline*}
 For fixed $q\!=\!r\!-\!t$ integrating this in $t\!+\!r$ between
$2\,r_1\!-\!q\!\leq\! r\!+\!t\!\leq \!2\,r_2\!-\!q$ gives
 \beqs
 \big|\, U(q,\omega,r_2)
 -U(q,\omega,r_1)\big|
 \les
 \varepsilon {(1+q_-)^{\gamma-\gamma^\prime}}{(1+r_1)^{-\gamma}},\qquad
 r>t/2,
 \eqs
 from which it follows that the limit exists and satisfies
 \eqref{eq:waveeqlimitapprox}-\eqref{eq:waveeqlimitapprox2}.
\end{proof}

 For Einstein's equations we have the extra difficulty
 that it is a system and the components do not separate due to angular
 derivatives on the frame:
 $$
 \Box ( T^\mu U^\nu h_{\mu\nu})-T^\mu U^\nu \Box h_{\mu\nu}
 =r^{-2}\triangle_\omega ( T^\mu U^\nu h_{\mu\nu})
 -T^\mu U^\nu r^{-2} \triangle_\omega h_{\mu\nu}.
 $$
 It can be estimated in terms of tangential derivatives
 of all components.  This procedure will give us the existence of the radiation
 field for all components in a null frame except for
 $h_{\underline{L}\underline{L}}$. This component will in fact not
 have as simple radiation field but there will
 also be a logarithm in its radiation field. However,
 the asymptotics of the
 source  $P\,(\pa_\mu h,\pa_\nu h)$
 can be calculated in terms of the radiation field of the components
 we already calculated. It will be of the form
 $$
 P_\mathcal{N}(\pa_\mu h,\pa_\nu h)\sim   C^{TUSV} L_\mu(\omega)
 L_\nu(\omega)r^{-2} \pa_q H^\infty_{TU}(r-t,\omega)
 \pa_q H^\infty_{SV}(r-t,\omega),\quad r>\tfrac{t}{2}.
 $$
We now want to find an approximate solution to $\Box\,
\phi\!=P\,(\pa_\mu h,\pa_\nu h)$. Formulas for the solution of the wave equation with such sources were obtained in \cite{L1}. First we use a simplified version which is sufficient for asymptotics in null directions.

\begin{prop}\label{prop:sourcedecay} Let $\chi\in C_0^\infty$ satisfy
 $\chi({q})=0$, when
$|\,{q}|\geq 3/4$ and $\chi({q})=1$, when $|\,{q}|\leq 1/2$.
Set
 \beqs\label{eq:approxsource}
  F[n](t,x)={n(r-t,\omega
)}{r^{-2}}\chi\big(\tfrac{\langle \,r-t\,\rangle }{t+r}\big),\qquad \langle q\rangle = \sqrt{1+q^2}
 \eqs
where $n$ is a smooth function satisfying
\beqs\label{eq:radiationderbound} {\sum}_{|\alpha|+k\leq N}\,\,
|\,(\langle q\rangle\pa_{q})^k\pa_\omega^\alpha n({q},\omega)|\les
{\langle q\rangle^{-1-a}},\qquad 0<a<1.
 \eqs
 Let $\Phi[n]$ be the solution of $-\Box\, \Phi[n]=F[n]$ with vanishing initial data and let
 \beqs
\Phi_1[n](t,r\omega)= \int_{r-t}^\infty \frac{1}{2r}\ln{\Big(\frac{t+r+q}{t-r+q}\Big)} n\big(q,\omega\big)\, d q\, \chi\big(\tfrac{\langle\,r-t\,\rangle}{t+r}\big).
 \eqs
 Set $\Phi_0[n]=\Phi[n]-\Phi_1[n]$,
$u=\Phi[n]$, $u_1=\Phi_1[n]$ and $u_0=\Phi_0[n]=u-u_1$.
 Then
 \beqs
 \big|(\partial_t+\partial_r)\big(\Omega^I S^J\partial_t^K r u_0)\big|
 \les \frac{(1+(t-r)_+)^a}{(1+t+r)^{1+a}}, \qquad |I|+|J|+|K|\leq N.
 \eqs
 Moreover, limit
 $$
U_0^\infty(q,\omega)=\lim_{r\to\infty} U_0(q,\omega,r), \qquad
 U_0(q,\omega,r)=r
u_0(r-q,r\omega),
$$
exists and satisfies
\begin{align*}
 \big|\,\Omega^I S^J\partial_t^K r u_0(t,r\omega)
 -\Omega^I (q\,\partial_q)^J(-\partial_q)^K U_0^\infty(q,\omega)\big|
 &\les
 \varepsilon \frac{(1+q_-)^{a}}{(1+r)^a},\\
 \big|\Omega^I (q\,\partial_q)^J\partial_q^K U_0^\infty(q,\omega)\big|
 &\les
 \varepsilon \frac{(1+q_-)^{a}}{(1+|\,q|)^a} ,
 \end{align*}
 for  $r>t/4$ and $|I|+|J|+|K|\leq N$. Furthermore
 \begin{equation*}
 \Big|\,\Omega^I S^J\partial_t^K u-\Phi_1\big[\Omega^I(q\partial_q)^J(-\partial_q)^K n\big]\Big|
 \les \frac{1}{(1+t+r)(1+(r-t)_+)^{a}},
 \end{equation*}
 and
 \begin{equation*}
 \!\!\!\!\!\big|\,(\partial_t+\partial_r)\big(r\,\Omega^I S^J\partial_t^K u)-\Phi_{1+}\big[\Omega^I(q\partial_q)^J(-\partial_q)^K n\big]\big|
 \les \frac{(1+(t-r)_+)^a}{(1+t+r)^{1+a}},
 \end{equation*}
where
 \beqs
\Phi_{1+}[n](t,r\omega)=\frac{1}{2\,r} \int_{r-t}^\infty  n\big(q,\omega\big)\, d q\, \chi\big(\tfrac{\langle\,r-t\,\rangle}{t+r}\big).
 \eqs
 \end{prop}
 For the proof we need the following technical lemma:
 \begin{lemma}\label{lem:technical} Suppose that
  $m$ is a smooth function satisfying
\beqs\label{eq:radiationderboundm} {\sum}_{|\alpha|+k\leq N}\,\,
|\,(\langle q\rangle\pa_{q})^k\pa_\omega^\alpha m({q},\omega)|\les
{\langle q\rangle^{-1-b}}.
 \eqs
 Let $\delta_{b1}=1$, when $b=1$ and $0$ otherwise.
 Then if $0<b\leq 2$
\begin{multline*}
\big|\!\!\int_{r-t}^\infty\!\!\!\!\!\!
\big(\ln{\big|\frac{t\!+r\!+q}{t\!-r\!+q}\big|}-\ln{\big|\frac{\langle t\!+\!r\rangle\!}
{\langle t\!-\!r\rangle\!}\big|}\big)m\big(q,\omega\big)\, d q\, \big|
\!\les \frac{1}{\langle t\!-\!r\rangle^{b}\!\!}
+\frac{H(t\!>\!r)\!\!}{\langle t\!-\!r\rangle\!}\,
(1\!+\delta_{b1}\ln{\langle t\!-\!r\rangle}).
 \end{multline*}
 \end{lemma}
 \begin{proof}[of Lemma \ref{lem:technical}] The integral over $q\geq t+r$ is easily bounded by
 $$\big(1+\ln{\big(\frac{\langle t+r\rangle}{\langle t-r\rangle}\big)}\big)
 \frac{1}{\langle t+r\rangle^b}\les \frac{1}{\langle t-r\rangle^b},
 $$
 so we may integrate over  $q\leq t+r$ only.
 \beqs
 \ln{\Big(\frac{t+r+q}{t-r+q}\Big)}-\ln{\Big(\frac{\langle \,t+r\,\rangle}{\langle\,t-r\,\rangle}\Big)}
 =\ln{\Big(\frac{t+r+q}{\langle\,t+r\,\rangle}\Big)}
 +\ln{\Big(\frac{\langle\,t-r\,\rangle}{t-r+q}\Big)}.
 \eqs
The integral of the first term is also easy to bound. If $r\!>\!t$ it can be bounded by
$$
\int_{r-t}^{t+r} \frac{\langle q\rangle}{t+r}\frac{dq}{\langle q\rangle^{1+b}}\les
\frac{1}{t+r} \Big(\langle t+r\rangle^{1-b}+\langle t-r\rangle^{1-b}
+\delta_{b1}\ln{\Big(\frac{\langle t\!+\!r\rangle}{\langle t-r\rangle}\Big)} \Big)
\les \frac{1}{\langle t-r\rangle^b}.
$$
where $\delta_{b1}=1$ if $b=1$ and $0$ otherwise,
and if $r>t$ by
$$
\int_{r-t}^{t+r} \frac{\langle q\rangle}{t+r}\frac{dq}{\langle q\rangle^{1+b}}\les
\frac{1}{t+r} \Big(1+\langle t-r\rangle^{1-b}
+\delta_{b1}\ln{\langle t-r\rangle}\Big).
$$
We may therefore concentrate on the second term. If $|r\!-t|\!\leq \!1$ the integral
is easily bounded so we may as well assume that $|r\!-t|\!>\!1$.
Moreover, in that case
\beqs
\Big|\ln{\Big(\frac{\langle\,t-r\,\rangle}{|\,t-r|}\Big)}\Big|\les \frac{1}{\langle t-r\rangle^2}
\eqs
so we are left with estimating
$$
\int_{r-t}^{t+r}\Big| \ln{\Big(\frac{|\, t-r|}{t-r+q}\Big)}\Big|\frac{dq}{(1+|\,q|)^{1+b}}\les
\frac{1}{|\, t-r|^{\,b}} \int_0^{2r} \!\!\!\frac{ |\ln{s}|\, ds}{\big({1}/{|t-r|}+|\,s\pm 1|\big)^{1+b}},
$$
where the sign $\pm$ is the same as the sign of $r-t$. If $r>t$ or $b<1$
this is bounded by $|\,t-r|^{-b}$. If $r<t$ and $b>1$ its bounded by
$|\,t-r|^{-1}$ and if $b=1$ its bounded by $|\,t-r|^{-1}\ln{|\,t-r|}$.
\end{proof}

 \begin{proof}[Proof of Proposition \ref{prop:sourcedecay}] We have
 \beqs
 (\pa_t-ï\pa_r)(\pa_t\!+\!\pa_r) \!\!
 \int_{r-t}^\infty \!\!\!\!\ln{\big|\frac{t\!+\!r\!+q}{t\!-r\!+q}\big|}
 \, \frac{n\big(q,\omega\big)dq\!\!}{2} \,
 \!=\!(\pa_t-\pa_r)\!\!\int_{r-t}^\infty \!\frac{n\big(q,\omega\big)dq\!\!}{t\!+r\!+q}\,
 =\frac{n(r-t,\omega)\!}{r}.
 \eqs
 We can write
 \beqs
 \Phi_1[n]=\frac{1}{2r}\int_0^\infty \ln{\Big(\frac{2r+s}{s}\Big)}
 n\big(s+r-t,\omega\big)\, d s\, \chi\big(\tfrac{\langle\,r-t\,\rangle}{t+r}\big).
 \eqs
 Hence
 \begin{multline*}
-\Box\,\Phi_1[n]-F[n]=-\frac{1}{2r^3} \int_{0}^\infty \ln{\Big(\frac{2r+s}{s}\Big)} \triangle_{\,\omega} n\big(s+r-t,\omega\big)\, d s\, \chi\big(\tfrac{\langle\,r-t\,\rangle}{t+r}\big)\\
 +\frac{1}{2r}(\pa_t+\pa_r)\int_0^\infty \ln{\Big(\frac{2r+s}{s}\Big)} n\big(s+r-t,\omega\big)\, d s\, (\pa_t-\pa_r)\chi\big(\tfrac{\langle\,r-t\,\rangle}{t+r}\big)\\
  +\frac{1}{2r}(\pa_t-\pa_r)\int_0^\infty \ln{\Big(\frac{2r+s}{s}\Big)} n\big(s+r-t,\omega\big)\, d s\, (\pa_t+\pa_r)\chi\big(\tfrac{\langle\,r-t\,\rangle}{t+r}\big)\\
 +\frac{1}{2r}\int_0^\infty \ln{\Big(\frac{2r+s}{s}\Big)} n\big(s+r-t,\omega\big)\, d s\, (\pa_t-\pa_r)(\pa_t+\pa_r)\chi\big(\tfrac{\langle\,r-t\,\rangle}{t+r}\big)
 \end{multline*}
 so
 \begin{multline}\label{eq:boxphizero}
 -\Box\,\Phi_0[n]=-\frac{1}{2r^3} \int_{0}^\infty \ln{\Big(\frac{2r+s}{s}\Big)} \triangle_{\,\omega} n\big(s+r-t,\omega\big)\, d s\, \chi\big(\tfrac{\langle\,r-t\,\rangle}{t+r}\big)\\
 +\frac{2}{r^2}\int_0^\infty\frac{1}{2r+s} n\big(s+r-t,\omega\big)\, d s\, \chi^\prime\big(\tfrac{\langle\,r-t\,\rangle}{t+r}\big)
 \big(\tfrac{t-r}{\langle\,r-t\,\rangle}
 -\tfrac{\langle\,r-t\,\rangle}{t+r}\big)\tfrac{r}{t+r}\\
  +\frac{2}{r^2}\int_0^\infty \ln{\Big(\frac{2r+s}{s}\Big)} n_{\,q}^\prime\big(s+r-t,\omega\big)\, d s\, \chi^\prime\big(\tfrac{\langle\,r-t\,\rangle}{t+r}\big)
  \tfrac{\langle\,r-t\,\rangle}{t+r}\tfrac{r}{t+r}\\
 +\int_0^\infty \!\!\!\!\!\ln{\Big(\frac{2r+s}{s}\Big)} \frac{n\big(s+r-t,\omega\big)}{r^3}\, d s\, \Big(\chi^{\prime\prime}\big(\tfrac{\langle\,r-t\,\rangle}{t+r}\big)
 \tfrac{\langle\,r-t\,\rangle}{t+r}
 -2\chi^\prime\big(\tfrac{\langle\,r-t\,\rangle}{t+r}\big)\Big)\times\\
 \times
 \big(\tfrac{t-r}{\langle\,r-t\,\rangle}
 -\tfrac{\langle\,r-t\,\rangle}{t+r}\big)\tfrac{r^2}{(t+r)^2}.
 \end{multline}
By Lemma \ref{lem:technical}
\begin{equation*}
\Big|\int_0^\infty \!\!\!\ln{\big|\frac{2r+s}{s}\big|}
n_{\,q}^\prime\big(s+r-t,\omega\big)\, d s
+\ln{\big|\frac{\langle \,t\!+\!r\,\rangle}{\langle\,t\!-\!r\,\rangle}\big|}n(r-t,\omega)\Big|
\les \frac{1}{\langle t\!-\!r\rangle}\frac{1}{\langle (r\!-\!t)_+\rangle^a\!\!}\,,
\end{equation*}
and
\begin{align*}
\big|\int_0^\infty \!\!\!\!\ln{\big|\frac{2r\!+s\!}{s}\big|}\, n\big(s\!+\!r\!-\!t,\omega\big)\, d s
+\ln{\big|\frac{\langle \,t\!+\!r\,\rangle}{\langle\,t\!-\!r\,\rangle}\big|}
\int_{0}^\infty \!\!\!\!\!n(s+r-t,\omega)\, ds\big|
\!&\les\! \frac{1}{\langle (r\!-\!t)_+\rangle^a\!},\\
\big|\!\int_0^\infty \!\!\!\!\!\ln{\big|\frac{2r\!+s\!}{s}\big|} \triangle_{\,\omega}
n\big(s\!+\!r\!-t,\omega\big)\, d s
\!+\ln{\big|\frac{\langle t\!+\!r\rangle}{\langle t\!-\!r\rangle}\big|}
\!\!\int_{0}^\infty\!\!\!\!\!\!\!\triangle_{\,\omega} n(s\!+\!r\!-\!t,\omega)\, ds\big|
\!&\les\! \frac{1}{\langle (r\!-\!t)_+\rangle^a\!}.
\end{align*}
In view of this and that $t-r\sim t+r$ in the support of $\chi^\prime$ it follows that
\beqs
\big|\,\Box\,\Phi_0[n]\big|\les\frac{1}{\langle t+r\rangle^{3}}
\ln{\Big(\frac{\langle \,t+r\,\rangle}{\langle\,t-r\,\rangle}\Big)}
\frac{1}{\langle (r-t)_+\rangle^a}.
\eqs
We claim that
\beqs
\big|\triangle_{\,\omega}\Phi_1[n]\big|+\big|\triangle_{\,\omega}\Phi[n]\big|\les
\frac{1}{\langle t+r\rangle}\ln{\Big(\frac{\langle \,t+r\,\rangle}{\langle\,t-r\,\rangle}\Big)}\frac{1}{\langle (r-t)_+\rangle^a}.
\eqs
For $\Phi_1[n]$ this follows from using Lemma \ref{lem:technical} applied to
\beq\label{eq:phione}
 \Phi_1[n]=\frac{1}{2r}\int_0^\infty \ln{\Big(\frac{2r+s}{s}\Big)} n\big(s+r-t,\omega\big)\, d s\, \chi\big(\tfrac{\langle\,r-t\,\rangle}{t+r}\big)
 \eq
 with $n$ replaced by $\triangle_{\,\omega}n$. For $\Phi[n]$ this follows from
Lemma \ref{lem:inhomwaveeqdecay} applied to
\beq\label{eq:boxphi}
-\Box\,\Phi[n]=F[ n]={ n(r-t,\omega)}{r^{-2}}\chi\big(\tfrac{\langle \, t-r\rangle}{t+r}\big)
\eq
with $n$ replaced by $\triangle_{\,\omega}n$.
Hence the same estimate holds for
$\triangle_{\,\omega} \Phi_0$.
It therefore follows from Lemma \ref{lem:radianfield} that the $u_0$ satisfies the required estimates for $N=0$ and that the limit $U^\infty$ exists and satisfies
the estimates for $N=0$. It remains to prove these estimates for $N>0$.
For angular derivatives this is clear since then $\Omega^I \Phi[n]=\Phi[\Omega^I n]$,
$\Omega^I \Phi_1[n]=\Phi_1[\Omega^I n]$ and $\Omega^I \Phi_0[n]=\Phi_0[\Omega^I n]$.
For time derivatives we have that modulo terms with the time derivative falling on the cutoffs and whenever a time derivative falls on cutoff we get terms of the same form
but with additional decay of $(t+r)^{-1}$ and moreover $t-r\sim t+r$ in the support of the cutoff so we have
$$
\big|\,\Box\,\partial_t \Phi_0[n] -\Box\,\Phi_0[-\partial_q n]\big|
\les{\langle t+r\rangle^{-4}}{\langle (r-t)_+\rangle^{-a}} .
$$
Now $\Box \,S\Phi_0[n]= (S+2)\Box \,\Phi_0[n]$. Recall that $Sf(t,r)=\partial_a f(at,ar)\big|_{a=1}$. Let $I[n](t,r)$ denote any of the integrals in
\eqref{eq:boxphizero} with the factors of $r$ in front. Then changing variables we see that $I[n(q,\omega)](at,ar)=a^{-2}I[n(aq,\omega)]$ so
$(S+2) I[n]=I[q\partial_q n]$. Any of the factors that multiply the integrals
would be homogeneous of degree $0$ if $\langle t-r\rangle$ was replaced by $t-r$
and hence they would have vanished in that case when $S$ is applied to them.
However the error is of lower order because
\beqs
S \frac{\langle t-r\rangle}{t-r}=\frac{1}{2}\frac{t-r}{\langle t-r\rangle} S\frac{\langle t-r\rangle^2}{(t-r)^2}=\frac{1}{2}\frac{t-r}{\langle t-r\rangle}
S\frac{1}{(t-r)^{2}}=-\frac{1}{\langle t-r\rangle^2}
\frac{\langle t-r\rangle}{t-r} .
\eqs
Hence
$$
\big|\,\Box\,S \Phi_0[n] -\Box\,\Phi_0[q\partial_q n]\big|
\les{\langle t+r\rangle^{-5}}{\langle (r-t)_+\rangle^{-a}} .
$$
If $S$ falls on the cutoff functions we get errors that decay like $(t+r)^{-2}$
and if another $S$ falls on these errors the errors still decays like $(t+r)^{-2}$.
If a $t$ derivative first falls of the cutoffs then we get an error that decay like
$(t+r)^{-1}$ and again $S$ applied to it still decays like $(t+r)^{-1}$.
We conclude that
$$
\big|\,\Box\,\Omega^I S^J\partial_t^K \Phi_0[n]
-\Box\,\Phi_0[\Omega^I (q\,\partial_q)^J(-\partial_q)^K n]\big|
\les{\langle t+r\rangle^{-4}}{\langle (r-t)_+\rangle^{-a}}
$$
and hence
\beqs
\big|\,\Box\,\Omega^I S^J\partial_t^K\Phi_0[n]\big|\les\frac{1}{\langle t+r\rangle^{3}}\ln{\Big(\frac{\langle \,t+r\,\rangle}{\langle\,t-r\,\rangle}\Big)}\frac{1}{\langle (r-t)_+\rangle^a} .
\eqs
By the same argument if we apply these operators to \eqref{eq:phione}
we get
\begin{align*}
\big|\,\partial_t \triangle_{\,\omega}\Phi_1[n] -\Phi_1[-\partial_q \triangle_{\,\omega} n]\big|
&\les{\langle t+r\rangle^{-2}}{\langle (r-t)_+\rangle^{-a}},\\
\big|\,S \triangle_{\,\omega}\Phi_1[n] -\Phi_1[q\partial_q \triangle_{\,\omega} n]\big|
&\les{\langle t+r\rangle^{-3}}{\langle (r-t)_+\rangle^{-a}}.
\end{align*}
Repeating these arguments give
$$
\big|\Omega^I S^J\partial_t^K \triangle_{\,\omega}\Phi_1[n]
 -\Phi_1[\Omega^I (q\,\partial_q)^J(-\partial_q)^K \triangle_{\,\omega} n]\big|
\les{\langle t+r\rangle^{-2}}{\langle (r-t)_+\rangle^{-a}}
$$
and hence
\beqs
\big|\Omega^I S^J\partial_t^K\triangle_{\,\omega}\Phi_1[n]\big|\les\frac{1}{\langle t+r\rangle}\ln{\Big(\frac{\langle \,t+r\,\rangle}{\langle\,t-r\,\rangle}\Big)}\frac{1}{\langle (r-t)_+\rangle^a}.
\eqs
Applying the vector fields to \eqref{eq:boxphi} gives the error term when
the vector fields fall on the cutoff function
\beqs
\big|\,\Box\, \partial_t \Phi[n]-\Box\,\Phi[-\partial_q n]\big|\les
{\langle t+r\rangle^{-4-a}}
\eqs
since $t\!-\!r\!\sim \!t\!+\!r$ in the support of the derivative $\chi^\prime$
and since $n$ decays.
Similarly
\beqs
\big|\,\Box\, S \Phi[n]-\Box\,\Phi[q\partial_q n]\big|\les
{\langle t+r\rangle^{-5-a}}.
\eqs
In general we get as above
\beqs
\big|\,\Box\, \Omega^I S^J \partial_t^K \Phi[n]-\Box\,\Phi[\Omega^I (q\,\partial_q)J(-\partial_q)^K n]\big|\les
{\langle t+r\rangle^{-4-a}},
\eqs
and
\beqs
\big|\triangle_{\,\omega}\Omega^I S^J \partial_t^K\triangle_{\,\omega} \Phi[n]\big|\les
\frac{1}{\langle t+r\rangle}\ln{\Big(\frac{\langle \,t+r\,\rangle}{\langle\,t-r\,\rangle}\Big)}\frac{1}{\langle (r-t)_+\rangle^a}.
\eqs
This finishes the proof of the first part of the theorem also for $N>0$.

It remains to prove the last estimate which since we already have an estimate for
vector fields of $(\pa_t+\pa_r)(r\Phi_0[n])$ would follow from
\begin{multline*}
 \sum_{|I|+|J|+|K|\leq N \!\!\!\!\!\!\!\!\!\!\!\!\!\!\!\!\!\!\!\!\!\!\!\!\!\!}
 \big|(\partial_t+\partial_r)\big(r\,\Omega^I S^J\!\partial_t^K \Phi_1[n])-\Phi_{1+}
 \big[\Omega^I(q\partial_q)^J(-\partial_q)^K n\big]\big|
 \les \frac{(1\!+(t\!-\!r)_+)^a\!\!}{(1\!+\!t\!+\!r)^{1+a}\!\!}\,.
 \end{multline*}
We have
\begin{multline*}
(\partial_t+\partial_r)(r \Phi_1[n])=(\partial_t+\partial_r)\frac{1}{2}\int_0^\infty \ln{\Big(\frac{2r+s}{s}\Big)} n\big(s+r-t,\omega\big)\, d s\, \chi\big(\tfrac{\langle\,r-t\,\rangle}{t+r}\big)\\
=\int_0^\infty \frac{1}{2r+s} n\big(s+r-t,\omega\big)\, d s\, \chi\big(\tfrac{\langle\,r-t\,\rangle}{t+r}\big)\\
-\frac{1}{2 r}\int_0^\infty \ln{\Big(\frac{2r+s}{s}\Big)} n\big(s+r-t,\omega\big)\, d s\, \chi^\prime\big(\tfrac{\langle\,r-t\,\rangle}{t+r}\big)
\tfrac{\langle\,r-t\,\rangle}{t+r}\tfrac{2r}{t+r}
 \end{multline*}
 Here the first integral on the right is
 \begin{equation*}
 \int_0^\infty \frac{n\big(s+r-t,\omega\big)}{2r+s} \, d s
 =\frac{1}{2 r} \int_{r-t}^\infty n(q,\omega)\, dq+
 \frac{1}{2r} \int_{r-t}^\infty \frac{(r-t-q)\, n(q,\omega)}{t+r+q}\, dq,
 \end{equation*}
 so we obtain
 \begin{multline*}
(\partial_t+\partial_r)(r \,\Phi_1[n])
-\Phi_{1+}[n]
=\frac{1}{2r} \!\int_{r-t}^\infty\!\! \frac{(r\!-\!t\!-\!q)\, n(q,\omega)\!\!}{t+r+q}\, dq\,
\chi\big(\tfrac{\langle\,r-t\,\rangle}{t+r}\big)\\
-\frac{1}{2r}\!\int_0^\infty \!\!\!\!\ln{\!\Big(\frac{2r\!+s}{s}\Big)} n\big(s\!+\!r\!-\!t,\omega\big)\, d s\, \chi^\prime\big(\tfrac{\langle\,r-t\,\rangle}{t+r}\big)
\tfrac{\langle\,r-t\,\rangle}{t+r}\tfrac{2r}{t+r}.
 \end{multline*}
 Here
 \begin{align*}
 \int_{r-t}^\infty \!\frac{|q|\,| n(q,\omega)|}{t+r+q}\, dq
 &\les \frac{1}{t\!+\!r} \int_{r-t}^{t+r} \!\!\!\!\! |q|\,|n(q,\omega)|\, dq
 +\!\int_{t+r}^\infty
 \!\! \!\!|n(q,\omega)|\, dq\les \frac{1}{(t\!+r)^a\!\!}\, ,\\
 \int_{r-t}^\infty \!\frac{|t\!-\!r|\, | n(q,\omega)|}{t+r+q}\, dq
 &\les \frac{|t-r|}{t+r} \int_{r-t}^{t+r} |n(q,\omega)|\, dq
 \les \frac{|t-r|}{t+r}\frac{1}{(1+(r-t)_+)^a} ,
 \end{align*}
 and by Lemma \ref{lem:technical}
 \beqs
 \int_0^\infty \ln{\Big(\frac{2r+s}{s}\Big)}|\, n\big(s+r-t,\omega\big)|\, d s
 \les \ln{\Big(\frac{\langle \,t+r\,\rangle}{\langle\,t-r\,\rangle}\Big)}
 \frac{1}{\langle (r-t)_+\rangle^a}.
 \eqs
 However, this term is multiplied with the derivative of the cutoff function
 that is supported when $t-r\sim t+r$ so we can multiply $\langle t-r\rangle^a/\langle t+r\rangle^a$ and we get that all remainder terms are bounded by
 \beqs
 \big|\,(\partial_t+\partial_r)\big(r \, \Phi_1[n]\big)-\Phi_{1+}\big[ n\big]\big|
 \les {(1+(t-r)_+)^a}{(1+t+r)^{-1-a}}.
 \eqs
 We now need to apply vector fields. Note that $S (\partial_t\!+\!\partial_r)(r\phi)=
 (\partial_t\!+\!\partial_r)(r S\phi)$ so
 \beqs
 \Omega^I S^J \partial_t^K (\partial_t+\partial_r) (r\,\Phi_{1}[n])
 =(\partial_t+\partial_r) (r\,\Omega^I S^J \partial_t^K\Phi_{1}[n]).
 \eqs
 As before we have
 \begin{multline*}
 \!\!\!\!\!\!\!\big|  \Omega^I S^J \partial_t^K\big((\partial_t+\partial_r)(r \Phi_1)
-\Phi_{1+}\big)[n] - \big((\partial_t+\partial_r)(r \,\Phi_1)
-\Phi_{1+}\big)[\Omega^I (q\,\partial_q)^J (-\partial_q)^K n] \big|\\
\les \frac{(1+(t-r)_+)^a}{(1+t+r)^{2+a}}
 \end{multline*}
 so
 \beqs
 \big| \Omega^I S^J \partial_t^K\big((\partial_t\!+\partial_r)\big(r  \Phi_1[n]\big)\!
 -\Phi_{1+}\big[ n\big]\big)\big|
 \les {\langle (t\!-\!r)_+\rangle^{-a}}{(1\!+t\!+r)^{-1-a}\!\!},
 \eqs
 which concludes the proof of the proposition.
  \end{proof}

\section{The asymptotic of the metric in
Schwarzschild coordinates} Let
\beqs
 h^{*}_{\mu\nu}(t,r^*\omega)=h_{\mu\nu}(t,r\omega),
 \qquad \quad
H^*_{TU}(q^*,\omega,r^*)=r^* h^{*}_{TU}(r^*-q^*, r^*\omega),
 \eqs
 where $r^*=r+M\ln{|1+r|}$,
 and similarly define $h^{1*}$, $H^{1*}$, $h^{0*}$ and $H^{0*}$.

\begin{prop}\label{prop:lim} The limit
\beqs \label{eq:limit}
H^{1\infty}_{TU}(q^*,\omega)
=\lim_{r^*\to\infty}H^{1\,*}_{TU}(q^*,\omega,r^*),
\eqs
 exists and satisfies $H_{TU}^{1\infty}\!=\!H_{UT}^{1\infty}$,
and
$
H^{1\infty}_{LT}(q^*,\omega)\!=\!\delta^{AB}H^{1\infty}_{AB}
\textbf{}(q^*,\omega)\!=\!0.
$
Moreover, for  $|\alpha|+k\leq N-6$ and $r>t/2$
\begin{align*}
\big|\,\pa_\omega^\alpha \big((1+|\,q^*|)\pa_{q^*}\big)^k
H^{1\infty}_{TU}(q^*,\omega)\big| &\les \varepsilon
(1+q^*_+)^{-\gamma^\prime},\\
\big|\pa_\omega^\alpha
\big((1+|q^*|)\pa_{q^*}\big)^k\big[H^{1*}_{TU}(q^*,\omega,r^*)
 -H^{1\infty}_{TU}(q^*,\omega)\big]\big|
 &\les\varepsilon \Big(\frac{1+q_-^*}{1+r^*}\Big)^{\gamma^\prime}.
 \end{align*}
\end{prop}
\begin{proof} Note first that
\beqs \big|\pa_\omega^\alpha \big((1+|q^*|)\pa_{q^*}\big)^k
H^{1*}_{TU}(q^*,\omega,r^*)\big| \les {\sum}_{|I|\leq |\alpha|+k}
\big| r^* Z^{*I} h^{1*}_{TU}\big|,
 \eqs
 since $|{Z^*}^Jr^*|\les 1+t+r^*$.
By Lemma \ref{lem:approxwaveequation3} for $|I|\leq N-6$ we have the
estimate
 \begin{multline*}
\big| (\pa_{t^*}+\pa_{r^*})
 (\pa_{r^*}-\pa_{t*})(r^* Z^{*I} h^{1*}_{TU})\big|\\
\les \frac{\varepsilon(1+q_+^*)^{-\gamma}}
{(1+t+r^*)^{2-C\varepsilon}(1+|\,q^*|)^{C\varepsilon}}
+\frac{\varepsilon^2(1+|\,q^*|)^{-2}}
{(1+t+r^*)^{1+\gamma-C\varepsilon}}.
 \end{multline*}
 Integrating this over the set $2r_1^*-q^*\leq r^*+t\leq
 2r_2^*-q^*$, $r^*-t\geq q^*$, $t\geq 0$ gives
 \begin{multline*}
 \Big|\pa_\omega^\alpha \big((1+|\,q^*|)\pa_{q^*}\big)^k
 \Big( H^{1*}_{TU}(q^*,\omega,r_2^*)
 -H^{1*}_{TU}(q^*,\omega,r_1^*)\Big)\Big|\\
 \les
 \frac{\varepsilon}{(1+r_1^*)^{\gamma-C\varepsilon}}
 +\varepsilon \Big(\frac{1+q_-^*}{1+r_1^*}\Big)^{1-C\varepsilon},
 \end{multline*}
 from which it follows that the limit exists. Moreover, by
 Proposition \ref{prop:sharpmetricdecay}
 \beqs
 \big|\pa_\omega^\alpha \big((1+|\,q^*|)\pa_{q^*}\big)^k
 H^{1*}_{TU}\big|\les {\varepsilon}{(1+q_+^*)^{-\gamma^\prime}}
 \eqs
 so that is true for the limit as well.
 That $H^{1\infty}_{LT}(q^*,\omega)=\delta^{AB}
H^{1\infty}_{AB}\textbf{}(q^*,\omega)=0$ follows from passing to
the
limit in the wave coordinate condition.
\end{proof}
Let
 $
 V^*_{TU}=\pa_{q^*}H^*_{TU}$ and
 $V^\infty_{TU}=\pa_{q^*}H^\infty_{TU}$.
 Note that this is the same as with $H^{1*}$ and $H^{1\infty}$.
Then by \cite{LR3} we have
 \begin{multline*}
 P\,\big(V,V\big)
=-\tfrac{1}{4}V_{LL} V_{\underline{L}\underline{L}}
-\tfrac{1}{4}\delta^{CD}\delta^{C^\prime
D^\prime}\big(2V_{CC^\prime}V_{DD^\prime}- V_{CD} V_{C^\prime
D^\prime}\big)\\
 +\tfrac{1}{2}\delta^{CD}\big(2V_{C
L}V_{D\underline{L}}- V_{CD} V_{L\underline{L}} \big).
 \end{multline*}
Since we shown that $V^\infty_{LL}\!=\!0$ it follows that this
expression can be calculated from just knowing $V^\infty_{TU}$
 and by the previous proposition
$V_{LT}^\infty\!=\!\delta^{AB}V^\infty_{AB}\!=\!0$ so
 \begin{equation*}
 n(q^*,\omega)=-P\,(V^\infty,V^\infty)(q^*,\omega)
 =\tfrac{1}{2}\delta^{CD}\delta^{C^\prime
D^\prime} V^\infty_{CC^\prime}(q^*,\omega)
V^\infty_{DD^\prime}(q^*,\omega)\geq 0.
 \end{equation*}
By the previous proposition; for $|\alpha|+k\leq N-7$
 \begin{align*}
  \big|\pa_\omega^\alpha
\big(\langle q^*\rangle \pa_{q^*}\!\big)^k \big[
P\,(V^\infty\!\!,V^\infty)-P\,(V^*\!\!,V^*)\big]\big|
&\les
\frac{\varepsilon^2}{(1\!+r^*)^{\gamma^\prime}\!\!}\,
\frac{1}{\!(1\!+|q^*|)^{2-\gamma^\prime\!}(1\!+q^*_+)^{2\gamma^\prime}\!},
 \\
 \big| \pa_\omega^\alpha \big((1+|\,q^*|)\pa_{q^*}\big)^k
n(q^*,\omega)\big|
&\les
\frac{\varepsilon^2}{(1+|\,q^*|)^{2}(1+q^*_+)^{2\gamma^\prime}}.
 \end{align*}

\begin{prop}
Let $ k_{\mu\nu}$ be the solution with vanishing initial data to
 \beqs\label{eq:einsteinsource}
-\Box^*  k_{\mu\nu}=L_\mu(\omega) L_\nu (\omega)
{n(r^*-t,\omega)}{{r^{*}}^{-2}}\chi\big(\tfrac{\langle\,r^*-t\,\rangle}{t+r^*}\big),
\eqs
where $\chi(s)\!=\!1$, when $|s|\leq\! 1\!/2$ and $\chi(s)\!=\!0$, when $|s|\geq\! 3/4$. We have
\begin{equation*}
\big|\,Z^{*I}\big(\Box^* h_{\mu\nu}-\Box^*  k_{\mu\nu}\big)
\big| \les \frac{\varepsilon^2}{(1+t+r^*)^{2+\gamma-C\varepsilon}
(1+|\,q^*|)^{2-\gamma}},\qquad |I|\leq N-7,
  \end{equation*}
and with $h^{1e}$ as in \eqref{eq:honee} we have for any $\gamma^\prime<\gamma-C\varepsilon$
 \beqs\label{eq:inhomwaveeqdecay10}
\big|\,Z^{*I}\big(h^{1e}_{\mu\nu}- k_{\mu\nu}\big) \big| \les
\frac{\varepsilon^2}{(1+t+r^*)(1+|\,q^*|)^{\gamma^\prime}},
\qquad |I|\leq N-7.
\eqs
\end{prop}
The proof is just an application of Proposition \ref{prop:approxwaveequationschwarzcoord} and Lemma \ref{lem:inhomwaveeqdecay}.
By Lemma \ref{lem:radianfield}
$$
 |\,(\pa_t+\pa_{r^*})\big(\Omega^I {S^*}^J{\partial^*}_t^K r^* (h^{1e}_{\mu\nu}- k_{\mu\nu})\big)|\les
\varepsilon\frac{1}{(1+t+r^*)^{1+\gamma^\prime}}.
$$
for $|I|+|J|+|K|\leq N-7$, and
by Proposition \ref{prop:sourcedecay}
 \begin{equation*}\label{eq:approxsource2}
 \big|\,(\pa_t+\pa_{r^*})\Omega^{*I}S^{*J}\partial_t^{*K}r^*\big( k_{\mu\nu}- k^1_{\mu\nu}\big) \big| \les
\varepsilon\frac{(1+(t-r^*)_+)^{a}}{(1+t+r^*)^{1+a}},
  \end{equation*}
where
\begin{equation*}
 k_{\mu\nu}^1(t,r^*\omega)=L_\mu(\omega) L_\nu(\omega) \int_{r^*-t}^\infty
 \frac{1}{2r^*}\ln{\Big(\frac{t+r^*+q^*}{t-r^*+q^*}\Big)}
  n\big(q^*,\omega\big)\, d q^*\, \chi\big(\tfrac{\langle\,r^*-t\,\rangle}{t+r^*}\big),
\end{equation*}
and $0<a<1$ is number such that
\beqs\label{eq:radiationderbound} {\sum}_{|\alpha|+k\leq N}
|\,(\langle q^*\rangle\pa_{q^*})^k\pa_\omega^\alpha n({q^*},\omega)|\les
{\varepsilon^2}{\langle q^*\rangle^{-1-a}}.
 \eqs
 This is true for any $a\!\leq \!1$,
 in particular for $a=\gamma^\prime$.   It therefore follows that
 $$
 |\,(\pa_t+\pa_{r^*})\big(\Omega^I S^J\partial_t^K r^* (h^{1e}_{\mu\nu}- k^1_{\mu\nu})\big)|\les
\varepsilon\frac{(1+(t-r^*)_+)^{\gamma^\prime}}{(1+t+r^*)^{1+\gamma^\prime}},
$$
for $|I|+|J|+|K|\leq N-7$.
Set
\beqs
H^{1e*}_{\mu\nu}(q^*,\omega,r^*)=r^* (h^{1e}_{\mu\nu}- k^1_{\mu\nu})(r^*-q^*, r^*\omega).
\eqs
We now have an improvement of Proposition \ref{prop:lim}:
\begin{prop}\label{prop:limtwo} The limit
\beqs \label{eq:limit}
H^{1e\infty}_{\mu\nu}(q^*,\omega)
=\lim_{r^*\to\infty}H^{1e\,*}_{\mu\nu}(q^*,\omega,r^*),
\eqs
 exists and satisfies $H_{TU}^{1e\infty}\!=\!H_{UT}^{1e\infty}$, and
 $
H^{1e\infty}_{LT}(q^*\!,\omega)\!=\!\delta^{AB} H^{1e\infty}_{AB}
\textbf{}(q^*\!,\omega)\!=\!0,
$ and
 \begin{align*}
\big|\,\pa_\omega^\alpha \big((1+|\,q^*|)\pa_{q^*}\big)^k
H^{1e\infty}_{\mu\nu}(q^*,\omega)\big|
&\les \varepsilon
(1+q^*_+)^{-\gamma^\prime},\\
\big|\pa_\omega^\alpha
\big((1+|\,q^*|)\pa_{q^*}\big)^k\big[H^{1e*}_{\mu\nu}(q^*,\omega,r^*)
 -H^{1e\infty}_{\mu\nu}(q^*,\omega)\big]\big|
 &\les\varepsilon \Big(\frac{1+q_-^*}{1+t+r^*}\Big)^{\gamma^\prime},
\end{align*}
for $|\alpha|+k\leq N-7$ and
 when $r^*>t/2$. Moreover
 \begin{align*} \big|
(t+r^*)(\pa_t+\pa_{r^*})\pa_\omega^\alpha\big((1+|\,q^*|)
\pa_{q^*}\big)^k\big[H^{1e*}_{\mu\nu}(q^*,\omega,r^*)
\big]\big| &\les
\varepsilon \Big(\frac{1+q_-^*}{1\!+t\!+r^*}\Big)^{\!\gamma^\prime}\!\!\!,\\
 \Big|(\partial_t\!+\partial_{r^*}\!)\big(r\Omega^I S^J\partial_t^K  k^1_{\mu\nu} )
 -\!\frac{1\!}{2r\!\!}\int_{r-t}^\infty \!\!\!\!\!\!\Omega^I(q\partial_q)^J\!
 (-\partial_q)^{K} n(q,\omega)\, dq \Big|
\! &\les \!\frac{\varepsilon(1\!\!+(t\!-\!r)_+)^{\gamma^\prime}\!\!\!}{(1\!+t+r)^{1+\gamma^\prime}\!\!\!}\,.
 \end{align*}
\end{prop}

\section{Interior asymptotics for the wave equation with sources}

The results so far sufficed to prove existence of the radiation
field. To get more precise behavior towards time like infinity we
 use formulas from
\cite{L1}:

\begin{prop}\label{lem:approxsource}
 Let $F$ and $\phi=\Phi[n]$ be as in the Proposition \ref{prop:sourcedecay}
and set
 \beq\label{eq:approxsourcesol2}
\phi_2(x,t)=\Phi_2[n](x,t)=\int_{r-t}^\infty \frac{1}{4\pi}\int{\frac{
n({q},{\omega})}{t+{q}-\langle\, x,{\omega}\rangle}\,
dS({\omega})}\,\chi\big(\tfrac{\langle\,{q}\,\rangle}{t+r}\big) d {q}. \eq
 Then for $|I|+|J|+|K|\leq N$
 \begin{align}\label{eq:errorest}
 |\,\Omega^I S^J\pa_t^K(\phi-\phi_2)|
 &\les \frac{1}{(1+t+r)(1+|\, r-t|)^a},\\
 \label{eq:dererrorest}
 \big|\,(\pa_t+\pa_r)\big( r\, \Omega^I S^J\pa_t^K(\phi-\phi_2)\big)\big|
 &\les \frac{1}{(1+t+r)^{1+a}},\quad 0<a<1.
 \end{align}
\end{prop}
\begin{proof}
 Following \cite{L1} we write the convolution of
 \beqs
F(t,x)= \eta\big(|\,x|-t,x\big)/|\,x{|\,}^2,
 \eqs
  with the fundamental solution $E$ of $\Box$
 $$\phi(t,x)=E*F(t,x)
 =\int_{|x|-t}^\infty \phi_{q}(t,x)\, d{q},
 $$  where
 $$
\phi_{q}(t,x)=\frac{1}{4\pi}  {\int
\frac{\eta({q},\rho\,\omega)}{t+{q}-\langle x,\omega\rangle}
\,dS(\omega)}H(t+{q}-|\,x|).
$$
Here $H(s)=1$, $s>0$ and $0$ otherwise, and
$$\rho=\rho{}({q},\omega)
=\frac{1}{2}\,\frac{(t+{q})^2-r^2}{t+{q}-\langle x,\omega\rangle},\qquad
s=\rho-q,
$$
satisfy
$$
0\leq t+{q}-r\leq 2\rho\leq t+{q}+r,\quad\text{and}\quad
t-r\leq \rho+s\leq t+r.
$$

In our case
$$
\eta({q},y)=n({q},\omega)\psi({q},\rho),\qquad\text{and}\quad
\psi({q},\rho)= \chi\big(\tfrac{\langle {q}\rangle}{s+\rho}\big),
\quad y=\rho\, \omega.
$$
We have
$$
2\langle\,{q}\,\rangle\leq \rho+s\quad\Rightarrow\quad
\quad 2\langle\,{q}\,\rangle\leq t+r.
$$
 Hence
\beq
\chi\big(\tfrac{\langle{q}\rangle}{s+\rho}\big)
=\chi\big(\tfrac{\langle{q}\rangle}{t+r}\big)=1,\quad
\text{when}\quad \tfrac{\langle{q}\rangle}{s+\rho}
\leq \tfrac{1}{2}.\label{eq:vanishingsetone}
\eq
We write
 $$
 \phi(t,x)=\int_{r-t}^{t+r} \frac{1}{4\pi}\int{
 \frac{ n({q},{\omega})}{t+{q}-\langle\, x,{\omega}\rangle}
 \,\psi({q},\rho({q},\omega))\,
dS({\omega})}\,d {q}=\phi_2(t,x)-{\cal E}(t,x),
$$
where
 $$
{\cal E}(t.x)=\int_{r-t}^{t+r} \frac{1}{4\pi}\int{
 \frac{ n({q},{\omega})}{t+{q}-\langle\, x,{\omega}\rangle}
 \big(\chi\big(\tfrac{\langle{q}\rangle}{t+r}\big)
 -\chi\big(\tfrac{\langle{q}\rangle}{s+\rho}\big)\big)\,
dS({\omega})}\,d {q}.
 $$
 It follows from \eqref{eq:vanishingsetone} that the integrand is nonvanishing only if
 $2\langle q\rangle\!\geq 2\rho-q$, i.e.
 $$
 \rho\geq \langle \,{q}\,\rangle + {q}/{2},\qquad \text{i.e.}\qquad
t+{q}-\langle x,\omega\rangle\geq \big({(t+{q})^2-r^2}\big)/
\big({2\langle {q}\rangle +{q}}\big).
  $$

  Hence we want to estimate the integral
 \beqs
 |\,{\cal E}(t,x)|\leq \int_{r-t}^{t+r}
\frac{1}{4\pi}\int {H\big(t+{q}-\langle x,\omega\rangle\geq
\tfrac{(t+{q})^2-r^2}{2\langle {q}\rangle +{q}}\big)}
\, \frac{dS(\omega)}{t+{q}-\langle x,\omega\rangle} \,
\frac{\,d{q}}{\langle q\rangle^{1+a}},
\eqs
where $H(f)$ is the characteristic function of the set where $f\geq 0$.
 We can choose coordinate so that $\langle
 x,\omega\rangle=r\omega_1$ and integrate over the other angular
 variables:
 \beqs
 |\,{\cal E}(t,x)|\leq \int_{r-t}^{t+r}
\frac{1}{2}\int_{-1}^1 H\big(t+{q}-r\omega_1\geq \tfrac{(t+{q})^2-r^2}{
2\langle {q}\rangle +{q}}\big)\frac{d\omega_1}{t\!+{q}\!-r\omega_1} \,
\frac{\, d{q}}{\langle q\rangle^{1+a}},
\eqs
 or if we also make the change of variables $u=t+{q}-r\omega_1$,
\beqs
|\,{\cal E}(t,x)|\leq \int_{r-t}^{t+r}\frac{1}{2}
\int_{t+{q}-r}^{t+{q}+r} H\big(u\geq \tfrac{(t+{q})^2-r^2}{
2\langle {q}\rangle +{q}}\big) \frac{du}{u r} \,
\frac{\, d{q}}{(1+|{q}|)^{1+a}}.
\eqs
We have
\beqs
\big({(t\!+{q})^2\!-r^2}\big)/\big({
2\langle {q}\rangle \!+{q}}\big)\leq t+{q}+r \qquad\Leftrightarrow\qquad
t-r\leq 2\langle \,{q}\,\rangle.\!\!\!\!\!\!
\eqs
If $t-r>2$ then we must also have
$
2|\,q|\geq \sqrt{(t-r)^2-4},
$
in which case
\begin{equation*}
 2|r\,{\cal E}(t,x)|\leq \int_{r-t}^{t+r}\!\!\!
\int_{\frac{(t+{q})^2-r^2}{2\langle {q}\rangle +{q}}}^{t+{q}+r} \frac{du}{u} \!\!
\frac{\, d{q}}{\langle q\rangle^{1+a}}
=\!\!\int_{r-t}^{t+r}\!\!\!\!
\ln{\Big(\frac{2\langle \,{q}\,\rangle +{q}}{t+{q}-r}\Big)}\!\!
\frac{H(2\langle  q\rangle\!\geq t-r)\, d{q}}{\langle q\rangle^{1+a}}.
\end{equation*}
If $|t\!-r|\leq 4$, this integral is bounded. If $q\geq r-t\geq 4$ then $2\langle q\rangle+q\leq 4q$ and if we change variables $q=(r-t)z$ we
see that the integral is bounded by
$$
\frac{1}{|\,t-r|^a}\int_1^\infty \ln{\big(\frac{4z}{z-1}\big)}\frac{dz}{z^{1+a}}\les \frac{1}{|\,t-r|^a}.
$$
If $t-r\!\geq 4$ then $|\,q|\geq\! \sqrt{(t-r)^2-4}/2\geq |\,t-r|/4$ and
the integral
$$
\frac{1}{|\,t-r|^{\,a}}\int_{-1}^\infty \ln{\big(\frac{4|z|}{1+z}\big)}\frac{H\big( |z|\geq 1/4\big) \, dz}{z^{1+a}}\les \frac{1}{|\,t-r|^a}.
$$

This proves \eqref{eq:errorest} for $|I|=|J|=0$ and the fact that we
can get the same bounds for higher $|I|>0$ follows from
rewriting $x=r Qe$, $e=(1,0,0)$, $Q^T Q=I$, and making an angular change of variables
$\omega^\prime=Q\omega$ so that any angular derivative on $\bold{x}$
becomes and angular derivative of $n$ which we control by assumption:
 $$
{\cal E}(t,r Qe)=\frac{1}{r}\int_{r-t}^{t+r} \frac{1}{4\pi}\int{
 \frac{ n(q,{Q\omega})}{(q+t-r)/r+1-\omega_1}
 \big(\chi\big(\tfrac{\langle\,q\,\rangle}{t+r}\big)
 -\chi\big(\tfrac{\langle\,q\,\rangle}{s+\rho}\big)\big)\,
dS({\omega})}\,d q.
$$
The result for the scaling vector fields follows similarly by noting
that a scaling $(t,r)\to (at,ar)$ by a change of variables corresponds to scaling
$n(q,\omega)\to n(aq,\omega)$ so $S(\phi-\phi_2)$ corresponds to replacing
$n(q,\omega)$ by $q\pa_q n(q,\omega)$ plus an error
$$
\partial_a \big(\chi\big(\tfrac{\langle\,aq\,\rangle/a}{t+r}\big)
 -\chi\big(\tfrac{\langle\,aq\,\rangle/a}{s+\rho}\big)\big)\big|_{a=1}
 =-\tfrac{1}{\langle\,q\,\rangle^2}\big(\chi^\prime\big(\tfrac{\langle\,q\,\rangle}{t+r}\big)
 \tfrac{\langle\,q\,\rangle}{t+r}
 -\chi^\prime\big(\tfrac{\langle\,q\,\rangle}{s+\rho}\big)
 \tfrac{\langle\,q\,\rangle}{s+\rho}\big),
$$
decaying faster. The main term in the time derivatives of $F$ is
given by replacing $n(r-t,\omega)$ by $n^\prime(r-t,\omega)$ up to a term of lower order with $\chi\big(\tfrac{\langle\, r-t\,\rangle}{t+r}\big)$ replaced by
$$
\partial_t \chi\big(\tfrac{\langle\,t-r\,\rangle}{t+r}\big)
 =\chi^\prime\big(\tfrac{\langle\,t-r\,\rangle}{t+r}\big)
 \tfrac{1}{t+r}\big(-\tfrac{\langle \,t-r\,\rangle}{t+r} +\tfrac{t-r}{\langle\, t-r\,\rangle}\big),
$$
which decays faster.
Let us now estimate $(\partial_t+\partial_r)(r{\cal E})$.
The terms where derivatives fall on $\chi$ are easily bounded
separately whereas for the main term where $\chi$ is not differentiated we need to use the cancellation between the two cutoff functions, as we did in the estimates above.  Let us assume that
$\bold{x}=r(1,0,0)$ and let us replace $n(q,\omega)$ by $n(q,\omega_1)$
its average over $\omega_2^2+\omega_3^2=1-\omega_1^2$. Writing
$$
\rho=\frac{1}{2}\frac{(t+q)^2-r^2}{t+q-r\omega_1}
=\frac{t+q-r}{2}\,\,\frac{t+q+r}{t+q-r\omega_1}
 $$
 we see that
 \begin{multline*}
 (\pa_t+\pa_r) \rho=\frac{t+q-r}{2}\,\,
 \frac{ 2(t+q-r\omega_1)-(t+q+r)(1-\omega_1)}{(t+q-r\omega_1)^2}\\
 =\frac{(t+q-r)^2(1+\omega_1)}{2(t+q-r\omega_1)^2}
 =\frac{2\rho^2(1+\omega_1)}{(t+q+r)^2}.
 \end{multline*}
 Since $s=\rho-q$ we also have $(\pa_t+\pa_r)(s+\rho)=2(\pa_t+\pa_r)\rho$
 and hence
\begin{equation*}
 (\pa_t+\pa_r)
 \chi\big(\tfrac{\langle {q}\rangle}{s+\rho}\big)
 =\chi^{\,\prime}\big(\tfrac{\langle {q}\rangle}{s+\rho}\big)
 \frac{2\,\langle {q}\rangle }{(s+\rho)^2}
 \frac{2\rho^2(1+\omega_1)}{(t+q+r)^2},
 \end{equation*}
which is bounded by a constat times $\langle\, {q}\,\rangle/(t+r)^2$ in the support of
$\chi^\prime\big(\tfrac{\langle\, {q}\,\rangle}{s+\rho}\big)$:
$$
f\big(\tfrac{q}{t+r},\tfrac{q}{s+\rho},\omega_1\big)= \frac{(t+r)^2}{(s+\rho)^2}
 \frac{2\rho^2(1\!+\omega_1)}{(t+q+r)^2}
 =\frac{\big(1\!+\frac{q}{s+\rho}\big)^2}{\big(1\!+\frac{q}{t+r}\big)^2}\,
 \frac{1\!+\omega_1}{2}
 \leq C,\quad\text{if } \tfrac{\langle {q}\rangle}{s+\rho}\leq \tfrac{3}{4}.
$$

 We have
$$
(\pa_t+\pa_r)(r{\cal E})=(\pa_t+\pa_r) \int_{r-t}^{t+r}\!\! \int_{-1}^{1}\!\!
 \frac{ n(q,\omega_1)\big(\chi\big(\tfrac{\langle\, {q}\,\rangle}{t+r}\big)
 -\chi\big(\tfrac{\langle\, {q}\,\rangle}{s+\rho}\big)\big)\!\!}{(t+q-r)/r+1-\omega_1} \,
 \,\frac{d\omega_1}{2}\,   d q
 ={\cal E}_+ +{\cal E}_+^\prime,
 $$
 where
 \begin{align*}
 {\cal E}_+&=\int_{r-t}^{t+r} \int_{-1}^{1}
 \frac{ n(q,\omega_1)}{\big((t+q-r)/r+1-\omega_1\big)^2} \, \frac{t+q-r}{r^2}\big(\chi\big(\tfrac{\langle\, {q}\,\rangle}{t+r}\big)
 -\chi\big(\tfrac{\langle\, {q}\,\rangle}{s+\rho}\big)\big)\, \frac{d\omega_1}{2}\,  d q,\\
 {\cal E}_+^\prime&=
 -2\int_{r-t}^{t+r} \int_{-1}^{1}\frac{ n(q,\omega_1)\langle\, {q}\,\rangle \big( \chi^\prime\big(\tfrac{\langle\, {q}\,\rangle}{t+r}\big)
 -\chi^\prime\big(\tfrac{\langle\, {q}\,\rangle}{s+\rho}\big) \, f\big(\tfrac{{q}}{t+r},\tfrac{ {q}}{s+\rho},\omega_1\big)\big)}{\big((t+q-r)/r+1-\omega_1\big)(t+r)^2}
 \,
 \, \frac{d\omega_1}{2}\,  d q,
 \end{align*}
 since the boundary term vanishes at $q\!=\!t\!+r$, since $\chi(s)\!\!=\!0$ if $s\!\geq \!1$.
 We claim
 \beqs
|\,{\cal E}_+|+|\,{\cal E}_+^\prime|\les (t+r)^{-1-a},\qquad 0<a<1.
\eqs
 Since
 \beqs
\chi^\prime\big(\tfrac{\langle\, {q}\,\rangle}{s+\rho}\big)=\chi^\prime(\tfrac{\langle\, {q}\,\rangle}{t+r})=0,\quad \text{when}\quad \tfrac{\langle\, {q}\,\rangle}{s+\rho}\leq \tfrac{1}{2}\label{eq:vanishingset}
\eqs
the previous argument applied to $n(q,\omega) (1\!+q^2)^{1/2}$
gives the bound for ${\cal E}_+^\prime$.
The same bound also hold for ${\cal E}_+$ if we estimate
differencess of cutoff functions using
$$
 \frac{1}{s+\rho}-\frac{1}{t+r}
 =\frac{r}{t+r}\frac{2\rho}{s+\rho}\,\,\,\frac{(1-\omega_1)}{t+q-r}.
 $$
This estimate in turn follows since
$$
 t+r-(s+\rho)=t+r+q-\frac{(t+q)^2-r^2}{t+q-r\omega_1}=
 (t+q+r)\frac{r(1-\omega_1)}{t+q-r\omega_1}=\frac{r\rho(1-\omega_1)}{t+q-r}.
$$
We have
$$
(\pa_t-\pa_r)\frac{ n(\tau+r-t,\omega_1)\,\tau\!}{(\tau+r(1-\omega_1))^2}
=-\,\frac{ 2\, n^\prime_q(\tau+r-t,\omega_1)\,\tau\!}{(\tau+r(1-\omega_1))^2}
 +\frac{ 2(1-\omega_1)\, n(\tau+r-t,\omega_1)\,\tau\!}{(\tau+r(1-\omega_1))^3}.
$$
\end{proof}

\section{The asymptotics of the characteristic surfaces}\label{section:characteristicsurfaces}
We solve the eikonal equation
\beq \label{eq:eikonal}
g^{\alpha\beta}\pa_\alpha u\,\, \pa_\beta
u=0,\qquad \text{in}\quad r>t/2>0 ,
\eq
 with asymptotic data as $t\to\infty$:
 \beq
\label{eq:asymptoticdata}
 u\sim \us=t-r^*, \qquad
 r^*=\varrho(r)=r+M\ln{r}+O\big(M/r\big).
 \eq
 Here we have modified the definition of $r^*$ slightly from before:
 \beq\label{eq:rstardef}
 \frac{dr^*}{dr}=\varrho^{\,\prime}(r)=\Big(\frac{1+M/r}{1-M/r}\Big)^{1/2}
 =1+\frac{M}{r}+O\big(M^2/r^2\big),
 \eq
 so that $\us$ is a solution of the eikonal equation for the metric
 $g_0^{\alpha\beta}\!=m^{\alpha\beta}\!+h_0^{\alpha\beta}\!$,
 where $h_0^{\alpha\beta}=-\frac{M}{r}\delta^{\alpha\beta}
 \tilde{\chi}\big(\frac{r}{1+t}\big)$ and $\tilde{\chi}(s)=1$, when $s\!>\!1/2$ and $0$
 when $s\!<\!1/4$:
 \beqs
 g_0^{\alpha\beta} \partial_\alpha \us\, \partial_\beta \us=0,
 \qquad r>t/2.
 \eqs
By the local existence theorem \eqref{eq:eikonal} with data
 $u\!=\!\us$ on $N_T\!=\!\{ t\!=\!T,\, r\!\geq \!T/2 \}$,
has a local solution
that extends to $\{0\!<\!t\!\leq \!T,\,r\!>\!t/2\!+1\}$, as long as
we have bounds for the first order derivatives and
$g^{\alpha\beta}\pa_\alpha u\,\pa_\beta$ is close to
$L$. We simply define $u$ to be constant along the
null geodesics from $N_T$ in the direction
$g^{\alpha\beta}\pa_\alpha u\, \pa_\beta$ (where on $N_T$, $u$ is
determined by data $\us$ and $\partial_t u$ by that this vector
should be null and ingoing backwards). This gives a well defined
solution as long as the derivatives are bounded, which we
control by integrating the system for the first order
derivatives obtained by differentiating the eikonal equation
\eqref{eq:eikonal}.

Let $\gamma(t)=\gamma(t,q^*,\omega)=(t,x(t))$ denote the integral
curve of the vector
field
\beqs
F^\alpha(g,\partial u)
={g^{\alpha\beta}\partial_\beta u}/{g^{0\beta} \partial_\beta u},
\eqs
going through $\gamma(T)\!=\!\gamma_0(T)$, where $\gamma_0(t)=(t,r\omega)$,
$t=\varrho(r)+q^*$\!, is an integral curve of $F(g_0,\us)$,
when $g,u$ is replaced by $g_0,\us$. Set $\tilde{\gamma}(t)\!=\!\gamma(t)\!-\!\gamma_0(t)$.

We will first show that we have uniform bounds independent
of $T$ and then that the solution $u$ will converge
to a limit as $T\!\to\!\infty$ satisfying \eqref{eq:eikonal}-\eqref{eq:asymptoticdata}.

\begin{prop} \label{prop:uniformeikonalbounds}
Let $\tilde{u}\!=\!u\!-\!\us$.
For $r>t/2$ and $0\leq t\leq T$ we have
\begin{align}(1+|q^*|) |\pa \tilde{u}|&\leq
C_1\varepsilon\Big(\frac{1+|q^*|}
{1+t+|q^*|}\Big)^{1-\gamma\prime-\varepsilon}
   \Big( \frac{1+q_-^*}{1+t+|q^*|}\Big)^{\gamma\prime}\!\!,
  \label{eq:uniformeikonalboundone}\\
(1+t+|\,q^*|) |\overline{\pa} \tilde{u}|+|\tilde{u}|&\leq  C_1\varepsilon
 \Big(\frac{1+q_-^*}{1+t+|q^*|}\Big)^{\gamma\prime}\!\!,
 \label{eq:uniformeikonalboundtwo}
 \end{align}
  for some constant $C_1$ independent of $T$.
 Here $q^*\!=r^*\!-t$.
 Moreover,
 \beq\label{eq:tildegammaest}
 |\,\tilde{\gamma}(t,q^*\!,\omega)|\leq C_1 \varepsilon
 \Big(\frac{1+q_-^*}{1\!+t\!+|q^*|}\Big)^{\!\gamma\prime}\!\!.
 \eq
\end{prop}
The convergence is proved by similar estimates for
 differences of solutions with data at $T_{\!1},T_2$.
 The differences satisfy homogeneous equations
 with data at $T\!\!=\min{\!(T_{\!1},T_2)}$ satisfying
 \eqref{eq:uniformeikonalboundone}-\eqref{eq:uniformeikonalboundtwo}
 that tend to $0$ as $T\!\!\!\to\!\infty$,
 see Proposition \ref{prop:chardiff}.

By time reflection we also have a solution $v$ to the eikonal equation
in $r\!\geq \!|t|/2$, $t\!\leq \!0$ satisfying $v\!\sim v^*\!\!=r^*\!\!+t$,
as $t\!\to\!-\infty$ and
$(1\!+r)|\pa \tilde{v}| \!+|\tilde{v}|\!\leq\! (1\!+r)^{-\gamma\prime}\!$ when $t\!=0$,
$\tilde{v}\!=\!v\!-\! v^*\!\!$.
This solution can be extended into the region $ r\!\geq \!t/2\!+\!1,\, t \!>\!0$:

\begin{prop} \label{prop:uniformeikonalboundsv} Suppose that
$(1\!+r)|\pa \tilde{v}| \!+|\tilde{v}|\!\leq\! C_1\varepsilon(1\!+r)^{-\gamma\prime}\!$ when $t\!=0$.
Then
\beq
(1\!+t\!+|q^*|) |{\opa} \tilde{v}|+(1\!+|q^*|) |{\pa} \tilde{v}|+
|\tilde{v}|\les
2C_1\varepsilon
   \Big( \frac{1+ q_-^*}{1\!+t\!+|q^*|}\Big)^{\!\gamma\prime}\!\!\!\!,
 \eq
 for $r\!\geq \!t\!/2\!>\!0$, where $q^*\!\!=r^*\!-t$.
 If $\tilde{\sigma}$ is the corresponding characteristic deviation:
 \beq
\qquad\quad
 |\,\tilde{\sigma}(t,q^*\!,\omega)|\leq C_1 \varepsilon
 \Big(\frac{1+q_-^*}{1+t+|\,q^*|}\Big)^{\gamma^\prime}.
\label{eq:sigmadeviation}
 \eq
\end{prop}



\begin{prop} \label{prop:vectorfieldsutilde} We have
$\!|{Z^*}{}^I \!\tilde{u}|\!+\!|{Z^*}{}^I \!\tilde{v}|
\!\les\!\varepsilon\big(\dfrac{1\!+q_-^*}{1\!\!+\!t\!+\!|q^*|}\big)^{\!\gamma^\prime}\!\!\!\!$,
for $|I|\!\leq \! 2$ and $r\!\geq \!|t|\!/2$.\end{prop}

 We will now derive the system for the derivatives.
Differentiating \eqref{eq:eikonal} gives
\beq\label{eq:diffeikonal}
2g^{\alpha\beta} \partial_\alpha u\, \partial_\beta Zu
= -g_Z^{\alpha\beta}\partial_\alpha u\, \partial_\beta u,
\eq
where the Lie derivative $g_Z^{\alpha\beta}=\mathcal{L}_Z g^{\alpha\beta}$ is given by
\beq\label{eq:liederivative}
g_Z^{\alpha\beta}\partial_\alpha u\,\partial_\beta w=
(Zg^{\alpha\beta}) \partial_\alpha u\,\partial_\beta w
+g^{\alpha\beta} \partial_\alpha u\, [Z,\partial_\beta] w+
g^{\alpha\beta} [Z,\partial_\alpha] u\,\partial_\beta w.
\eq
Hence with the notation $g_Z(U,V)=g_Z^{\alpha\beta}U_\alpha V_\beta$
 \eqref{eq:diffeikonal} respectively \eqref{eq:eikonal} become
\beq\label{eq:eikonalsystem}
\partial_{\widetilde{L}} Z u=-\tfrac{1}{2}g_Z(\partial u,\partial u),\qquad
\partial_{\widetilde{L}}  u=0\qquad
\text{where}\quad
\widetilde{L}^\beta =g^{\alpha\beta}\partial_\alpha u .
\eq

We will now give a sequence of lemmas used to estimate this system.
\begin{lemma} If $Z\in{\mathcal Z}=\{ \Omega_{ij},\partial_t\}$ then
$g_{0Z}(U,V)=0$ and $\pa Z \us=0$.
\end{lemma}
Using the eikonal system \eqref{eq:eikonalsystem} and  the lemma above we obtain:
\begin{lemma} If $Z\!\in\!\{ \Omega_{ij},\partial_t\}$
then with $h_1^{\alpha\beta}\!\!=\!g^{\alpha\beta}\!-g_0^{\alpha\beta}\!\!$ and
$\Ls_{\!\!0\,\,}^\alpha\!\!\!=g_0^{\alpha\beta} \partial_\beta u^{\!*}\!$
we have
\begin{align}
\partial_{\widetilde{L}} Z \tilde{u}
&=-\tfrac{1}{2}h_{1Z}(\partial u,\partial u),\label{eq:eikonalsystemone}\\
\partial_{\Ls_{\!\!0\,}} \tilde{u}+\tfrac{1}{2}g_0(\partial \tilde{u},\partial \tilde{u})
&=-\tfrac{1}{2}h_1(\partial u,\partial u).\label{eq:eikonalsystemtwo}
\end{align}
\end{lemma}
This system gives control of all derivatives of $\tilde{u}$,
first the rotations $\Omega$, by integrating \eqref{eq:eikonalsystemone}
and the derivative along
the outgoing light cones $\Ls_{\!\!0}$ directly from \eqref{eq:eikonalsystemtwo}.
This gives good control of all the tangential
derivatives $\opa$ and then we control the time derivatives by integrating
\eqref{eq:eikonalsystemone}, and hence
all derivatives.

In order to estimate the system we need to  express the above quadratic forms
in a null frame: $S_1,S_2$,
$ \Ls\!\!=\pa_t\!+\pa_{r^*}$,
 $\Lbs\!\!=\pa_t\!-\pa_{r^*}$.
 With respect to the dual frame
 \beq\label{eq:dualstarframe}
 \Ls_{\!\!\alpha}\!\!=-\partial_{\alpha} u^{\!*}\!,\quad
\Lbs_{\!\!\alpha}\!\!=-\partial_\alpha\underline{u\!}^*\!,
\quad\text{where } \underline{u\!}^*\!\!=\!t\!+r^*\!\!, \quad \text{and }
A_i=\delta_{ij}A^j,
\eq
we have with $w_U=U^\alpha w_\alpha$
\beq\label{eq:starnullframederivative}
\pa_\mu\!=-\tfrac{1}{2}\Ls_{\!\!\mu}\pa_\Lbs-\tfrac{1}{2}\Lbs_{\!\!\mu}\pa_\Ls
+A_\mu\pa_A,\quad w_\mu\!=-\tfrac{1}{2}\Ls_{\!\!\mu}w_\Lbs-\tfrac{1}{2}\Lbs_{\!\!\mu}w_\Ls
+A_\mu w_A.
\eq
Previously we defined $k_{UV}\!\!=k_{\alpha\beta}U^\alpha V^\beta\! $,
for $U,V\!\!\in\! \mathcal{N}$.
This is however equal to $k_{UV}\!=k^{\alpha\beta} U_\alpha V_{\!\beta}$,
where $k_{\alpha\beta}\!=m_{\alpha \mu}m_{\beta\nu}k^{\mu\nu}\!$ and
$U_\alpha\! \!=m_{\alpha\beta} U^\beta\!$, $V_{\!\alpha}\!\! =m_{\alpha\beta} V^\beta$
are the corresponding covectors in the dual frame.
We now use the later as definition.
 In particular for
 $\Us\!\!,\Vs\!\in\! \mathcal{N}^*\!$, we define
 $k_{\Us\Vs}\!\!=k^{\alpha\beta}\Us_{\!\!\!\!\alpha} \Vs_{\!\!\!\!\!\beta}$,
 where the corresponding vectors in the dual frame are given by
 \eqref{eq:dualstarframe}. Note that the corresponding covectors
 in the new null frame differs at most by a factor of
 $d r^*\!/dr$ from the original null frame: $U\!-\Us\!\!\!\sim M\!/r$.
 Since the good components decay at most a factor of $1\!/r$ better
  we have the same estimates for $k_{\Us\Vs}$ as we do for $k_{UV}$.
  For some coefficients $k^{\Us\Vs}$ to be calculated we have
 $k^{\alpha\beta}w_\alpha v_\beta=k^{\Us\Vs}W_{\Us}V_{\Vs}$:

\begin{lemma} Let ${k}_{UV}\!\!=k^{\alpha\beta}U_{\!\alpha} \!V_{\!\beta}$
and $W_{\!U}\!=\!U^\alpha w_\alpha\!$.\!
 We have $k^{BD}\!\!=\!\delta^{AB}\delta^{CD}k_{AC}$,
\begin{multline}\label{eq:quadraticinnullframe}
 k^{\alpha\beta}w_\alpha  v_\beta
 \!= \!\big(k_{\Ls\Ls}W_{\!\Lbs} V_{\Lbs}
 +k_{\Ls \Lbs}(W_{\!\Lbs} V_{\Ls} \!+W_{\!\Ls} V_{\Lbs} )
 +k_{\Lbs\Lbs}W_{\!\Ls} V_{\Ls} \big)\!/4\\
-\delta^{AB\!}\big(k_{\Ls\! A}
(W_{\!\Lbs} V_{\!B} \!+W_{\!B} V_{\Lbs})\!
 +k_{\Lbs\!  A}(W_{\!\Ls} V_B \!+W_{\!B} V_{\Ls})\big)\!/2\!
 +k^{BD}W_{\!B} V_D,\!\!
 \end{multline}
 \vspace{-0.2in}
\begin{align}\label{eq:gzeroframe}
 g_0^{\alpha\beta} w_\alpha w_\beta \!
 &=-\tfrac{1}{2}\big(1\!+\!\tfrac{M}{r}\big)
 \big(W_{\Lbs}  W_{\Ls} \!
 +W_{\Ls}  W_{\Lbs} \big)\!
 +\!\big(1\!-\!\tfrac{M}{r}\big)\delta^{AB} W_{\!A}  W_B ,\\
 \label{eq:honeustarv}
g^{\alpha\beta}\pa_{\alpha\!} u\, \partial_\beta
&=-\tfrac{1}{2} g^{\alpha\beta}\pa_{\alpha \!} u \,\Ls_{\!\!\beta }\,\pa_\Lbs
 -\tfrac{1}{2} g^{\alpha\beta}\pa_{\alpha\!} u\, \Lbs_{\!\!\beta}\,\pa_\Ls
\!+\delta^{AB\!} g^{\alpha\beta}\pa_{\alpha\!} u\, A_\beta\,\partial_B.
\end{align}
\end{lemma}

\begin{lemma} If $|\partial \tilde{u}|\!\leq\! 1\!/16$ and $|h_1|\!\leq \!1\!/16$
we have with
$|h_{\Ls \!A}|\!=|h_{\Ls S_1}|\!+|h_{\Ls S_1}|$
\beq\label{eq:lstarv}
|\partial_{{\Ls}} \tilde{u}|\lesssim |h_{1{\Ls}{\Ls}}|
\! +|h_{1{\Ls}\!A}| |\pas \tilde{u}|\!
+ |\pas \tilde{u}|^2\!\!, \quad \text{where}\quad |\pas \tilde{u}|^2\!\!
=\!\!\sum|\overline{\pa}_i \tilde{u}|^2\!\!,
\eq
 \vspace{-0.25in}
\begin{multline}\label{eq:k}
\!\!\!\!|k(\partial u,\partial u)|
\les |k_{{\Ls}{\Ls}}|+ |k_{{\Ls} \!A}| |\pas \tilde{u}|+
 ( |k_{{\Ls}\Lbs}|\!+|k| |h_{1\Ls\Ls}|)
(|h_{1\Ls\Ls}|\!+|h_{1\Ls \!A}|\,|\pas \tilde{u}|)\\
+(|k_{{\Ls} \Lbs}| \!+|k_{AB}|\!+|k| |h_{1\Ls \!A}|)|\pas \tilde{u}|^2\!\!
+|k| |\pas  \tilde{u}|^3\!\!.
\end{multline}
\end{lemma}
\begin{proof} Using \eqref{eq:gzeroframe}, \eqref{eq:eikonalsystemtwo} becomes
\beq\label{eq:Lstarv}
\big(1\!+\tfrac{M}{r}\big) \pa_\Lbs u\,\pa_\Ls \tilde{u}
+\tfrac{1}{2}\big(1\!-\tfrac{M}{r}\big)  |\pas \tilde{u}|^2\!\!
=\!-\tfrac{1}{2}h_1(\pa u,\pa u),
\qquad \pa_\Lbs u\!=2\!+\pa_\Lbs \tilde{u},\!
\eq
from which it follows that $|\pa_\Ls \tilde{u}|\!\leq\!|\pas \tilde{u}|^2\!\!+|h_1(\pa u,\pa u)|$.
 By \eqref{eq:quadraticinnullframe}
\begin{multline}\label{eq:honeu}
|k(\partial u,\partial u)|
\leq |k_{{\Ls}{\Ls}}|\,|\pa_\Lbs u|^2\!/4
+\big(2|k_{{\Ls}\! A}|\,|\pas \tilde{u}|+|k_{{\Ls} \Lbs}| \,
 |\partial_{{\Ls}} \tilde{u}|\big)\,|\pa_\Lbs u|\\
+2|k|\, |\partial_{{\Ls}} \tilde{u}| (|\partial_{{\Ls}} \tilde{u}|+|\pas  \tilde{u}|)
+|k_{AB}|\, |\pas  \tilde{u}|^2.
\end{multline}
\eqref{eq:lstarv} follows from this applied to $h_1$.
\eqref{eq:k} follows from \eqref{eq:honeu}  and \eqref{eq:lstarv}.
\end{proof}

We now turn to estimating the quadratic terms in the right of \eqref{eq:eikonalsystemone}.
\begin{lemma} If $\Omega\!=\!x^i\pa_j-x^j\pa_i$ then with
$k^{\alpha \Omega/r}\!\!=k^{\alpha i} \omega_{j}-k^{\alpha j}
\omega_i$ we have
\begin{align}\label{eq:LieOmega}
(\mathcal{L}_\Omega k)(\pa u,\pa v)&=(\Omega k)(\pa u,\pa v)+k([\Omega,\pa]u,\pa v)
+k(\pa u,[\Omega,\pa]v),\\
 \label{eq:commutarorrotationquadraticform}
k^{\alpha\beta} [\pa_\beta, \Omega] u
&=k^{\alpha \Omega/r}\pa_r u
+\big(k^{\alpha i}
\overline{\pa}_{j}-k^{\alpha j} \overline{\pa}_i\big)u.
 \end{align}
If $|\pa u|\leq 1$ and $|\pa v|\leq 1$ we have
\begin{multline}
|k([\Omega,\partial] u,\pa v)|
\lesssim |k_{{\Ls}\!A}|\!\! +\!|k_{\Ls\Us}| |\pas  {u}|\!
+( |k_{BA}|\!\! +\!|k_{B\Us}| |\pas  {u}|)|\pas {v}|\! \\
+( |k_{{\Lbs}\!A}|\!\! +\!|k_{\Lbs\Us}| |\pas  {u}|) |\pa_\Ls {v}|.
\label{eq:commutarorrotationquadraticformest}
\end{multline}
 \end{lemma}
 \begin{proof} \eqref{eq:commutarorrotationquadraticform} follows from
 $[\pa_k,\!\Omega]\!=\delta_{ki\,}\pa_j \!-\delta_{k\!j\,}\pa_i$ and
 $\partial_i\!=\omega_{i\,} \partial_r \!+\overline{\partial}_i$.
 \end{proof}

\begin{lemma} Suppose that $|\partial \tilde{u}|\leq 1/16$ and $|h_1|\leq 1/16$ and $|Zh_1|\leq 1/16$.
Then
\begin{multline}\label{eq:htestone}
\!\!\!\!\!|h_{1\partial_t}(\partial u,\partial u)|\lesssim
|\partial h_{1{\Ls}{\Ls}}|\! +\big( |\partial h_{1{\Ls} \Lbs}|\!
+|\partial h_1|\,|h_{1{\Ls}{\Ls}}| \big)
( |h_{1{\Ls}{\Ls}}|\!+|h_{1\Ls A}|\,|\pas \tilde{u}|)
\\
\!+\!|\partial h_{1{\Ls}\!A}| |\pas \tilde{u}|\!
 +\!(|\partial h_{1{\Ls} \!\Lbs}|\!+|\partial h_{1AB}|\!
 +|\partial h_1| |h_{1{\Ls}\!A}|) |\pas  \tilde{u}|^2\!\!
 +|\partial h_1| |\pas  \tilde{u}|^3\!\!,
\end{multline}
 \vspace{-0.2in}
\begin{multline}\label{eq:Omegahone}
\!\!\!\!\!\!\!\!|h_{1\Omega}(\partial u,\partial u)|
\!\les \!|\Omega h_{1\Ls\Ls}|\!+|h_{1\Ls \Ts}|\!
+\big(|\Omega h_{1\Ls\!A}|\!+|h_{1AB}|\!+|h_{1\Ls \Us}|\big)|\pas  \tilde{u}|\\
+\big(|\Omega h_{1\Ls \Lbs}|\!+|\Omega h_{1AB}|\!+|h_{1 \Us\!A}|\big)|\pas  \tilde{u}|^2\!\!
+\big(|\Omega h_1|\!+|h_1|\big)|\pas \tilde{u}|^3\!\!.
\end{multline}
\end{lemma}
 \vspace{-0.15in}
\begin{proof}  \eqref{eq:htestone} follows from \eqref{eq:k}.
\eqref{eq:Omegahone} follows from \eqref{eq:k} and
\eqref{eq:commutarorrotationquadraticformest}.
\end{proof}

We are now going to substitute our estimates for $h_1$ into the previous lemma.
Note that the estimates in Proposition \ref{prop:sharpmetricdecay}
hold if we replace $\!L\!$ by $\!\Ls\!$ and $\!\uL$ by $\!\Lbs\!$, since the difference is
$\!\sim \!1\!/r$. Here
$(\Omega h_1)_{\Us\Vs}\!\!\neq\!\Omega(h_{1\Us\Vs\!})$ but the lower order terms generated
are exactly the ones that show up in \eqref{eq:Omegahone}.
(In fact we are estimating the Lie derivative which satisfy the same estimates.)
We have
\begin{lemma} Suppose that $|\pa \tilde{u}|+|h_1|+|\Omega h_1|+|\pa h_1|+M\!\leq c_0$.
Then with $q^*\!\!=r^*\!-t$
\begin{align}
|h_{1\Omega}(\partial u,\partial u)|&\les \frac{\varepsilon }{1\!+t\!+r^*}
 \Big(\frac{1\!+q_-^*}{1\!+t\!+r^*}\Big)^{\!\!\gamma\prime}\!\!\!\!
 +\frac{\varepsilon (1+|\,q^*|)^{-\varepsilon}\,
  |\pas  \tilde{u}|^2}{(1\!+t\!+r^*)^{1-\varepsilon}
 (1\!+q^*_+)^{\gamma\prime}},\label{eq:finalL1}
 \\
|h_{1\partial_t}(\partial u,\partial u)|
&\les
\frac{\varepsilon (1+|q^*|)^{-\varepsilon}}{(1\!+t\!+r^*)^{2-\varepsilon}
(1+q_+^*)^{\gamma\prime}\!\!}\,
 +\frac{\varepsilon (1+|\,q^*|)^{-1-\varepsilon}\,
 |\pas  \tilde{u}|^2}{(1\!+t\!+r^*)^{1-\varepsilon}
 (1\!+q^*_+)^{\gamma\prime}\!\!}\,.
 \label{eq:finalL1partial}
 \end{align}
\end{lemma}

\vspace{-0.3in}
\subsection{Proof of the uniform bounds for $\tilde{u}$ in Proposition
\ref{prop:uniformeikonalbounds}}\label{section:uniformbounds}
Integrating
\beqs\label{eq:eikonaldv2}
  2 \partial_{\widetilde{L}} \Omega \tilde{u}
  =- h_{1\Omega}(\partial u,\partial u),
  \eqs
backwards from $t\!=\!T$ where $\Omega \tilde{u}\!=\!0$ gives an estimate
for $r|\pas  \tilde{u}|\!=c(\sum_{\Omega} |\Omega \tilde{u}|^2)^{\!1/2}$ independent of $T$.
At first we assume that $|\tilde{u}|\leq 1$ so $q^*$ changes at most by $1$ along the
integral curves.
The integral of the first term in \eqref{eq:finalL1} is bounded by
\beq
 \int_{t}^T\frac{\varepsilon }{1+\xi+|q^*|}
 \Big(\frac{1\!+q_-^*}{1\!+\xi+|\,q^*|}\Big)^{\gamma\prime} \, d\xi
 \leq \frac{\varepsilon}{\gamma^\prime}
 \Big(\frac{1\!+q_-^*}{1\!+t+|q^*|}\Big)^{\gamma\prime}.\label{eq:finalL1integral}
\eq
Assuming that $r|\pas  \tilde{u}|$ is bounded by a constant times this,
 the integral of the other term in \eqref{eq:finalL1}
is smaller than half this if $\varepsilon$ is small,
and we get back a better bound which proves the bound by continuity. Dividing by $r$
this proves
\beqs
|\pas  \tilde{u}|\les \frac{\varepsilon}{1\!+t+r^*}
\Big(\frac{1\!+q_-^*}{1\!+t+r^*}\Big)^{\gamma\prime}.
\eqs
Since the same estimate holds for $|h_{1{\Ls}{\Ls}}|$ it
follows from \eqref{eq:lstarv} that also
$|\partial_{{\Ls}} \tilde{u}|$ is bounded by this.
The estimate for $|\partial_t \tilde{u}|$
follows in a similar way integrating
\beqs\label{eq:eikonaldv2}
  2 \partial_{\widetilde{L}} \partial_t \tilde{u}
  =- h_{1\partial_t}(\partial u,\partial u),
  \eqs
  from $t\!=\!T$ where now $|\,\partial_t \tilde{u}|\!\les |\,\partial_{L^{\!*}}\tilde{u}| $.
 Using the estimate for $\pas  \tilde{u}$ we see that both terms in
\eqref{eq:finalL1partial} can be estimate by the first and the
integral is estimated by
\begin{equation*}
\int_{t}^T \!\!\!\!\frac{\varepsilon \, d\xi }
{(1\!+\xi+|q^*|)^{2-\varepsilon}(1\!+|q^*|)^\varepsilon
(1\!+q_+^*)^{\gamma\prime}}\les \frac{\varepsilon }
{(1\!+t+|q^*|)^{1-\varepsilon}(1\!+|q^*|)^\varepsilon
(1\!+q_+^*)^{\gamma\prime}\!}\,.
\end{equation*}

\vspace{-0.25in}
\subsection{Proof of the estimates for $\gamma$ in Proposition
 \ref{prop:uniformeikonalbounds}}\label{section:uniformeikonalbounds}
Recall that $\gamma(t)=\gamma(t,q^*\!,\omega)\linebreak=(t,x(t))$ denotes the
integral curve of the vector
field
\beq
F^\alpha(g,\partial u)
={g^{\alpha\beta}\partial_\beta u}/{g^{0\beta} \partial_\beta u},
\label{eq:flowvectorfield}
\eq
going through $\gamma(T)\!=\!\gamma_0(T)$, where $\gamma_0(t)\!=(t,r\omega)$, $t\!=\varrho(r)\!+q^*$\!\!,
 is an integral curve of $F(g_0,\!\us)$,
when $g,\!u$ is replaced by $g_0,\!\us\!$.
Then $\!\tilde{\gamma}(t)\!=\!\gamma(t)\!-\!\gamma_0(t)$ satisfies
\beq
{d}\tilde{\gamma}^\alpha\!/dt=F(g,\partial u)^\alpha
-F(g_0,\partial \us)^\alpha
={(g^{\alpha\beta}\!-F_0^\alpha g^{0\beta})\,\partial_\beta u}/
{g^{0\beta}\partial_\beta u},
\eq
where $F_0=F(g_0,\partial \us)=(1,\omega/{\varrho^{\,\prime}})$,
 where $\varrho^{\,\prime}\!=dr^*\!/dr$. Here
\beq
(g^{\alpha\beta}\!\!-F_0^\alpha g^{0\beta})\partial_\beta u
\! =\!(h_1^{\alpha\beta}\!\!-F_0^\alpha h_1^{0\beta})\partial_\beta \us\!
 +(h_1^{\alpha\beta}\!\!-F_0^\alpha h_1^{0\beta})\partial_\beta \tilde{u}
 +(g_0^{\alpha\beta}\!\!-F_0^\alpha g_0^{0\beta})\partial_\beta \tilde{u}.
\eq
If $\alpha=0$ this is $0$ and if $\alpha=i>0$ then
$F_0^i=\omega_i /\varrho^{\,\prime}$ so
\beq\label{eq:vectorfieldminusvectorfieldstarg0}
( g_0^{i\beta}\!\!-F_0^i g_0^{0\beta})\partial_\beta  =
\omega^i {\varrho^{\,\prime}}^{-1}
(1\!+\!{M\!}/{r}) \pa_\Ls \!+
\delta^{ij}(1\!-\!{M\!}/{r}) \overline{\partial}_j ,\quad i\!>\!0,
\eq
\beq\label{eq:vectorfieldminusvectorfieldstarh1}
h_1^{i\beta}\!\!-\!F_0^i h_1^{0\beta}\!\!=(\omega^i \omega_k\!
+\!A^i \!A_k )h_1^{k\beta}\!\!
-\omega^i {\varrho^\prime}^{-1} h_1^{0\beta}\!\!
=\omega^i {\varrho^\prime}^{-1} h_1^{\alpha\beta\!}\Ls_{\!\!\alpha}\!\!
+  A^i\! A_k h_1^{k\beta}\!\!.
\eq
Hence we get the following, that integrated with respect to $t$ gives \eqref{eq:tildegammaest},
\beq\label{eq:vectorfieldminusvectorfieldstar}
\big| F(g,\partial u)-F(g_0,\partial \us) \big|
\!\les  |h_{1\, {\Ls} \Ts} |\!+ |h_{1\,\Ts\Us} ||\pa \tilde{u}|\!+|\overline{\partial} \tilde{u}|
\!\les \!\frac{\varepsilon}{1\!\!+t\!+\!|q^*|}
\Big(\frac{1\!+q_-^*}{1\!\!+t\!+\!|q^*|}\Big)^{\!\!\gamma\prime}\!\!\!.
\eq

\vspace{-0.12in}
\subsection{\!\!Proof\! of\! Proposition \!\ref{prop:uniformeikonalboundsv}}\label{section:uniformeikonalboundsv}
\!\!\!\!
We can also commute with scaling $\Ss\!\!\!=t\pa_t\!+r^*\pa_{r^*}$:
\begin{lemma}\label{lem:Liescaling}
The Lie derivative $\mathcal{L}_{\Ss}k^{\alpha\beta}\!=\Ss^{\!}
 k^{\alpha\beta}\!\!-\pa_\gamma \Ss^{\alpha\!}k^{\gamma\beta}\!\!
-\pa_\gamma \Ss^{\beta\!}k^{\alpha\gamma}\!$
satisfy
\begin{multline}
\big(\mathcal{L}_{\Ss}k^{\alpha\beta}\!\!+2k^{\alpha\beta}\!\!-\Ss^{\!}
 k^{\alpha\beta}\big)\pa_\alpha u\, \pa_\beta v
\!=\! -\big(\kappa_1\!-\kappa_2\big)
k^{i\beta}(\opa_i u \,\pa_\beta v+\pa_\beta u\, \opa_i v)\\
-\kappa_2
k^{i\beta}(\pa_i u\,\pa_\beta v+\pa_\beta u\,\pa_i v),\label{eq:Liek}
\end{multline}
where
$\kappa_1\!\!={r^{\!*}}\!/{r}{\varrho^{\,\prime}}\!-\!1$,
$\kappa_2\!\!=\!(1\!+\kappa_1)M\!/(r\!-\!M^{2\!\!}/r)$,
$r^{\!*}\!\!=\!\varrho(r)$.\!
With
$\kappa_3\!\!=\!\kappa_2(1\!-\!M\!\!/r)\!:$
\beq
\mathcal{L}_{\Ss}g_0\!+2 g_0\!
=\kappa_3 g_0\!-2(\kappa_1\!- \kappa_2)\overline{g}_0,
\quad\text{where}\quad\overline{g}_0(\pa u,\pa v)\!=g_0^{ij}\opa_i u\,\opa_{\!j} v.\label{eq:Lieg0}
\eq
\end{lemma}
\begin{proof}
Writing $\Ss\!\!=\Ss^{\alpha}\!\pa_\alpha\!$ we have $\pa_0 \Ss^{i}\!\!=\pa_i\Ss^{0}\!\!=\!0$,
$\pa_0 \Ss^{0}\!\!=\!1$ and using \eqref{eq:rstardef}
\begin{multline*}
\pa_k \Ss^i\!\!=\pa_k\big(\frac{r^*\!\!}{r}\,\frac{\pa r}{\pa r^*\!}\,x^i\big)\!
=\frac{r^*\!\!}{r}\,\frac{\pa r}{\pa r^*\!\!}\,\,\delta_k^{\,\,i}\!
+r\pa_r\big(\frac{r^*\!\!}{r}\,\frac{\pa r}{\pa r^*\!}\big) \omega_k \omega^i\!
=\frac{r^*\!\!}{r}\,\frac{\pa r}{\pa r^*\!\!}\,\,(\delta_k^{\,\,i}\!-\omega_k\omega^i)
+\omega_k\omega^i\\
+\frac{r^*\!\!}{r}\,\frac{\pa r}{\pa r^*\!\!}\,\,
\frac{\!\omega_k\omega^i M\!/r}{1\!-\!M^{2\!}\!/r^2\!\!} \!= \delta_k^{\,i}
+\Big(\frac{r^*\!\!}{r}\,\frac{\pa r}{\pa r^*\!\!}\,
\Big(\!1-\frac{M\!/r}{1\!-\!M^{2\!}\!/r^2\!\!}\,\Big)-1\!\Big)
\big(\delta_k^{\,i}-\omega_k\omega^i\big)\!
+\frac{r^*\!\!}{r}\,\frac{\pa r}{\pa r^*\!\!}\,\,
\frac{\delta_{k}^{\,\,i}M\!/r }{1\!-\!M^{2\!}\!/r^2\!\!}\, .
\end{multline*}
 This proves \eqref{eq:Liek}.
\eqref{eq:Lieg0} follows from this noting that
 $\Ss^{\!} g_0^{\alpha\beta}\!\!=\!{M}(1\!+\kappa_1)\delta^{\alpha\beta}\!\!/r$.
 \end{proof}

We need to estimate the scaling $S^*\tilde{v}$ since we do not get an estimate for $\Ls\tilde{v}$ directly.
It follows from \eqref{eq:commutarorrotationquadraticform} and \eqref{eq:Liek}\eqref{eq:Lieg0}
that with $\widetilde{\underline{L}}\!\!=\!g^{\alpha\beta}\pa_{\alpha \!}v\,\pa_\beta$
\beq \label{eq:vsystem}
 |\widetilde{\underline{L}}\Omega\tilde{v}|
\!\les |\Omega h_1|\!+|h_1|,\qquad  |\widetilde{\underline{L}}\Ss\tilde{v}|
\!\les |\Ss h_1|\!
+|h_1|\!+|\pas \tilde{v}|^2 M \!\ln{r}\!/r.
\eq
There is no gain in components.
Integrating the first equation from $t\!=\!0$ gives
\begin{multline}
|\Omega \tilde{v}|\! \les |\Omega \tilde{v}|_{t\!=\!0}
+\!\!\int_{q_+^{\!*}}^{\,t+r^*} \!\!\!\!\!
\varepsilon\frac{(1\!+\eta)^{-\gamma^\prime-\varepsilon} \, d\eta }
{(1\!+t\!+r^*)^{1-\varepsilon}\!\!}+\!\int_{-q_-^{\!*}}^{\,0} \!\!\!\!
\varepsilon\frac{(1\!+|\eta|)^{-\varepsilon} \, d\eta }
{(1\!+t\!+r^*)^{1-\varepsilon}\!\!}\\
\les
\frac{\varepsilon}{(1\!+t\!+r^*)^{\gamma^\prime}\!\!}
+ \frac{\varepsilon (1\!+q^*_-)^{1-\varepsilon}}
{(1\!+t\!+r^*)^{1-\varepsilon}\!}
\les   \varepsilon \Big( \frac{1+ q_-^*}{1\!+t\!+|q^*|}\Big)^{\!\gamma\prime}\!\!.
\label{eq:vsystemintegral}
\end{multline}
The estimate for $\!|\Ss \tilde{v}|\!$ is the same.
 Integrating $\widetilde{\underline{L}}\pa_t\!$ does not work.
The bounds for $\Lbs \tilde{v}$
and $\tilde{v}$ follows from  $\widetilde{\underline{L}}\tilde{v}\!=\!0$ since $\widetilde{\underline{L}}\!\sim \!g_0^{\alpha\beta\!} \partial_\beta v^{\!*}\!\!\sim
\!\Lbs\!$. As for \eqref{eq:lstarv} we have
\beq\label{eq:Lbarvest}
|\Lbs\tilde{v}|\les |h_1|+|\pas \tilde{v}|^2,\qquad
\text{if}\qquad |\pa \tilde{v}|\!\leq \! 1\!/16.
\eq
As in section \ref{section:uniformeikonalbounds}
 $|d\tilde{\sigma}\!/dt|\les |h_1|+|\pas \tilde{v}|+|\Lbs \tilde{v}|$
and \eqref{eq:sigmadeviation} follows as in \eqref{eq:vsystemintegral}.

\subsection{Proof of Proposition \ref{prop:vectorfieldsutilde}\label{section:vectorfieldsutilde} for $|Z^{*I}\tilde{u}|$}
For $X\!\in \!\mathcal{X}\!=\{\Ss\!\!,\Omega_{ij},\langle q^*\rangle \pa_t\}$ let
$\widehat{X}\!\!=\!X\!-\!\delta_{X \!\Ss}\!$ and
$\widehat{\mathcal{L}}_{\!X}\!\!=\!\mathcal{L}_{\!X}\!\!+\!2\delta_{X\!\Ss}$,
where $\delta_{X\! \Ss}\!\!=\!1$ if $X\!\!=\Ss\!\!$, and $=\!0$ otherwise.
Then $X \big(k(\pa u,\pa v\big)\!=\!( \widehat{\mathcal{L}}_X k) (\pa u,\pa v)
\!+k(\pa\widehat{X} u,\pa v)\!+k(\pa u,\pa\widehat{X}v)\!$ and $\pa \widehat{X}  \us\!\!=\!0$. Since $g(\pa u,\pa u)\!=\!0$ we get $\pa_{\widetilde{L}}\widehat{X}\!\widehat{Z} \tilde{u}\!\!=\!g(\pa u,\pa \widehat{X}\!\widehat{Z} \tilde{u})\!\!
=\!-H(g,\!u)\!/2$, where $\widetilde{L}^{\!\mu\!\!}=g^{\mu\nu\!}\partial_\nu u$ and
\begin{multline*}
\!\!\!\!\! H(g, \!u)\!=\!
\widehat{\mathcal{L}}_{\!X }\widehat{\mathcal{L}}_Z g(\pa u,\pa u)
+2\widehat{\mathcal{L}}_{\!X }g(\pa u, \pa \widehat{Z} \tilde{u})
+2\widehat{\mathcal{L}}_Z g(\pa u, \pa \widehat{X} \tilde{u})
+2 g(\pa\widehat{Z}\tilde{u},\pa\widehat{X}\tilde{u})\qquad\qquad\\
\qquad\quad=\!
\widehat{\mathcal{L}}_{\!X }\widehat{\mathcal{L}}_Z g(\pa u^{\!*}\!\!,\pa u^{\!*})\!
+2\widehat{\mathcal{L}}_{\!X }\widehat{\mathcal{L}}_Z g(\pa u^{\!*}\!\!,\pa \tilde{u})\!
+2\widehat{\mathcal{L}}_{\!X }g(\pa u^{\!*}\!\!, \pa \widehat{Z} \tilde{u})\!
+2\widehat{\mathcal{L}}_Z g(\pa u^{\!*}\!\!, \pa \widehat{X} \tilde{u})\\
\qquad\quad\quad+
\widehat{\mathcal{L}}_{\!X }\widehat{\mathcal{L}}_Z g(\pa \tilde{u},\pa\tilde{ u})
+2\widehat{\mathcal{L}}_{\!X }g(\pa \tilde{u}, \pa \widehat{Z} \tilde{u})
+2\widehat{\mathcal{L}}_Z g(\pa \tilde{u}, \pa \widehat{X} \tilde{u})
+2 g(\pa\widehat{Z}\tilde{u},\pa\widehat{X}\tilde{u}).
\end{multline*}
Here $\widehat{\mathcal{L}}_X g
\!=\!\widehat{\mathcal{L}}_X g_0\!+\widehat{\mathcal{L}}_X h_1$, where
$\widehat{\mathcal{L}}_\Omega g_0\!=\!0$ and
$\widehat{\mathcal{L}}_{\Ss} g_0\!=\!\kappa_3 g_0\!-2(\kappa_1\!-\kappa_2)\overline{g}_0$.
Here $\kappa_1\!\sim \! M\! \ln{r}\!/r$, $\kappa_2\!\sim \! \kappa_3\!\sim \!M\!/r$ and $\overline{g\overline{}^{}}_0(\pa u,\pa v)\!={g}_0^{ij}\pas_{\!i } u\,\pas_{\!\!j}v$.
If we apply this to $g_0(\pa u^*\!\!,\pa u^*)\!=\!0$ we get
$H(g_0,u^*)\!=(\widehat{\mathcal{L}}_X \widehat{\mathcal{L}}_Z g_0)(\pa u^*\!\!,\pa u^*)\!=0$. Moreover, $\overline{g}_0(\pa u^*\!,\pa w)\!=0$
so $\widehat{\mathcal{L}}_{X} \overline{g}_0(\pa u^*\!,\pa w)
=-\overline{g}_0(\pa \widehat{X} u^*\!,\pa w)
-\overline{g}_0(\pa u^*\!,\pa \widehat{X} w)\!=0$.  It follows that
$ |\widehat{\mathcal{L}}_X^I g_0(\pa u^*\!,\pa w)|\les \kappa_3 |\pa_{\Ls} w|$, for $|I|\!\leq \!2$. Hence
\begin{multline}
\!\!\!\!\! |H(g_0,u)|\les (\kappa_3\!+\!|\pa\tilde{u}|\!+\!|\pa\widehat{X}\tilde{u}|\!+\!|\pa \widehat{Z}\tilde{u}|) (|\pa_\Ls\tilde{u}|\!+\!|\pa_\Ls \widehat{X}\tilde{u}|\!+\!|\pa_\Ls \widehat{Z}\tilde{u}|)\\
+(|\pas\tilde{u}|\!+\!|\pas\widehat{Z}\tilde{u}|)
(|\pas\tilde{u}|\!+\!|\pas \widehat{X}\tilde{u}|),\label{eq:Hg0u}
\end{multline}
\vspace{-0.35in}
\begin{multline}
\!\!\!\!\! |H(h_{\!1}, \!u)|\les
|(\widehat{\mathcal{L}}_{\!X }\widehat{\mathcal{L}}_Z h_{\!1})_{\Ls\Ls}|\\
+(|(\widehat{\mathcal{L}}_{\!X }\widehat{\mathcal{L}}_Z h_{\!1})_{\Ls\Ts}|
+|(\widehat{\mathcal{L}}_{\!X }h_{\!1})_{\Ls\Ts}|
+|(\widehat{\mathcal{L}}_Z h_{\!1})_{\Ls\Ts}|)(|\pab\tilde{u}|\!+\!|\pab \widehat{X}\tilde{u}|\!+\!|\pab\widehat{Z}\tilde{u}|)\\
+(|(\widehat{\mathcal{L}}_{\!X }\widehat{\mathcal{L}}_Z h_{\!1})_{\Ls\Lbs}|\!
+\!|(\widehat{\mathcal{L}}_{\!X }h_{\!1})_{\Ls\Lbs}|\!
+\!|(\widehat{\mathcal{L}}_Z h_{\!1})_{\Ls\Lbs}|)(|\pa_\Ls\tilde{u}|\!+\!|\pa_\Ls \widehat{X}\tilde{u}|\!+\!|\pa_\Ls \widehat{Z}\tilde{u}|)\!\!\!\!\!\!\!\!    \\
+(|\widehat{\mathcal{L}}_{\!X }\widehat{\mathcal{L}}_Z h_{\!1}|
+|\widehat{\mathcal{L}}_{\!X }h_{\!1}|
+|\widehat{\mathcal{L}}_Z h_{\!1}|+| h_{\!1}|)
(|\pab \tilde{u}|\!+\!|\pab \widehat{X}\tilde{u}|)
(|\pab \tilde{u}|\!+\!|\pab \widehat{Z}\tilde{u}|).\label{eq:Hh1u}
\end{multline}
We have $\pa_{\widetilde{L}}( \widehat{X} \tilde{u})=g(\pa u,\pa \widehat{X} \tilde{u})\!
=-\widehat{\mathcal{L}}_X g(\pa u,\pa u)/2$, where
\begin{multline*}
\widehat{\mathcal{L}}_X g(\pa u,\pa u)
=2\widehat{\mathcal{L}}_X g_0(\pa u^{\!*}\!\!,\pa \tilde{u})
+\widehat{\mathcal{L}}_X g_0(\pa \tilde{u},\pa \tilde{u})\\
+\widehat{\mathcal{L}}_X h_1(\pa u^*\!\!,\pa u^*)
+\widehat{\mathcal{L}}_X h_1(\pa u^*\!\!,\pa \tilde{u})
+\widehat{\mathcal{L}}_X h_1(\pa \tilde{u},\pa \tilde{u}).
\end{multline*}
Using that $|\pa \tilde{u}|\!+|h_1|\!+|Xh_1|\!+|\widehat{\mathcal{L}}_X h_1|\!\leq \!c_0$
and $|\partial_{{\Ls}} \tilde{u}|\!\lesssim |h_{1{\Ls}{\Ls}}|
\! +|h_{1{\Ls}\!A}|\, |\pas \tilde{u}|\!
+ |\pas \tilde{u}|^2$
\begin{multline*}
\!\!\!\!\! |\widehat{\mathcal{L}}_X g(\pa u,\pa u)|\les \kappa_3 (|\pa_{\Ls} \tilde{u}|+|\pas \tilde{u}|^2)
+|(\widehat{\mathcal{L}}_X h_1)_{\Ls\Ls}|
+|(\widehat{\mathcal{L}}_X h_1)_{\Ls \Ts}| |\pas \tilde{u}|\\
\qquad\qquad\qquad\quad +|(\widehat{\mathcal{L}}_X h_1)_{\Ls \Lbs}| |\pa_\Ls \tilde{u}|\!
+|(\widehat{\mathcal{L}}_X h_1)_{\Ts\Us}| |\pas \tilde{u}|^2\!\!
+|(\widehat{\mathcal{L}}_X h_1)_{\Lbs\Lbs}| |\pa_\Lbs\tilde{u}|^2\\
\les |(\widehat{\mathcal{L}}_X h_1)_{\Ls \Ts}|
+|h_{1\Ls\Ts}| +|\pas \tilde{u}|^2
\end{multline*}
By \eqref{eq:honeustarv} $\pa_{\widetilde{L}}\!=F^\Ls \pa_{\Ls}\! +F^\Lbs \pa_{\Lbs}\!+F^A\pa_{A}$, where
$|F^\Ls \!\!\!-1|\!\leq\! 1\!/2$,
$|F^{A}|\!\les |\pas \tilde{u}|\!+|h_{1\Ls \!A}|\!+|\pa_\Ls \tilde{u}|
\!\les \!|\pas \tilde{u}|\!+|h_{1\Ls \Ts}|$ and $|F^\Lbs|\!\les |h_{1\Ls \Ts}|\!+|\pa_\Ls \tilde{u}|\!\les \!|h_{1\Ls \Ts}|\!+|\pas \tilde{u}|^2 $ so
\beq
|\pa_{\Ls} \widehat{X}\tilde{u}|\les  |(\widehat{\mathcal{L}}_X h_1)_{\Ls \Ts}|
+(|h_{1\Ls\Ts}| +|\pas \tilde{u}|^2)(1\!\!+ \!|\pa \widehat{X}\tilde{u}|)+|\pas \tilde{u}| |\pas \widehat{X}\tilde{u}|.\label{eq:LsXu}
\eq
If we make the inductive assumption  $|\pa \widehat{X}\tilde{u}|\!\leq \!c$ above we get from \eqref{eq:Hg0u}-\eqref{eq:LsXu}
\beq
|\pa_{\widetilde{L}}(\widehat{X}\widehat{Z} \tilde{u})|
\les {\sum}_{|I|\leq 2}|(\widehat{\mathcal{L}}_X^I h_1)_{\Ls \Ts}|
+(|\pas \tilde{u}|+|\pas \widehat{X}\tilde{u}|)(|\pas \tilde{u}|+|\pas \widehat{Z}\tilde{u}|),
\eq
and using the estimates from section 6 and the fact that
$r|\pas  v|\!=c(\sum_{\Omega} |\Omega v|^2)^{\!1/2}$
\beq
|\pa_{\widetilde{L}}(\widehat{X}\widehat{Z} \tilde{u})|\les \frac{\varepsilon }{1\!+t\!+r^*}
 \Big(\frac{1\!+q_-^*}{1\!+t\!+r^*}\Big)^{\!\!\gamma\prime}\!\!\!\!
 +{\sum}_{\Omega}\!\!\! \frac{
  (|\Omega\tilde{u}|+|\Omega\widehat{X}\tilde{u}|+|\Omega \widehat{X}\tilde{u}|)^2\!\!}{(1\!+t\!+r^*)^{2}
 }.\label{eq:finalL1two}
\eq
The first term is the same as in \eqref{eq:finalL1}-\eqref{eq:finalL1integral} and there is  room in the second.

\vspace{-0.15in}
\subsection{Proof of Proposition \ref{prop:vectorfieldsutilde} for $|Z^{*I}\tilde{v}|$} To generalize the bounds in section \ref{section:uniformeikonalboundsv} to two vector fields
we replace
$u$, $u^*\!\!$, $\tilde{u}$, $\Ls$ and $\Lt$ in section \ref{section:vectorfieldsutilde} by
$v$, $v^*\!$, $\tilde{v}$, $\Lbs\!$ and $\Lbt$ respectively.
We have $|H(h_{1}, v)|\!\les \! |\widehat{\mathcal{L}}_{\!X }\!\widehat{\mathcal{L}}_{\!Z }h_{1}|
\!+\! |\widehat{\mathcal{L}}_{\!X }\!h_{1}|\!+ \!|\widehat{\mathcal{L}}_{\!Z }h_{1}|
\!+\!|h_1|$ and
\begin{multline*}
\!\!\!\!\! |H(g_0,v)|\les (\kappa_3\!+\!|\pa\tilde{v}|\!+\!|\pa\widehat{X}\tilde{v}|\!+\!|\pa \widehat{Z}\tilde{v}|) (|\pa_\Lbs\tilde{v}|\!+\!|\pa_\Lbs \widehat{X}\tilde{v}|\!+\!|\pa_\Lbs \widehat{Z}\tilde{v}|)\\
+(|\pas\tilde{v}|\!+\!|\pas \widehat{Z}\tilde{v}|)
(|\pas\tilde{v}|\!+\!|\pas \widehat{X}\tilde{v}|),
\end{multline*}
where $
|\pa_{\Lbs} \widehat{X}\tilde{v}|\les  |\widehat{\mathcal{L}}_X h_1|\!
+|h_{1}|\! +|\pas \tilde{v}|^2 \!+|\pas \tilde{v}| |\pas \widehat{X}\tilde{v}|.
$
Hence
\beqs
|\pa_{\underline{\widetilde{L}}}(\widehat{X}\widehat{Z} \tilde{v})|
\les
|\widehat{\mathcal{L}}_{\!X }\!\widehat{\mathcal{L}}_{\!Z }h_{1}|
\!+\! |\widehat{\mathcal{L}}_{\!X }\!h_{1}|\!+ \!|\widehat{\mathcal{L}}_{\!Z }h_{1}|
\!+\!|h_1|
+(|\pas \tilde{v}|\!+|\pas \widehat{X}\tilde{v}|)(|\pas \tilde{v}|\!+|\pas \widehat{Z}\tilde{v}|).
\eqs
This is of the form \eqref{eq:vsystem} and can be integrated
as in \eqref{eq:vsystemintegral}. This gives the estimate in Proposition
\ref{prop:vectorfieldsutilde} for $X\!Z\tilde{v}$ if $X,\!Z\!\in \!\{S^*\!\!,\Omega\}$
and for $\langle q^*\rangle |\Lbs \!X\tilde{v}|$.
To prove it  for $\langle q^*\rangle^2 |\Lbs^2 \tilde{v}|$ we apply $\Lbs$ to
$\pa_{\widetilde{\underline{L}}}\tilde{v}\!=\!0$.
By \eqref{eq:honeustarv} $\pa_{\widetilde{\underline{L}}}\!=F^\Ls\! \pa_{\Ls}\! +F^\Lbs \! \pa_{\Lbs}\!+F^A\pa_{A}$, where $|F^\Lbs\!\! -\!1|\!\leq\! 1\!/2$ so
$|\pa_\Lbs^2\tilde{v}|\!\les\! |\pa_\Lbs F^{T^{\!*}}||\pa_{T^*}\tilde{v}|\!+
|F^{T^{\!*}}\!||\pa_{T^{\!*}} \pa_\Lbs \tilde{v}|\!+|\pa_\Lbs F^{\Lbs}||\pa_{\Lbs}\tilde{v}|$
and we repeat it for the terms containing $\pa_\Lbs$. The terms generated
 decay at least like $(1\!+t)^{-\gamma\prime}\langle q^*\rangle^{-2+\gamma\prime}$ so multiplying by
$\langle q^*\rangle^2$ gives the desired bound.

\vspace{-0.15in}
\subsection{Convergence}
Suppose that $u_1$ and $u_2$ are solutions such that $u_i\!=\us$ when $t\!\!=\!T_i$,
for $i\!=\!1,2$. Let $X_i(t)\!\!=\!X_i(t,q^*\!,\omega)\!=(t,x_i(t))$ be the
integral curve of
$F(g,u_i)$ in \eqref{eq:flowvectorfield}, going through $X_i\!=\! X_0$, when $t\!=T_i$,
where $X_0(t)\!=(t,r\omega)$, $t\!=\varrho(r)\!+q^*\!\!$,
 is an integral curve of $F(g_0,\us)$.
 Then $u_1\!=\!u_2\!=q^*$ along the curves, but $\partial u_1\!\neq \pa u_2$.
 Let $W_i$ denote the covector
 $w_{i\alpha}\!(t)\!=\!(\partial_\alpha u_i)\circ X_i(t)$ expressed
 in the frame $\mathcal{N}^*\!\!=\!\{ \Lbs,\Ls,S_1,S_2\}$,
 i.e $W_{iU^*}\!\!=w_{i\alpha} U^{*\,\alpha}\!\!$, for $U^*\!\in\! \mathcal{N^*}\!$.
 Similarly, let $W_i^*$ be the covector
 $w_{i\alpha}^*\!(t)\!=(\partial_\alpha u^*)\!\circ \! X_i(t)$ expressed in this frame.
 We also write $W\!\!=\!(W_{\Lbs},\overline{W})$, where
 $\overline{W}\!\!=\!(W_{\Ls},\sls{W\!}\,)$ is the tangential part, and
 $\sls{W\!}\!=(W_{S_1},W_{S_2})$.
The frame is chosen so that $W_{i{\Lb}^*}^*\!=2$ and $\overline{W}_{\!i}^*\!\!=0$.
 Let
 $W_{Z}\!=\!Z^\alpha  w_{\alpha}$.
 We rewrite
 \begin{align}\label{eq:charcurvei}
 d X^\alpha\!/dt&=F^\alpha\big(g(X),W\big)
 ={ g^{\alpha U^*}(X) W_{U^*}}/{g^{0 U^*}(X)W_{U^*} },\\
 \label{eq:eikonaldvi}
  d W_{\!Z} /dt  &=H_Z\big(g(X),W\big)
  =-{h_{1Z}^{U^* V^*}(X)W_{U^*}W_{V^*}}/{g^{0 U^*}(X) W_{U^*}},
\end{align}
for  $Z=\Omega_{ij},\, \partial_t$.
Convergence as $T\to\infty$ follows from:
\begin{prop} \label{prop:chardiff}Let $0<\gamma^{\prime\prime}<\gamma^\prime$.
Then if $\varepsilon>0$ is sufficiently small we have for
$r>t/2$ with constants independent of $T=\min{(T_1,T_2)}$,
\begin{align}
(1+|q^*|) |W_2-W_1|&\leq  C_2\varepsilon
 \Big( \frac{1+q_-^*}{1+T+|\,q^*|}\Big)^{\gamma^\prime\!-\gamma^{\prime\prime}}\!\!\!
 \Big( \frac{1+q_-^*}{1+t+|\,q^*|}\Big)^{\gamma^{\prime\prime}}\!\!,
 \label{eq:chardiffone}\\
(1+t+|q^*|) |\,\overline{W}_2-\overline{W}_1|&\leq  C_2\varepsilon
 \Big( \frac{1+ q_-^*}{1+T+|\,q^*|}\Big)^{\gamma^\prime\!-\gamma^{\prime\prime}}\!\!\!
 \Big( \frac{1+ q_-^*}{1+t+|\,q^*|}\Big)^{\gamma^{\prime\prime}}\!\!,
\label{eq:chardifftwo}\\
 |X_2-X_1|&\leq C_3\varepsilon
 \Big( \frac{1+ q_-^*}{1+T+|\,q^*|}\Big)^{\gamma^\prime\!-\gamma^{\prime\prime}}\!\!\!
 \Big( \frac{1+ q_-^*}{1+t+|\,q^*|}\Big)^{\gamma^{\prime\prime}}\!\!.
 \label{eq:chardiffthree}
 \end{align}
\end{prop}

Note that $h_1^{\alpha U^*}\!\!\!$ satisfy the same estimates as the corresponding
components in the Minkowski null frame
$h_1^{\alpha U}\!$ since the difference $U\!-\!U^{{}_{\!}*}\!\!\sim \!1\!/r$.
By \eqref{eq:tildegammaest}
$X_i(t)\!\in \! B\big(X_0(t),C_{\!1}\varepsilon\big)$, the ball of radius
$C_{\!1}\varepsilon\!$ at $\!X_0(t)\!$ so estimating with $\!L^{\!\infty\!}$ norm
\beq\label{eq:differenceestimate}
\big| h\circ X_2(t)-h\circ X_1(t)\big|
\leq \| \pa h(t,\cdot)\|_{ B(X_0(t),C_{\!1\,}\varepsilon)}\, |X_2(t)-X_1(t)|.
\eq

\begin{lemma}\label{lem:Fdiff} We have
\begin{multline}
\big| F(X_2,W_2)-F(X_1,W_1)\big|
\leq \frac{C_0\varepsilon
|\,X_2-X_1|}{(1+t+|\,q^*|)^{2-\varepsilon}
 (1+|\,q^*|)^{\varepsilon}} \\
  +2|\overline{W}_2-\overline{W}_1|
 +\frac{C_0\varepsilon|\,W_2-W_1| }{1+t+|\, q^*|}
 \Big(\frac{1+q_-^*}{1+t+|\,q^*|}\Big)^{\!\gamma\prime}\!\!.\label{eq:Fdifftwo}
\end{multline}
\end{lemma}
\begin{proof} Let
$g_i=g\circ X_i$, $g_i^*=g_0\circ X_i$ and $h_{1i}=h_1\circ X_i$.
We have
\begin{align}
F^\alpha(g,W)-F^\alpha(g,W\!-\!V)
\!&=\!
\big(g^{\alpha\beta}\!-F^\alpha(g,\!W)g^{0\beta}\big)V_\beta/g^{0\beta}(W\!-\!V)_\beta,
\label{eq:vectorfieldvectordifference}\\
F^\alpha(g\!+\!h,W)-F^\alpha(g,W)&=
\big(h^{\alpha\beta}\!-F^\alpha(g,W)h^{0\beta}\big)W_\beta/(g\!+\!h)^{0\beta}W_\beta.
\label{eq:vectorfieldmetricdifference}
\end{align}
We want to use the above to estimate
\beq
F(g_2,W_2)-F(g_1,W_1)\!=F(g_2,W_2)-F(g_2,W_1)+F(g_2,W_1)-F(g_1,W_1).
\eq
We can replace $g_2-g_1$ by $\delta h_1\!=h_1\circ X_2-h_1\circ X_1$
since $g_2^*-g_1^*$ is bounded by $CM r^{-2} |X_2-X_1|$.
Moreover, when applying
\eqref{eq:vectorfieldvectordifference}-\eqref{eq:vectorfieldmetricdifference}
 we can replace $F(g_i,W_i)$ in the right
by $F(g_i^*,W_i^*)$ since the error is bounded by:
\beqs
|F(g_2,W_2)-F(g_2^*,W_2^*)|\, |W_2-W_1|+
|F(g_1,W_1)-F(g_1^*,W_1^*)|\, |g_2-g_1| ,
\eqs
where $F(g_i,W_i)-F(g_i^*,W_i^*)$ is controlled by
\eqref{eq:vectorfieldminusvectorfieldstar}, and
modulo lower order $g_2-g_1$ by $\delta h_1$, which by \eqref{eq:differenceestimate}
 is controlled by
$\|\pa h_1\|_{B(X_0,C_1\varepsilon)}$ times $\delta X=X_2-X_1$.
With $\delta W\!\!=W_{\!1}\!-W_2$ it therefore remains to estimate
\beqs
\big(g_2^{\alpha\beta}\!-F^\alpha(g^*_2,\!W_2^*)g_2^{0\beta}\big)\delta W_\beta
\quad\text{and}\quad
\big(\delta h_1^{\alpha\beta}\!-F^\alpha(g^*_1,W^*_1)\delta h_1^{0\beta}\big)W_{1\beta}.
\eqs
Moreover, using \eqref{eq:vectorfieldminusvectorfieldstarh1} we see using
\eqref{eq:quadraticinnullframe} that modulo errors
controlled by $|h_{1LT}|\, |\delta W|\!+|h_{1T\Lbs}| \, |\delta \overline{W}|$
respectively $|\delta h_{1LT}|\, |W|\!+|\delta h_{1T\Lbs}| \, | \overline{W}|$
we can replace $g_2$ by $g^*_2$ respectively $W_1$ by $W^*_1$ so we are left with the
main parts
\beqs
\big(g_2^{\!*\,\alpha\beta}\!-F^\alpha(g^*_2,\!W_2^*)g_2^{\!*\,0\beta}\big)\delta W_\beta
\quad\text{and}\quad
\big(\delta h_1^{\alpha\beta}\!-F^\alpha(g^*_1,W^*_1)\delta h_1^{0\beta}\big)W^*_{1\beta}.
\eqs
Using \eqref{eq:vectorfieldminusvectorfieldstarg0} the first can be estimated
by $\delta \overline{W}$ and using
\eqref{eq:vectorfieldminusvectorfieldstarh1} the second can be estimated
by $\delta h_{1LT}$, which in turn is bounded by
$\|\pa h_{1LT}\|_{B(X_0,C_1\varepsilon)}\,|\delta X|$.
\end{proof}

\begin{lemma}\label{lem:WL} If $|W_i-W_i^*|\leq 1/4$ and $|h_1|\leq 1/16$ then
\begin{equation}\label{eq:WL}
|W_{2{\Ls}}\!-W_{\!1{\Ls}}|
\!\les \frac{C_0\varepsilon}{1\!+t}
\Big(\frac{1\!+q_-^*}{1\!+t}\Big)^{\!\!\gamma\prime} \!\!|W_2\!-W_1|
 +\frac{C_0\varepsilon(1\!+q_+^*)^{-\gamma\prime}
  |X_2\!-\!X_1|}
 {(1\!+t)^{2-\varepsilon}(1\!+\!|\,q^*|)^{\varepsilon}} .
\end{equation}
\end{lemma}
\begin{proof} Let $\widetilde{W}_i=W_i-W_i^*$. By \eqref{eq:Lstarv}
\beqs
\big(1+\tfrac{M}{r}\big)(2+ \widetilde{W}_{i\Lbs}) W_{i\Ls}
=- 2h_{1i}^{\Lbs\Lbs}\!- 2h_{1i}^{\Lbs\Us}\widetilde{W}_{i\Us}
-\tfrac{1}{2}h_{1i}(\widetilde{W}_i,\widetilde{W}_i)
-\tfrac{1}{2}\big(1-\tfrac{M}{r}\big)  |\sls{W\!}_{\!i}|^2,
\eqs
where $h_{1i}=h_1\!\circ\!X_i$. Subtracting the case $i=2$ from $i=1$ we see using
\eqref{eq:quadraticinnullframe} that
the difference of the linear terms in the right is bounded by
\begin{multline}
(|h_{11\Ls\Lbs}|+|h_{12\Ls\Lbs}|)|\delta W_{\Ls}|
+(|h_{11\Ls \Ts}|+|h_{12\Ls\Ts}|)|\delta W|\\
+|\delta h_{1\Ls \Ts}|+(|W_{1\Ls}|+|W_{2\Ls}|)|\delta h_{1\Ls \Lbs}|,
\end{multline}
where $\delta W\!\!=W_2-W_{\!1}$ and $\delta h_1\!=h_{12\!}-h_{11}$.
The first term is small compared to the left of \eqref{eq:WL}
and the second is bounded by the first term in the right.
By \eqref{eq:differenceestimate} the term
$|\delta h_{1\Ls\Ts}|\!\leq \|\pa h_{1\Ls\Ts}\|_{B(X_0,C_1\varepsilon)}\, |\delta X|$
is bounded by the second term in \eqref{eq:WL} and so is the last.
The difference in the quadratic terms
is bounded
\beqs
(|\widetilde{W}_{\!1}|+|\widetilde{W}_2|)|W_{2\Ls}-W_{\!1\Ls}|+\big(|\overline{W}_{\!\!1}|\!+|\overline{W}_{\!2}|
+|h_{11 \Ls \Ts}||\widetilde{W}_{\!1}|+|h_{12\Ls \Ts}|\,|\widetilde{W}_2|\big)
|W_2-W_{\!1}|,
\eqs
where the first term is small compared to the left of \eqref{eq:WL} and the second
can be estimated by the first term on the right of \eqref{eq:WL}.
In the above we neglected the change in the coefficients $M\!/r$ which is bounded
by $M\!/r^2$ times $\delta X$ multiplied by $|\overline{W}|$
which is bounded by the second term in the right of \eqref{eq:WL}.
\end{proof}

We will estimate $\delta H_Z\!=H_Z(X_2,\!W_2)\!-\!H_Z(X_{\!1},\!W_{\!1})$
by the $L^{\!\infty}\!$ norms over ${D_{\!\varepsilon_{\!}}(t)}$:
\beq
|\delta H_{\!Z}(t)|\!\leq\!
\big\|\frac{\pa H_{\!Z}\!\!\!}{\pa X\!}\,(t,\cdot)\big\|_{D_{\varepsilon_{\!}}(t)} |\delta X(t)|
+\!{\sum}_{\Us\in \mathcal{N}^{\!*}}
\big\|\frac{\pa H_{\!Z}\!\!\!}{\pa W_{\!\Us}\!\!\!\!\!}\,(t,\cdot)\big\|_{D_{\varepsilon_{\!}}(t)}
|\delta W_{\!\Us}(t)|,
\eq
where
$\!{D_{\!\varepsilon_{\!}}(t)}\!\!=\!\{X\!\!\in\!\! B(X_0(t),\!C_{\!1}\varepsilon),\,
(1\!+|q^{\!*}|) | \widetilde{W}|\!+(1\!+t\!+|q^{\!*}|) | \overline{W}|\!
\leq \! 2C_{\!1}\varepsilon
 \big(\!\frac{1\,+\,q_-^*}{1+t+|q^{\!*}|\!}\big)^{\!\!\gamma^\prime\!} \!\}$,
\beq\label{eq:HZs}
H_Z( X,W)=-\frac{4 h_{1Z}^{\Lbs\Lbs}\!(X)
  +4h_{1Z}^{\Lbs\Us}\!\!(X)\, \widetilde{W}_{\Us}\!
  +h_{1Z}^{\Us\Vs}\!\!(X)\,\, \widetilde{W}_{\Us} \widetilde{W}_{\Vs}\!\!\! }
  {-2+h^{0\Lbs}_{1}\!(X)+g^{\,0 \Us}(X) \,\widetilde{W}_{\Us}}.
\eq
\begin{lemma} We have
\beq\label{eq:Htdiff}
\big|H_{\partial_t\!}(X_2,W_2)-H_{\partial_t\!}(X_{\!1},W_{\!1})\big|
\leq \frac{C_0\varepsilon\big(|X_2\!-\!X_1|+(1\!+|q^*|)|W_2\!-W_{\!1}|\big)\!\!}
{(1\!+t\!+| q^*|)^{2-\varepsilon}(1\!+|q^*|)^{1+\varepsilon}
(1\!+q_+^*)^{\gamma\prime}\!\!}
\,.
\eq
\end{lemma}
\begin{proof} Substituting
$h_{1\partial_t\!}(V\!,\!W)\!=\!\partial_t h_1^{\alpha\beta} V_{\!\alpha} W_{\!\beta}$
into \eqref{eq:HZs} using
\eqref{eq:quadraticinnullframe} we get
\begin{align*}
|H_{\partial_t}|\!&\les |\partial h_{1{\Ls}T}|\!
+|\partial h_{1TU}|\,( |W_{{\Ls}}|\!+|\sls{W\!}|^2)\!
+|\partial h_1| |W_{{\Ls}}|^2\!\!
\les \frac{\varepsilon(1\!+q_+^*)^{-\gamma\prime}}
{\!(1\!+t)^{2-\varepsilon}
 (1\!+|q^*|)^{\varepsilon}\!},\\
\Big| \frac{\partial H_{\partial_t}\!\!}{\partial X\!\!}\,\Big|
\!&\les |\partial^2 h_{1{\Ls}T}|
+|\partial^2 h_{1TU}|\, |W_{{\Ls}}|+|\partial^2 h_1|\, |W_{{\Ls}}|^2
\\
&+ | H_{\partial_t}| \big(|\partial h_{1TU}|\!
+|\partial h_1| |W_{{\Ls}}|\!+\!\frac{M}{\!(1\!+t)^2\!\!}\,\big)\!
\les \frac{\varepsilon(1+q_+^*)^{-\gamma\prime}}
{(1\!+t)^{2-\varepsilon}(1\!+|q^*|)^{\!1+\varepsilon}\!},\!\!\!\!\!\\
\Big| \frac{\partial H_{\partial_t}\!\!}{\!\partial \sls{W\!}\!}\,\Big|\!
&\les |\partial h_{1{\Ls}T}|
+|\partial h_{1TU}|\, |\overline{W}|\!
+| H_{\partial_t}| \,|h_{1TU}|\!\les
\frac{\varepsilon(1+|\,q_+^*|)^{-\gamma\prime}}{(1\!+t)^{2-\varepsilon}
 (1\!+|q^*|)^{\varepsilon}\!},\\
 \Big| \frac{\partial H_{\partial_t}\!\!}{\partial W_{\!{\Lbs}}\!\!}\,\Big|\!
&\les |\partial h_{1{\Ls}T}|
+|\partial h_{1TU}|\, |W_{{\Ls}}|
+ |H_{\partial_t}|\,
\les \frac{\varepsilon(1+|\,q_+^*|)^{-\gamma\prime}}
{(1+t)^{2-\varepsilon}(1+|\,q^*|)^{\varepsilon}},\\
\Big| \frac{\partial H_{\partial_t}\!\!}{\partial W_{\!{\Ls}} \!\!}\,\Big
|\!&\les |\partial h_{1TU}|
+|\partial h_{1}|\, |W_{\!{\Ls}}|
+|H_{\partial_t}| \les
\frac{\varepsilon(1+|\,q_+^*|)^{-\gamma\prime}}{(1+t)
 (1+|\,q^*|)}.
\end{align*}
\end{proof}

\begin{lemma} We have
\begin{multline}
{\sum}_{\Omega}\big|H_{\Omega}(X_2,W_2)
-H_{\Omega}(X_1,W_1)\big|
\leq \frac{C_0\varepsilon(1+q_+^*)^{-\gamma\prime}}
{(1+t+|\,q^*|)^{2-\varepsilon}
 (1+|\,q^*|)^{\varepsilon}}|X_2-X_1|\\
 + \frac{C_0\varepsilon|W_2-W_{\!1}|}{(1\!+t\!+|q^*|)}
 \Big(\frac{1+q_-^*}{1\!+t\!+|q^*|}\Big)^{\!\gamma\prime}\!\!
 +\frac{C_0\varepsilon(1\!+q_+^*)^{-\gamma\prime}\!\!\!}{1\!+t\!+| q^*|}\,\,
 |\overline{W}_{\!2}\!-\overline{W}_{\!1}|.\label{eq:Hdifftwo}
\end{multline}
\end{lemma}
\begin{proof}
If
$\Omega=\Omega_{ij}=x^i\pa_j-x^j\pa_i$, then with $R=\omega$,
we have by \eqref{eq:commutarorrotationquadraticform}
\begin{multline}
\!\!\!h_{1\Omega}(V,W)\!=(\Omega h_1^{\alpha\beta})V_\alpha W_\beta
+h_1^{\alpha \Omega/r} \big(V_\alpha W_R+V_R W_{\!\alpha}\big)
+ h_1^{\alpha R} \big( V_{\alpha}  \overline{W}_{\!\!\Omega/r}
+ \overline{V}_{\!\!\Omega/r} W_{\!\alpha}\big)\\
+h_1^{\alpha C}
\big( V_\alpha (C^{i\,} \overline{W}_{\!\!j}-C^{j\,}\overline{W}_{\!i} )
+ (C^{i\,} \overline{V}_{\!\!j}-C^{j\,}\overline{V}_{\!i} )W_{\!\alpha}\big).\label{eq:h1Omega2}
\end{multline}
Hence
\begin{multline*}
|h_{1\Omega}(V,W)|\les |\Omega h_1|\, |\overline{V}|\, |\overline{W}|
+(|\Omega h_{1TU}|+|h_{1TU}|)\,(|\overline{V}|\, |W|
+|V|\, |\overline{W}|)\\ +(|\Omega h_{1LL}|+|h_{1LT}|) \, |V|\, |W|,
\end{multline*}
and since $|W|\leq 1$ and $H_\Omega(X,W)=h_{1\Omega}(W,W)/g^{0\Us}W_{\Us}$
it follows that
\beqs
| H_\Omega|\les {\sum}_{k\leq 1}\,|\Omega^k h_{1{\Ls}T}|
+|\Omega^k h_{1TU}|\,|\overline{W}|
+|\Omega^k h_1|\, |\overline{W}|^2
\les \frac{\varepsilon}{1\!+t}
 \Big(\frac{1\!+q_-^*}{1+t}\Big)^{\!\gamma\prime}\!\!\!.
\eqs
By \eqref{eq:h1Omega2} and \eqref{eq:starnullframederivative} we have with
$\rho^{\,\prime} =dr^*\!/dr$
\begin{align*}
h_{1\Omega}(W^*\!,W^*)=\,&\,\Omega h_{1{\Ls}{\Ls}} +2h_{1{\Ls}\Omega/r}\rho^{\,\prime},\\
h_{1\Omega}(W^*\!,\widetilde{W})=\,&\,\Omega h_{1{\Ls}}^{\,\,\,\beta} \widetilde{W}_\beta
+ h_{1{\Ls}\Omega/r}\big(2\widetilde{W}_R-\rho^{\,\prime}\, \widetilde{W}_{{\Lbs}}\big)
-h_{1{\Lbs}\Omega/r}\rho^{\,\prime}\, \widetilde{W}_{{\Ls}} \\
&-h_{1C\Omega/r}\rho^{\,\prime}\,\overline{W}_{\!C}+h_{1{\Ls}R}
\overline{W}_{\!\Omega/r}+h_{1{\Ls} C}
\big(C^{i\, }\overline{W}_{\!\!j}-C^{j\,}\overline{W}_{\!i }\big),\\
h_{1\Omega}(\widetilde{W}\!,\widetilde{W})
=\,&\,(\Omega h_1^{\alpha\beta})\widetilde{W}_\alpha \widetilde{W}_\beta
+2 h_1^{\alpha \Omega/r} \,\widetilde{W}_\alpha \widetilde{W}_R\\
&+ 2 h_1^{\alpha R} \,\widetilde{W}_\alpha  \overline{W}_{\!\Omega/r}
+2 h_1^{\alpha C} \widetilde{W}_\alpha
 (C^{i\,} \overline{W}_{\!\!j}-C^{j\,}\overline{W}_{\!i} ),\\
g^{0\Us}W_{\Us}
=\,&-g^{0}_{\,\,\, \Ls}-\tfrac{1}{2}g^{0}_{\,\,\, \Ls}\widetilde{W}_{\Lbs}
-\tfrac{1}{2}g^{0}_{\,\,\, \Lbs}\overline{W}_{\!\Ls}
+\delta^{AB}g^{0}_{\,\,\, A}\overline{W}_{\!B}.
\end{align*}
Since $h_{1\Omega}(W,W)
=h_{1\Omega}(W^*\!,W^*)+2h_{1\Omega}(W,\widetilde{W})
+h_{1\Omega}(\widetilde{W}\!,\widetilde{W})$ it follows that
\begin{multline*}
\!\!\!\!\Big|\frac{\partial H_\Omega}{\partial X}\Big|
\les { \sum}_{k\leq 1}|\,\pa\Omega^k h_{1{\Ls}T}|
+|\,\pa\Omega^k h_{1TU}|\,|\overline{W}|
+|\pa\Omega^k h_1|\, |\overline{W}|^2\\
+|\,H_\Omega| \Big(|\,\partial h_{1TU}|+|\,\partial h_1|\,|W_{{\Ls}}|
+\frac{M}{\!(1\!+t)^2\!\!} \,\Big)
\les
\frac{\varepsilon(1\!+q_+^*)^{-\gamma^\prime}}{(1\!+r)^{2-\varepsilon}
 (1\!+|q^*|)^{\varepsilon}\!\!}\,,
\end{multline*}
when $r\geq t/2$.
Moreover, since $|W|\leq 1$
\begin{align*}
\Big|\frac{\partial H_\Omega}{\partial \sls{W\!}}\Big|
&\les |\,\Omega h_{1TU}|+|h_{1TU}|+|h_1|\, |W_{{\Ls}}|
+|H_\Omega| \,|\,h_{1TU}|
\les
\frac{\varepsilon(1\!+q_+^*)^{-\gamma^\prime}\!\!\!\!\!\!}{1+t}\,\,\,,\\
\Big|\frac{\partial H_\Omega}{\partial W_{{\Ls}}}\Big|
&\les
|\,\Omega h_{1TU}|\!+|\,h_{1TU}|\!+|\,\Omega h_{1}|\, |W_{{\Ls}}|\!
+| h_{1}|\,|\overline{W}|\!
+|H_\Omega| \les
\frac{\varepsilon(1\!+q_+^*)^{-\gamma^\prime}\!\!\!\!\!\!}{1\!+t}\, ,\,\,\\
\Big| \frac{\partial H_{\Omega}}{\partial W_{{\Lbs}}}\Big|
&\les{ \sum}_{k\leq 1}|\Omega^k h_{1{\Ls}\Ts}|\!
 +| \Omega^k h_{1TU}| \,|\overline{W}|+ | H_{\Omega} |
\les \frac{\varepsilon}{1\!+t}
\Big(\frac{1\!+q_-^*}{1\!+t}\Big)^{\gamma^\prime}\!\!\!.
\end{align*}
\vspace{-0.25in}
\end{proof}

\subsection{Proof of Proposition \ref{prop:chardiff}}
By Proposition \ref{prop:uniformeikonalbounds} these estimates are true
when $t\!=\!T$ for any constant
$C_2\!\geq \!2C_1$ and $C_3\!\geq \!2C_1$. We we claim that if $\varepsilon\!>\!0$ is sufficiently
small they are true for $q_-^*\!\leq \!t\!\leq\! T$, with
$C_2\!=\!8C_1$ and
$C_3\!=\!8C_2C_{\gamma^{\prime\prime}}$, for some universal constant $C_{\gamma^{\prime\prime}}$.
Since as we shall see below
we have differential equations for these quantities they are
continuous functions so we can prove this
by assuming that these estimates are true for $t\geq t_1$ and
show that they imply better estimates
as long as $t_1\geq  q_-^*$.
If we integrate \eqref{eq:eikonaldvi} with $Z\!=\!\Omega$ we get
\beqs
 |r_2^*\sls{W\!}_{2}\!-r_1^*\sls{W\!}_1|
\leq |r_2^* \sls{W\!}_2\!-r_1^* \sls{W\!}_1|_{\,t=T}
+\!\int_t^T\!\!{\sum}_{\Omega} | H_\Omega(X_2,W_2)
-H_\Omega(X_1,W_1)| \, dt.
\eqs
If we use \eqref{eq:uniformeikonalboundtwo} at $t\!=\!T$,
\eqref{eq:Hdifftwo}, \eqref{eq:WL} and the assumed bounds
\eqref{eq:chardiffone}-\eqref{eq:chardiffthree}:
\begin{multline*}
(1\!\!+\!t\!+|q^*|)|\overline{W}_{\!2}\!-\!\overline{W}_{\!1}|\\
\leq 2C_1\varepsilon
\Big( \frac{1\!+q_-^*}{1\!\!+\!T\!\!+\!|q^*|}\Big)^{\!\gamma^\prime}
\!\!\!\!\!+ C (C_2\!\!+C_3) \varepsilon^2\!\Big( \frac{1+q_-^*}{1\!\!+\!T\!\!+\!|q^*|}\Big)^{\!\gamma^\prime\!-\gamma^{\prime\prime}\!}\!\!\!
 \Big( \frac{1+q_-^*}{1\!\!+\!t\!+\!|q^*|}\Big)^{\!\gamma^{\prime\prime}}\\
 <4C_1\varepsilon
 \Big( \frac{1+q_-^*}{1\!\!+\!T\!\!+\!|q^*|}\Big)^{\!\!\gamma^\prime\!-\gamma^{\prime\prime\!}\!\!\!}
 \Big( \frac{1+q_-^*}{1\!\!+\!t\!+\!|q^*|}\Big)^{\!\!\gamma^{\prime\prime}}\!\!\! ,
\end{multline*}
if $\varepsilon$ is sufficiently small
which proves \eqref{eq:chardifftwo}.
If we integrate \eqref{eq:eikonaldvi} with $Z\!=\!\partial_t$:
\beqs
|W_{\!2}-{W}_{\!1}|
\leq |\overline{W}_{\!2}\!-\!\overline{W}_{\!1}|+|W_{\!2}-W_{\!1}|_{\,t=T}
+\!\!\int_t^T\!\!\!|
H_{\partial_t}(X_2,W_2)-H_{\pa_t}(X_1,W_1)| \, dt.
\eqs
If we use \eqref{eq:uniformeikonalboundone} at $t\!=\!T$,
\eqref{eq:Htdiff} with the bounds
\eqref{eq:chardiffone}-\eqref{eq:chardiffthree} we get as above
\beqs\label{eq:aprioribound}
(1\!+|q^*|)|{W}_2\!-\!{W}_1|
\!< 8C_1\varepsilon
 \Big( \frac{1+q_-^*}{1\!\!+\!T\!+\!|q^*|}\Big)^{\!\gamma^\prime\!-\gamma^{\prime\prime\!}\!\!\!}
 \Big( \frac{1+q_-^*}{1\!+t\!+\!|q^*|}\Big)^{\gamma^{\prime\prime}}\!\!\!\!,
\eqs
if $\varepsilon\!>\!0$ is small which proves \eqref{eq:chardiffone}.
Integrating \eqref{eq:charcurvei} we get
\beqs
|\,X_2-X_1|
\leq |\,X_2-X_1|_{\,t=T}
+\int_t^T | F(X_2,W_2)-F(X_1,W_1)| \, dt .
\eqs
\eqref{eq:chardiffthree} follows
if we use \eqref{eq:tildegammaest} at $t\!=\!T$ and \eqref{eq:Fdifftwo}
with
\eqref{eq:chardiffone}-\eqref{eq:chardiffthree}
\begin{multline*}
|X_2-X_1|\\
\leq 2C_1\varepsilon
\Big( \frac{1\!+q_-^*}{1\!+\!T\!\!+\!|q^*|}\Big)^{\!\gamma^\prime}
\!\!\!\!\!+\big( C (C_2\!+C_3) \varepsilon^2\!\!+C_2\varepsilon C_{\gamma^{\prime\prime}\!}\big)\!\Big( \frac{1+q_-^*}{1\!+T\!\!+\!|q^*|}\Big)^{\!\gamma^\prime\!-\gamma^{\prime\prime}\!}\!\!\!
 \Big( \frac{1+q_-^*}{1\!+t\!+\!|q^*|}\Big)^{\!\gamma^{\prime\prime}}\\
 < 2C_2\varepsilon C_{\gamma^{\prime\prime}}
 \Big(\frac{1+q_-^*}{1+T+|q^*|}\Big)^{\gamma^\prime-\gamma^{\prime\prime}}
 \Big( \frac{1+q_-^*}{1+t+|q^*|}\Big)^{\gamma^{\prime\prime}}\!\!\!.
\end{multline*}

\vspace{-0.2in}

\section{The mass loss law}
We have seen that the asymptotics of $H_{\!LL\!}$ close to the light cone is $2M\!$.
On the other hand it is also given by the asymptotics for the wave equation with the
source determined by $\!n$.\! Therefore there has to a be a relation between $\!M\!$ and $\!n$:
\begin{prop}\label{prop:totalmassloss} We have
\beq\label{eq:totalmassloss}
\frac{1}{2}\int_{-\infty}^{+\infty} \int_{\bold{S}^2}n(q^*,\omega)\frac{ dS(\omega)\!\!}{4\pi}\,
dq^*\!=M.
\eq
\end{prop}

\begin{proof} Since $-L^\mu(\omega)L_\mu(\sigma)=1-\langle \omega,\sigma\rangle$
we have \beqs
 k^2_{LL}(t,r^*\omega)
=\frac{1}{r^*} \int_{q^*}^\infty \int_{\bold{S}^2}{\frac{
\big(1-\langle \omega,\sigma\rangle\big)^2
  n({\rho^*},{\sigma})}{({\rho^*}-q^*)/r^*+1-\langle\, \omega,{\sigma}\rangle}\,
\frac{dS({\sigma})}{4\pi}}\,\chi\big(\tfrac{\langle{\rho^*}\rangle}{t+r^*}\big)
d {\rho^*},
\eqs
where $q^*=r^*-t<<0$.
Hence
\begin{multline*}
\!\!\!\int_{\bold{S}^2}\!r^* k^2_{LL}(t,r^*\omega) \frac{dS(\omega)}{4\pi}
=\!\int_{q^*}^\infty \!\!\int_{\bold{S}^2}\int_0^2
{\frac{\omega_1^2  n({\rho^*},{\sigma})}{({\rho^*}\!-q^*)/r^*\!+\omega_1}\,\frac{d\omega_1}{2}\,
\frac{dS({\sigma})}{4\pi}}\,\chi\big(\tfrac{\langle\,{\rho^*}\,\rangle}{t+r^*}\big)
d {\rho^*}\\
=\int_{q^*}^\infty \!\!\!\int_{\bold{S}^2}\!\!\big(2-2a+a^2\ln{\Big|\frac{2+a}{a}\Big|}\big) n({\rho^*},{\sigma}) \,
\frac{dS({\sigma})}{8\pi}\,\chi\big(\tfrac{\langle\,{\rho^*}\,\rangle}{t+r^*}\big) d {\rho^*}\!\!,
\end{multline*}
where  $a\!=\!({\rho^*}\!-q^*)/r^*$.
Since $|n(\rho^*,\sigma)|\!\les\! \varepsilon^2(1\!+|\rho^*|)^{-2} (1\!+\rho_+^*)^{-\gamma^\prime}$
it follows that
\beqs
\Big|\int_{\bold{S}^2}r^* k^2_{LL}(r^*-q^*,r^*\omega) \frac{dS(\omega)}{4\pi}
-\int_{q^*}^\infty \int_{\bold{S}^2} n({\rho^*},{\sigma}) \,
\frac{dS({\sigma})}{4\pi}\, d\rho^*\Big|\les \frac{\varepsilon^2 |q^*|}{r^*}.
\eqs
Let $\Phi_{LL}^{2*}(q^*\!\!,\omega,r^*)\!=\!r^* k^2_{LL}(r^*\!-q^*\!\!,r^*\omega)$.
By the previous arguments
$\Phi_{LL}^{2\infty}(q^*\!\!,\omega)\!=\lim_{r^*\to\infty}\Phi_{LL}^{2*}(q^*,\omega,r^*)$
exists and satisfies
$\big|\Phi_{LL}^{2\infty}(q^*,\omega)-2M\big|\les \varepsilon (1+q_-^*)^{-\gamma^\prime}\!$. Hence if we pass to the limit $r^*\!\to\!\infty$ in the above
\beqs
\Big|\,2M
-\int_{q^*}^\infty \int_{\bold{S}^2} n({\rho^*},{\sigma}) \,
\frac{dS({\sigma})}{4\pi}\, d\rho^*\Big|\les\frac{\varepsilon }{(1+q_-^*)^{\gamma^\prime}} .
\eqs
Taking $q^*\to-\infty$ proves the theorem.
\end{proof}
 The proposition in particular implies that, if $n\!=\!0$ then $M\!\!=\!0$, and then by the positive mass theorem the space time is Minkowski space.
The proposition can be interpreted as that the total radiated energy equals the initial mass, if there is no black hole.
 By \cite{BBM,C2} the radiated energy density is equal to the limit
along outgoing null hypersurfaces of the square
of the trace less part of the conjugate null second fundamental form of surfaces in the null
 hypersurface.

We define the radius of a surface $S$ to be  $r(S)\!=\!\sqrt{\!\text{Area}(S)\!/4\pi}$.
Let $\Lh$ and $\Lbh$ be the outgoing respectively incoming
null normals to $S$ satisfying
$g(\Lh,\Lbh)\!\!=\!-2$.\!
$\Lh$ and $\Lbh$ are unique up to the transformation
 $\Lh\!\!\to \!a \Lh$ and
 $\Lbh\!\to  a^{\!-1\!}\Lbh$.
The null second fundamental form and the conjugate null second fundamental form
are defined to be the tensors
$\chi(X,Y)\!\!=\!g(\nabla_{\!\!X}\Lh,Y)$
respectively $\underline{\chi}(X,Y)\!\!=\!g(\nabla_{\!\!X} \Lbh,Y)$
for any vectors $X,Y$ tangent to $S$ at a point,
where $\nabla_{\!\!X}$ is covariant differentiation.
Under the transformation above $\chi\!\!\to\! a\chi$ and
 $\underline{\chi}\!\!\to\! a^{-1} \underline{\chi}$ so
the  Hawking mass of $S$,
$M_{\mathcal{H}}(S)\!=r(S)\big(1\!+\int_S \tr \chi \,\tr \underline{\chi} \,dS\!/16\pi
\big)\!/2$
is invariant. If $\tr \chi \,\tr \underline{\chi}\!<\!0$ we
can fix $\Lh$ and $\Lbh$ by
$\tr\chi\!+\tr \underline{\chi}\!=\!0$. Let $\hat{\chi}$ and $\hat{\underline{\chi}}$
be the traceless parts. The incoming respectively outgoing energy flux through $\!S_{\!}$ are
$E(S)\!=\!\int_{S}\! \hat{\chi}^2 dS\!/32\pi$ and
$\underline{E}(S)\!=\!\int_{S}\! \hat{\underline{\chi}}^2 \!dS_{\!}/32\pi\!$.

Let $C_{\!u}\!$ and $\underline{C\!\!}_{\,v}\!$ be the characteristic surfaces
of constant $u\!\sim t\!-r^*\!$ respectively $v\!\sim t+r^*\!$ as in section \ref{section:characteristicsurfaces}.
Let $S_{u,v}\!\!= \!C_{\!u}\cap\underline{C\!}_{v}$. For fixed
$v$, $E(\underline{C\!}_{v})\!\!=\!\!\int\!\! \underline{E\!}\,(S_{u,v})du$ is the norm of characteristic initial data on $\underline{C\!\!}_{\,v}\!\!$ \cite{C3}.
The energy at null infinity is $\!\int\! \!E(u)du$ where $E(u)$ is the limit of
$\underline{E\!}\,(S_{u,v})$ as $v\!\!\to \!\infty$.
The Bondi mass $M_{{}_{\!}}(u_{\!})$
is the limit of $\!M_{\!\mathcal{H}}(S_{u_{\!},v_{\!}})\!$ as $v\!\!\to \!_{\!}\infty$.
By the Bondi mass loss law $d M_{{}_{\!}}(_{{}_{\!}}u)\!/du\!=\!-E(u)$.
For asymptotically flat data $M_{{}_{\!}}(_{{}_{\!}}u_{{}_{\!}})\!\!\to\!\! M\!$, the ADM mass, as $u\!\to \!-\infty$,  and in the
absence of a black hole, $M_{{}_{\!}}(_{{}_{\!}}u_{{}_{\!}})\!\!\to \!0$ as $u\!\to\! +\infty$.
If $E(q^{\!*})\!\!=\!\!\int_{\bold{S}^2}\!n(q^{\!*}\!\!,\omega){ dS(\omega)\!}/{8\pi}$ then
\eqref{eq:totalmassloss} says that $\int_{-\infty}^{+\infty}\! \!E(u) du\!=\!M\!$.
This will be explored in forthcoming papers.

\section*{Acknowledgments}
I would like to thank Igor Rodnianski for many important discussions and
initial collaboration.  The results here came up in connection with curvature decay,
  but over time diverged from that.
 I would like to thank Piotr Chrusciel, Sergiu Klainerman, Volker Schlue and Martin Taylor
  for useful discussions.
 I would like to thank
  the Mittag-Leffler Institute and MSRI and the organizers
  of the programs in Mathematical Relativity.
   I also thank NSF for support.

\end{document}